\documentclass[reqno,12pt]{amsart}

\usepackage{amscd,amssymb,comment,epic,eepic,euscript}
\usepackage{graphicx}
\usepackage{enumerate}
\usepackage[initials]{amsrefs}
\usepackage{array}

\theoremstyle{plain}
\newtheorem{thm}{Theorem}[section]
\newtheorem{cor}[thm]{Corollary} 
\newtheorem{lemma}[thm]{Lemma} 
\newtheorem{prop}[thm]{Proposition}
\newtheorem{conj}[thm]{Conjecture}

\theoremstyle{remark}

\newtheorem{remark}[thm]{Remark}

\theoremstyle{definition}
\newtheorem{defi}[thm]{Definition}
\newtheorem{example}[thm]{Example}

\newtheorem{notation}[thm]{Notation}

\newtheorem{gp}{Group}[subsubsection]
\newtheorem{matcases}{Cases}[subsection]

\newcount\theTime
\newcount\theHour
\newcount\theMinute
\newcount\theMinuteTens
\newcount\theScratch
\theTime=\number\time
\theHour=\theTime
\divide\theHour by 60
\theScratch=\theHour
\multiply\theScratch by 60
\theMinute=\theTime
\advance\theMinute by -\theScratch
\theMinuteTens=\theMinute
\divide\theMinuteTens by 10
\theScratch=\theMinuteTens
\multiply\theScratch by 10
\advance\theMinute by -\theScratch

\def\today{{\number\day\space
 \ifcase\month\or
  January\or February\or March\or April\or May\or June\or
  July\or August\or September\or October\or November\or December\fi
 \space\number\year}}

\newcommand\Afr{{\mathfrak A}}
\newcommand\at{{\tilde a}}

\newcommand\bt{{\tilde b}}
\newcommand\Dih{\operatorname{Dih}}
\newcommand\Fb{{\mathbb F}}
\newcommand\id{{\operatorname{id}}}
\newcommand\Ints{{\mathbb Z}}
\newcommand\lspan{\mathrm{span}\,}
\newcommand\Nats{{\mathbb N}}
\newcommand\rank{\operatorname{rank}}

\newcommand\pit{{\tilde\pi}}
\newcommand\rt{{\tilde r}}
\newcommand\simc{\overset{c}{\sim}}
\newcommand\simpi{\overset{\pi}{\sim}}
\newcommand\simr{\overset{r}{\sim}}
\newcommand\st{{\tilde s}}
\newcommand\tor{{\operatorname{tor}}}
\newcommand\ttilde{{\tilde t}}
\newcommand\Tt{{\widetilde T}}
\newcommand\ULIE{\operatorname{ULIE}}
\newcommand\bZ{\mathbb Z}

\begin{document}

\title[Direct finiteness conjecture]{Finitely presented groups related to Kaplansky's Direct Finiteness Conjecture}

\author[Dykema]{Ken Dykema$^{*}$}
\address{Dykema, Department of Mathematics, Texas A\&M University,
College Station, TX 77843-3368, USA}
\email{kdykema@math.tamu.edu}
\thanks{\footnotesize ${}^{*}$Research supported in part by NSF grant DMS--0901220.}

\author[Heister]{Timo Heister$^{\dagger}$}
\address{Heister, Department of Mathematics, Texas A\&M University,
College Station, TX 77843-3368, USA}
\email{heister@math.tamu.edu}
\thanks{\footnotesize ${}^{\dagger}$This publication is based, in part, on the work supported by Award No.
KUS-C1-016-04, made by King Abdullah University of Science and Technology (KAUST)}

\author[Juschenko]{Kate Juschenko}
\address{Juschenko, Vanderbilt University, Department of Mathematics, 1326 Stevenson Center,
Nashville, Tennessee 37240}
\email{kate.juschenko@gmail.com}

\subjclass[2000]{20C07, (20E99)}

\keywords{Kaplansky's Direct Finiteness Conjecture, stable finiteness, Invertibles Conjecture,
sofic groups}

\date{June 25, 2012}

\begin{abstract}
We consider a family of finitely presented groups, called Universal Left Invertible Element (or ULIE) groups,
that are universal for existence of one--sided invertible elements in a group ring $K[G]$, where
$K$ is a field or a division ring.
We show that for testing Kaplansky's Direct Finiteness Conjecture, it suffices to test it on ULIE groups,
and we show that there is an infinite family of non-amenable ULIE groups.
We consider the Invertibles Conjecture 
and we show that it is equivalent to a question about ULIE groups.
We also show that for any group $G$, direct finiteness of $K[G\times H]$ for all finite groups $H$ implies stable finiteness
of $K[G]$.
Thus, truth of the Direct Finiteness Conjecture implies stable finiteness.
By calculating all the ULIE groups over the field $K=\Fb_2$ of two elements,
for ranks $(3,n)$, $n\le 11$ and $(5,5)$, we show that the Direct Finiteness Conjecture
and the Invertibles Conjecture (which implies the Zero Divisors Conjecture) hold for these ranks over $\Fb_2$.
\end{abstract}

\maketitle

\section{Introduction}

In the middle of the last century, Kaplansky showed (see~\cite{K}, p.\ 122) that for every field $K$ of characteristic
$0$ and every discrete group $\Gamma$, the group ring $K[\Gamma]$ (which is, actually, the $K$--algebra with
basis $G$ and multiplication determined by the group product on basis elements and the distributive law)
is
{\em directly finite},
namely, that for every $a,b\in K[\Gamma]$ the equation $ab=1$ implies
$ba=1$.
This is clearly equivalent to saying that all one--sided invertible elements in $K[G]$ are invertible.
However, the situation for fields of positive characteristic is
unresolved; the following conjecture of Kaplansky is still open:
\begin{conj}\label{conj:Kap}
For every discrete group $\Gamma$ and every field $K$,
the group ring $K[\Gamma]$ is directly finite.
\end{conj}

We will call this Kaplansky's Direct Finiteness Conjecture, or simply the Direct Finiteness Conjecture (DFC).

Ara, O'Meara and Perera proved~\cite{AOP02} that the DFC holds (and also when $K$ is a division ring)
for residually amenable groups.
Elek and Szab\'o~\cite{ES04} generalized this result to a large class of groups,
namely the sofic groups, (also with $K$ a division ring, and they proved also stable finiteness --- see below).
Since currently there are no known examples of non-sofic groups, the Kaplansky DFC is even more intriguing.
Moreover it is well known that, in the case of finite fields,
Gottschalk's conjecture~\cite{Go73} implies Kaplansky's DFC (see~\cite{ES04} for a proof).

The notion of a sofic group was introduced by Gromov in~\cite{G99}
as a group with Cayley graph that satisfies a certain approximation property.
He showed that Gottschalk's conjecture is satisfied for sofic groups.
Many interesting properties are known about sofic groups.
The class of sofic groups is known to be closed under taking 
direct products, subgroups, inverse limits, direct limits, free products, and extensions by amenable groups (by~\cite{ES06})
and under taking free products with amalgamation over amenable groups (see~\cite{CD}, \cite{ES11} and \cite{P11}).

In this paper, we describe finitely presented groups that are universal for existence of one--sided
invertible elements in a group algebra.
To test Kaplansky's DFC, it will be enough to test it on these universal groups.
In fact, this idea, at least in the case of the field of two elements,
has been around in discussions among several mathematicians for some time.
See for example the MathOverflow posting~\cite{T2} of Andreas Thom, or Roman Mikhailov's preprint~\cite{M}.
Who was the first to describe these groups is unclear to the authors, and we believe
that these groups may have been rediscovered by several persons at different times.
After we posted an earlier version of this paper
(which lacked sections~\ref{sec:alg} and~\ref{sec:calcs} and is still available on the arXiv), a paper of
Pascal Schweitzer~\cite{Schw} about similar calculations for the Zero Divisors Conjecture appeared.
These efforts were independent of each other.

To illustrate, let us work over the field $K=\Fb_2$ of two elements.
If $a,b\in\Fb_2[G]$ and $ab=1$, then we may write 
\begin{equation}\label{eq:F2ab}
a=a_0+a_1+\cdots+a_{m-1}\text{ and }b=b_0+b_1+\cdots+b_{n-1}
\end{equation}
for group elements $a_0,\ldots,a_{m-1}$ that are distinct and group elements $b_0,\ldots,b_{n-1}$ that
are distinct.
The identity $ab=1$ implies that $a_ib_j=1$ for some $i$ and $j$;
after renumbering, we may assume $i=j=0$, and then, replacing $a$ by $a_0^{-1}a$ and $b$ by $bb_0^{-1}$,
we may assume $a_0=b_0=1$.
Now distributing the product $ab$ we get
\[
\sum_{i=0}^{m-1}\sum_{j=0}^{n-1}a_ib_j=1
\]
and, thus, there is a partition $\pi$ of $\{0,1,\ldots,m-1\}\times\{0,1,\ldots,n-1\}$ with one singleton set
$\{(0,0)\}$ and all other sets containing two elements, such that if $(i,j)\simpi(k,\ell)$ (i.e., if $(i,j)$ and $(k,\ell)$
belong to the same set of $\pi$), then $a_ib_j=a_kb_\ell$.
Consider the finitely presented group
\begin{equation}
\Gamma_\pi=\langle a_0,a_1,\ldots,a_{m-1},b_0,b_1,\ldots,b_{n-1}
\mid a_0=b_0=1,\,(a_ib_j=a_kb_\ell)_{\{(i,j),(k,\ell)\}\in\pi}\rangle,  
\end{equation}
where the relations are indexed over all pairs $\{(i,j),(k,\ell)\}$ of the partition $\pi$.
Then there is a group homomorphism $\Gamma_\pi\to G$ sending the given generators of $\Gamma_\pi$ to their
namesakes.
Furthermore, the corresponding elements $a$ and $b$ in $\Fb_2[\Gamma_\pi]$, defined by equation~\eqref{eq:F2ab},
satisfy also $ab=1$.
If $ba=1$ holds in $\Fb_2[\Gamma_\pi]$, then it holds in $\Fb_2[G]$ as well.
Therefore, to test Kaplansky's Direct Finiteness Conjecture
over $\Fb_2$, it will suffice to test it on the groups $\Gamma_\pi$.

We call these groups (and their analogues for more general $K$) ULIE groups, short for Universal Left Invertible Element groups.
In this paper, we will show that studying the ULIE will be enough to answer Kaplansky's Direct Finiteness
Conjecture, and we will prove a few facts about them, including that there is an infinite family of
non-amenable ULIE groups.
With the aid of computers,
we have found
all ULIE groups (for the field $\Fb_2$) up to sizes $3\times 11$ and $5\times 5$ and used soficity results
to obtain partial confirmation of Kaplansky's DFC over $\Fb_2$.

Throughout the paper,
if $K$ is said to be a division ring, then it may also be a field, and will be assumed to be nonzero.
We let $1$ denote the identity element of a group $G$, or
the multiplicative identity of a division ring $K$ or of a group ring $K[G]$,
depending on the context.

We would like to mention two other well known conjectures about group rings.
Let us call the following the {\em Invertibles Conjecture} (IC).
See Conjecture~2 of~\cite{V02} for a statement when $K$ is the complex numbers.
\begin{conj}\label{conj:IT}
If $K$ is a division ring and $G$ is a group and if $K[G]$ contains a one--sided
invertible element that is not of the form $kg$ for $g\in G$ and $k\in K$, then $G$ has torsion.
\end{conj}

As is well known,
it implies the famous {\em Zero Divisors Conjecture} (ZDC):
\begin{conj}\label{conj:ZD}
If $K$ is a division ring and $G$ a group and if 
$K[G]$ contains
zero divisors, then $G$ has torsion.
\end{conj}
Though the proof is well known, it seems appropriate to describe it here.
We are indepted to a posting~\cite{T} by Andreas Thom on MathOverflow for the following argument.
\begin{proof}[Proof of (IC)$\implies$(ZDC)]
If ZDC fails, then there is a torsion free group $G$ and
there are nonzero $a,b\in K[G]$ so that $ab=0$.
Since $G$ is torsion free, a result of Connell~\cite{C63} (or see Thm.\ 2.10 of~\cite{Pa77})
implies that $K[G]$ is prime.
This entails that for nontrivial ideals $A$ and $B$, we cannot have $BA=0$.
By primality, there must be $c\in K[G]$
so that $bca\ne0$, and then $(bca)^2=0$;
we have $(1-bca)(1+bca)=1$ and, since $(bca)^2=0$, we have $bca\notin K1$ and $K[G]$ has one--sided invertible elements.
So IC fails.
\end{proof}

In Section~\ref{sec:Obs}, we introduce notation and make some preliminary observations about
the three conjectures mentioned above, including the well known fact that
the rank $2$ cases of all three hold.

Regarding Kaplansky's Direct Finiteness Conjecture, one can also ask for more:
one can ask for all matrix algebras $M_n(K[G])$ to be directly finite.
If this holds, the group ring $K[G]$ is said to be {\em stably finite}.
In Section~\ref{sec:StabFin}, we show that direct finiteness
of $K[G\times H]$ for all finite groups $H$ implies stable finiteness of $K[G]$.

In Section~\ref{sec:K}, we introduce ULIE groups and show that for solving Kaplansky's Direct Finiteness
Conjecture (or various subcases thereof), it is enough to consider ULIE groups
and we state that our calculations (described in Section~\ref{sec:calcs}) imply that the DFC holds for
ranks $(3,n)$ with $n\le 11$ and $(5,5)$.
In Section~\ref{sec:InfFam}, we exhibit an infinite family of non-amenable ULIE groups.

In Section~\ref{sec:InvTor}, we show that the Invertibles Conjecture can be reformulated
in terms of certain quotients of ULIE groups
and we state that our calculations (described in Section~\ref{sec:calcs}) imply that the Invertibles Conjecture holds for
ranks $(3,n)$ with $n\le 11$ and $(5,5)$.

In Section~\ref{sec:alg} we describe the algorithm we employed to list all the ULIE groups over the field
$\Fb_2$ of two elements
for given ranks
and in Section~\ref{sec:calcs} we report on the results of these calculations.

\medskip
\noindent
{\bf Acknowledgment.}
The authors thank Beno\^it Collins, Denis Osin, Andreas Thom and Alain Valette for helpful discussions.

\section{Notation and Preliminary observations}
\label{sec:Obs}

Let $K$ be a division ring (or field).
Consider a group $G$ and elements $a$ and $b$ in the group ring $K[G]$, satisfying $ab=1$.
Just to fix notation: we say that $b$ is a right inverse of $a$ and that $a$ is right invertible, and that
$a$ is a left inverse of $b$ and that $b$ is left invertible.
We suppose that not both $a$ and $b$ are supported on single elements of $G$, and then neither of them may be,
and we are interested in the question of whether $ba=1$ must then hold.
We may write
\begin{equation}\label{eq:a=b=}
a=r_0a_0+\cdots+r_{m-1}a_{m-1},\qquad b=s_0b_0+\cdots+s_{n-1}b_{n-1}
\end{equation}
for integers $m,n\ge2$, for nonzero elements $r_0,\ldots,r_{m-1},s_0,\ldots,s_{n-1}$ of $K$ and for
distinct elements $a_0,\ldots,a_{m-1}$ of $G$ and distinct elements $b_0,\ldots,b_{n-1}$ of $G$.
We then say that the {\em rank} of $a$ is $m$ and of $b$ is $n$,
and that the {\em support} of $a$ is $\{a_0,a_1,\ldots,a_{m-1}\}$ and of $b$ is $\{b_0,\ldots,b_{n-1}\}$.
We must have $a_ib_j=1$ for at least one pair $(i,j)$, and by renumbering, we may assume $a_0b_0=1$.
Replacing $a$ by $a_0^{-1}a$ and $b$ by $bb_0^{-1}$, we may assume $a_0=b_0=1$.
Replacing $a$ by $r_0^{-1}a$ and $b$ by $br_0$, we may also assume $r_0=1$.

The rest of this section is devoted to making some observations that include the well known fact that
the rank~$2$ cases of the three conjectures described in the introduction are true (over any division ring).
It seems convenient to collect the proofs here, and the related results, (Propositions~\ref{prop:EHab=1}
and~\ref{prop:ZDCH}) may be useful in future.

If $H$ is a subgroup of $G$, then $K[H]$ is naturally contained as a subalgebra in $K[G]$.
Let $E=E^G_H:K[G]\to K[H]$ be the idempotent, surjective linear mapping defined by
\[
E(g)=\begin{cases}g,&g\in H \\ 0,&g\notin H.\end{cases}
\]
Of course, $E$ satisfies the conditional expectation property,
namely, that $E(abc)=aE(b)c$ if $b\in K[G]$ and $a,c\in K[H]$.

\begin{lemma}\label{lem:supports}
If $a\in K[G]$ is of rank $\ge2$ and has a right (or, respectively, left) inverse,
then it has a right (respectively, left) inverse whose support lies in the subgroup of $G$ generated by
the support of $a$.
\end{lemma}
\begin{proof}
Let $H$ be the subgroup generated by the support of $a$ and let $E=E^G_H$.
If $ab=1$ for $b\in K[G]$, then $1=E(ab)=aE(b)$.
Similarly, if $ba=1$, then $E(b)a=1$.
\end{proof}

It is now immediate that over a commutative field $K$, right invertible elements of rank $2$ must be invertible.
Moreover, as we see below, it is not hard to prove the same result also for division rings.
\begin{prop}\label{prop:rank2}
If $a\in K[G]$ is one--sided invertible and of rank $2$, then it is invertible.
Furthermore a rank $2$ element of $K[G]$ is invertible if and only if it is of the form
$a=sh(1-rg)$ for $s,r\in K\backslash\{0\}$ and
$g,h\in G$, where $g$ has finite order $n>1$ and $r^n\ne1$;
then we have
\begin{equation}\label{eq:a-1}
a^{-1}=(1-r^n)^{-1}(1+rg+r^2g^2+\cdots+r^{n-1}g^{n-1})h^{-1}s^{-1}.
\end{equation}
\end{prop}
\begin{proof}
If $a$ has rank $2$, then it can be written in the form $sh(1-rg)$
for $s,r\in K\backslash\{0\}$ and $h,g\in G$, $g\ne1$.
It will suffice to consider $a=1-rg$.
If $a$ has a right inverse, then by Lemma~\ref{lem:supports} it has a right inverse $c$ whose 
support belongs to the group $H$ generated by $g$.
We may, thus, write
\[
c=s_0+s_1g+s_2g^2+\cdots+s_{n-1}g^{n-1},
\]
where $n\ge2$ is such that $1,g,\ldots,g^{n-1}$ are distinct and $s_j\in K$. 
Now multiplying out $ac=1$ and solving, we must have $g^n=1$, $r^n\ne1$ and
$c=(1-r^n)^{-1}(1+rg+r^2g^2+\cdots+r^{n-1}g^{n-1})$.
But in this case, we have $ca=1$, so $a$ is invertible.
The general form~\eqref{eq:a-1} of $a^{-1}$ follows immediately.
\end{proof}

Thus, the rank $2$ cases of the Direct Finiteness Conjecture and the Invertibles Conjecture
are trivially true.
This conditional expectation trick also gives us the following:
\begin{prop}\label{prop:EHab=1}
Suppose $a,b\in K[G]$, each of rank $\ge2$, satisfy $ab=1$.
Then there is a subgroup $H$ of $G$ such that letting $c=E^G_H(a)$ and $d=E^G_H(b)$, we have:
\begin{enumerate}[(i)]
\item $cd=1$,
\item the ranks of $c$ and $d$ are both $\ge2$,
\item the support of $c$ generates $H$,
\item the support of $d$ generates $H$.
\end{enumerate}
Furthermore, if we also have $dc=1$, then we must have $a=c$ and $b=d$.
\end{prop}
\begin{proof}
We argue by induction on the sum of the ranks of $a$ and $b$.
For the initial step,
if $\rank(a)=\rank(b)=2$, (or, in fact, if either $\rank(a)=2$ or $\rank(b)=2$), then using
Proposition~\ref{prop:rank2} and taking $H$ to be the subgroup generated by the support of $a$,
the conclusion holds.
For the induction step, suppose $\rank(a)+\rank(b)>4$ and
let $H_1$ be the subgroup of $G$ generated by the support of $a$.
Letting $b^{(1)}=E^G_{H_1}(b)$, we have $ab^{(1)}=1$, so we must have $\rank(b^{(1)})\ge2$.
Of course, we have $a,b^{(1)}\in K[H_1]$.
If $b\ne b^{(1)}$, then the rank of $b^{(1)}$ is strictly smaller than the rank of $b$,
and we may apply the induction hypothesis to find a subgroup $H$ of $H_1$
so that $c=E^G_H(a)$ and $d=E^G_H(b^{(1)})=E^G_H(b)$
satisfy (i)--(iv)
and also such that $dc=1$ implies $c=a$ and $d=b^{(1)}$; in this last case, we have that $a$ is invertible,
and together with $ab=1=ab^{(1)}$ this yields $d=b^{(1)}=b$ (which is actually contrary to hypothesis).
Thus, we may suppose $b=b^{(1)}$, namely, that the support of $b$ is contained in $H_1$.
If the support of $b$ also generates $H_1$, then taking $H=H_1$, we are done.
Otherwise, letting $H_2$ be the subgroup generated by the support of $b$, we
have $H_2\subsetneq H_1$.
Therefore, letting  $a^{(1)}=E^G_{H_2}(a)$, this element $a^{(1)}$ 
must have rank strictly smaller than the rank of $a$.
But we still have $a^{(1)}b=1$, so $\rank(a^{(1)})\ge2$ and we may apply the induction hypothesis,
as above, to obtain a subgroup $H$ of $H_2$ so that  $c=E^G_H(a^{(1)})=E^G_H(a)$ and $d=E^G_H(b)$
satisfy (i)--(iv)
and also such that $dc=1$ implies $c=a^{(1)}$ and $d=b$; in this last case, we have that $b$ is invertible,
and together with $ab=1=a^{(1)}b$ this yields $c=a^{(1)}=a$ (which is actually contrary to hypothesis).
\end{proof}

\begin{cor}
The group ring $K[G]$ is directly finite if and only if whenever $a,b\in K[G]$ satisfy
$ab=1$ and that the supports of $a$ and of $b$, respectively, generate the same subgroup of $G$,
then we have $ba=1$.
\end{cor}

\medskip
Let us now turn to zero divisors.
Again, we suppose $K$ is a division ring and $G$ is a group and $a,b\in K[G]$ are nonzero and
are written as in~\eqref{eq:a=b=}, and we are interested in the situation when $ab=0$, in which case
we say that $a$ is a left zero divisor and $b$ is a right zero divisor and both are zero divisors.
Of course, by replacing $a$ with $r_0^{-1}a_0^{-1}a$ and $b$ with $bb_0^{-1}s_0^{-1}$,
we may assume $a_0=b_0=1$ and $r_0=s_0=1$.
We will use the following easy fact:
\begin{lemma}\label{lem:ab=0subgroup}
Let $a\in K[G]$ be a left (respectively, right) zero divisor.
Then there is nonzero $b\in K[G]$ whose support lies in the subgroup generated by the support of $a$,
satisfying $ab=0$ (respectively, $ba=0$).
\end{lemma}
\begin{proof}
We treat the case of $a$ being a left zero divisor, the other case being similar.
We suppose there is nonzero $c\in K[G]$ such that $ac=0$.
Let $H$ be the subgroup generated by the support of $a$ and let $E=E^G_H$.
By right multiplying $c$ by an appropriate group element, we may without loss of generality suppose
that the support of $c$ contains at least one element of $H$, so $E(c)\ne0$.
Then $0=E(ac)=aE(c)$, and $E(c)\ne0$.
Taking $b=E(c)$ we are done.
(Note that the rank of $b$ must actually be at least $2$).
\end{proof}

\begin{prop}\label{prop:ZDCrank2}
If $a\in K[G]$ is a zero divisor of rank $2$, then $a=sh(1-rg)$ for some $h,g\in G$ with
$g$ having finite order $n>1$,
and for some $r,s\in K$ such that $r^n=1$.
Moreover, in this case we have $ab=ba=0$ for
\begin{equation*}
b=(1+rg+r^2g^2+\cdots+r^{n-1}g^{n-1})h^{-1}.
\end{equation*}
\end{prop}
\begin{proof}
Suppose $a$ is a left zero divisor of rank $2$.
Replacing $a$ by $s^{-1}h^{-1}a$ for some $s\in K\backslash\{0\}$
and some $h\in G$, it will suffice to treat the case when $a=1-rg$ is 
a left zero divisor for some nontrivial $g\in G$ and some nonzero $r\in K$.
By Lemma~\ref{lem:ab=0subgroup}, we have $ab=0$ for some $b\in K[G]$ having support
in the cyclic subgroup generated by $g$, and, after right multiplying by an appropriate
rank--one element we may assume
\[
b=1+s_1g+s_2g^2+\cdots+s_{n-1}g^{n-1},
\]
where $n\ge2$ is such that $1,g,\ldots,g^{n-1}$ are distinct, and $s_j\in K$.
Now writing out $ab=0$, we get that $n$ is the order of $g$ and $r^n=1$
and $b=1+rg+r^2g^2+\cdots+r^{n-1}g^{n-1}$.

The case when $a$ is a right zero divisor of rank $2$ is treated similarly.
\end{proof}

Thus, the rank $2$ case of the zero divisor conjecture is also trivially true.
Furthermore, we have an analogue of Proposition~\ref{prop:EHab=1} for zero divisors
as well.
\begin{prop}\label{prop:ZDCH}
Suppose nonzero elements $a,b\in K[G]$ satisfy $ab=0$
and both contain $1$ (the identity element of $G$) in their supports.
Then there is a subgroup $H$ of $G$ such that letting $c=E^G_H(a)$ and $d=E^G_H(b)$, we have:
\begin{enumerate}[(i)]
\item $cd=0$,
\item $c\ne0$ and $d\ne0$,
\item the support of $c$ generates $H$,
\item the support of $d$ generates $H$.
\end{enumerate}
\end{prop}
\begin{proof}
We use induction on the sum of the ranks of $a$ and $b$.
Under the hypotheses, we must have $\rank(a),\rank(b)\ge2$.
For the initial step,
in the case when $\rank(a)=\rank(b)=2$, or, indeed, when either $\rank(a)$ or $\rank(b)$
equals $2$, the conclusion follows from Proposition~\ref{prop:ZDCrank2}, by letting $H$ be the group generated
by the support of $a$.
For the induction step, assume $\rank(a)+\rank(b)>4$ and
let $H_1$ be the subgroup generated by the support of $a$.
Then letting $b^{(1)}=E^G_{H_1}(b)$, we have $b^{(1)}\ne0$ and $0=E^G_{H_1}(ab)=ab^{(1)}$.
If the support of $b^{(1)}$ generates $H_1$, then letting $H=H_1$, we are done.
Otherwise, letting $H_2$ be the subgroup (of $H_1$)
generated by the support of $b^{(1)}$ and letting $a^{(1)}=E^G_{H_2}(a)$,
we have $0\ne a^{(1)}\ne a$, so $\rank(a^{(1)})<\rank(a)$, and $a^{(1)}b^{(1)}=0$.
Now the existence of $H$ follows from the induction hypothesis.
\end{proof}

\section{Stable Finiteness}
\label{sec:StabFin}

\begin{lemma}\label{lem:finGp}
Given a field $F$ and a positive integer $n$, there is a finite group $H$ such that
the group ring $F[H]$ has a subring isomorphic to $M_n(F)$.
\end{lemma}
\begin{proof}
We prove first the case $n=2$.
Let $p$ be the characteristic of the field $F$.
Consider the symmetric group $S_3=\langle a,b\mid a^3=b^2=1,\,bab=a^{-1}\rangle$.
Consider the representation $\pi$ of $S_3$ on $F^2$ given by 
\[
\pi(a)=\left(\begin{matrix}-1 & -1 \\1 & 0\end{matrix}\right),\qquad
\pi(b)=\left(\begin{matrix}0 & 1 \\1 & 0\end{matrix}\right),
\]
extended by linearity to a representation of $F[S_3]$.
We have
\[
\pi(a^2)=\left(\begin{matrix} 0& 1 \\-1 & -1\end{matrix}\right),\qquad
\pi(ab)=\left(\begin{matrix}-1 & -1 \\0 & 1\end{matrix}\right),\qquad
\pi(a^2b)=\left(\begin{matrix} 1 & 0 \\-1 & -1\end{matrix}\right).
\]
An easy row reduction computation shows that when $p\ne3$, we have $\lspan\pi(S_3)=M_2(F)$.
Let us assume $p\ne3$.
In the case $p>3$, the desired conclusion of the lemma will follow from Maschke's theorem, but by
performing the actual computations, we will now see that the conclusion holds also for $p=2$.
Let $Q=\frac13(2-a-a^2)\in F[S_3]$.
Then $Q^2=Q$, and $Q(F[S_3])Q$ is a subalgebra of $F[S_3]$.
An easy computation shows that $Q(F[S_3])Q$ has dimension $4$ over $F$, and
$\pi(Q)=\left(\begin{smallmatrix}1&0\\0&1\end{smallmatrix}\right)$;
this implies that the restriction of $\pi$ to $Q(F[S_3])Q$ is an isomorphism onto $M_2(F)$
and $Q(F[S_3])Q\cong M_2(F)$ as algebras.
The lemma is proved in the case of $n=2$ and $p\ne3$.

We now suppose $p>2$ and consider the dihedral group of order $8$
\[
\Dih_4=\langle c,d\mid c^4=d^2=1,\,dcd=c^{-1}\rangle
\]
and its representation on $F^2$ given by
\[
\sigma(c)=\left(\begin{matrix} 0& -1 \\ 1 & 0\end{matrix}\right),\qquad
\sigma(d)=\left(\begin{matrix} 1& 0 \\ 0 & -1\end{matrix}\right),
\]
which gives
\begin{gather*}
\sigma(c^2)=\left(\begin{matrix} -1& 0 \\ 0 & -1\end{matrix}\right),\qquad
\sigma(c^3)=\left(\begin{matrix} 0& 1 \\ -1 & 0\end{matrix}\right), \\
\sigma(cd)=\left(\begin{matrix} 0& 1 \\ 1 & 0\end{matrix}\right),\qquad
\sigma(c^2d)=\left(\begin{matrix} -1& 0 \\ 0 & 1\end{matrix}\right),\qquad
\sigma(c^3d)=\left(\begin{matrix} 0& -1 \\ -1 & 0\end{matrix}\right).
\end{gather*}
We easily see $\lspan\sigma(\Dih_4)=M_2(F)$.
Now the result follows by Maschke's theorem, but let us perform the easy calculation.
Letting $Q=\frac12(1-c^2)\in F[\Dih_4]$, we have $Q^2=Q$ and $Q(F[\Dih_4])Q$ is a subalgebra
of $F[\Dih_4]$.
We have $\sigma(Q)=\left(\begin{smallmatrix}1&0\\0&1\end{smallmatrix}\right)$ and $\dim(Q(F[\Dih_4])Q)=4$
and the restriction of $\sigma$ to $Q(F[\Dih_4])Q$ is an isomorphism onto $M_2(F)$.
Thus, the lemma is proved in the case $n=2$ and $p>2$.
Taken together, these considerations prove the lemma in the case of $n=2$.

For groups $H_1$ and $H_2$ we have the natural identification
$F[H_1\times H_2]\cong F[H_1]\otimes_FF[H_2]$, and for positive integers $m$ and $n$ we have
$M_m(F)\otimes_F M_n(F)\cong M_{mn}(F)$.
Therefore, starting from the case $n=2$ of the lemma and taking cartesian products of 
an appropriate group, arguing by induction we prove the lemma in the case when $n$ is a power of $2$.
Now taking corners of the matrix algebras $M_{2^k}(F)$ proves the lemma for arbitrary $n$.
\end{proof}

\begin{remark}
From the above proof, we see that the finite group $H$ can always be taken to be a cartesian product
of copies of the symmetric group $S_3$ or of the dihedral group $\Dih_4$, depending on the characteristic of $F$.
Thus, the hypothesis of Theorem~\ref{thm:KGsf} below can be correspondingly weakened
by requiring $K[\Gamma\times H]$ to be directly finite only for these groups $H$.
\end{remark}

Recall from the introduction that an algebra $\Afr$ is said to be stably finite if all matrix algebras $M_n(\Afr)$
over it are directly finite.
\begin{thm}\label{thm:KGsf}
Let $K$ be a division ring and $\Gamma$ a group.
If $K[\Gamma\times H]$ is directly finite for every finite group $H$,
then $K[\Gamma]$ is stably finite.
\end{thm}
\begin{proof}
Let $n$ be a positive integer.
Let $F$ be the base field of $K$.
By Lemma~\ref{lem:finGp}, choose a finite group $H$ so that $F[H]$ contains $M_n(F)$ as a subalgebra.
Then $F[H]$ contains a copy of $M_n(F)\oplus F$ as a unital subalgebra, and we have
\begin{multline*}
K[\Gamma\times H]\cong K[\Gamma]\otimes_{F}F[H]\supseteq K[\Gamma]\otimes_F(M_n(F)\oplus F) \\[1ex]
\cong(K[\Gamma]\otimes_FM_n(F))\oplus K[\Gamma]
\cong M_n(K[\Gamma])\oplus K[\Gamma]\supseteq M_n(K[\Gamma])\oplus K,
\end{multline*}
where all inclusions are as unital subalgebras.
Now given $c,d\in M_n(K[\Gamma])$ such that $cd=1$, take $a=c\oplus1$ and $b=d\oplus1$ in
$M_n(K[\Gamma])\oplus K$.  We have $ab=1$, and by the above inclusions and the direct finiteness of
$K[\Gamma\times H]$, we must have $ba=1$, so $dc=1$.
\end{proof}

Consequently, truth of the Direct Finiteness Conjecture implies truth of the stronger looking
Stable Direct Finiteness Conjecture.
\begin{cor}
For $K$ is a division ring, if $K[\Gamma]$ is directly finite for all groups $\Gamma$,
then $K[\Gamma]$ is stably finite for all groups $\Gamma$.
\end{cor}

\section{Universal  Left Invertible Element groups}
\label{sec:K}

As at the start of Section~\ref{sec:Obs}, let us consider elements $a,b\in K[G]$ whose ranks are $\ge2$
and so that $ab=1$,
and let us write
\[
a=r_0a_0+\cdots+r_{m-1}a_{m-1},\qquad b=s_0b_0+\cdots+s_{n-1}b_{n-1}
\]
with the same conventions, and with $a_0=b_0=1$.
Let $\pi$ be the partition of the set
\begin{equation}\label{eq:set}
\{0,\ldots,m-1\}\times\{0,\ldots,n-1\}
\end{equation}
defined by 
\begin{equation}\label{eq:pirel}
(i,j)\simpi(i',j')\text{ if and only if }a_ib_j=a_{i'}b_{j'},
\end{equation}
where $(i,j)\simpi(i',j')$ means that $(i,j)$ and $(i',j')$ belong
to the same set of the partition $\pi$.
Then we have, for all $E\in\pi$,
\begin{equation}\label{eq:rssum}
\sum_{(i,j)\in E}r_is_j=
\begin{cases}
1,&(0,0)\in E, \\
0,&(0,0)\notin E.
\end{cases}
\end{equation}
\begin{defi}
We call $\pi$ the {\em cancellation partition} for the pair $(a,b)$ with respect
to the orderings $(a_0,\ldots,a_{m-1})$ and $(b_0,\ldots,b_{n-1})$ of their supports.
\end{defi}

\begin{defi}\label{def:Gammapi}
Given a partition $\pi$ of the set~\eqref{eq:set}, we consider the group $\Gamma_\pi$ with presentation
\begin{equation}\label{eq:Gammapi}
\Gamma_\pi=\langle a_0,a_1,\ldots,a_{m-1},b_0,b_1,\ldots,b_{n-1}
\mid a_0=b_0=1,\,(a_ib_j=a_{i'}b_{j'})_{(i,j)\simpi(i',j')}\rangle,
\end{equation}
where the relations are indexed over the set of all pairs $((i,j),(i',j'))$ of elements of~\eqref{eq:set}
that belong a same set of the partition $\pi$.
\end{defi}

\begin{defi}\label{def:degen}
Let $m,n\in\{2,3,\ldots\}$ and let $\pi$ be a partition of the set~\eqref{eq:set}.
We say that $\pi$ is  {\em degenerate}
if, in the group $\Gamma_\pi$ above,
the group elements $a_0,a_1,\ldots,a_{m-1}$ are not distinct or the group elements $b_0,b_1,\ldots,b_{n-1}$
are not distinct.
\end{defi}

\begin{remark}\label{rem:degen}
A partition $\pi$ is clearly degenerate if we have $(i,j)\simpi(i,j')$ for any $j\ne j'$ or $(i,j)\simpi(i',j)$
for any $i\ne i'$, i.e., if $\pi$ groups together any two elements in the same row or column.
\end{remark}

\begin{defi}\label{def:real}
Let $m,n\in\{2,3,\ldots\}$ and let $\pi$ be a partition of the set~\eqref{eq:set}.
Let $K$ be a division ring and let $r_0,\ldots,r_{m-1},s_0,\ldots,s_{n-1}$ be nonzero
elements of $K$.
We say that $\pi$ is {\em realizable with}
$r_0,\ldots,r_{m-1},s_0,\ldots,s_{n-1}$ if the equalities~\eqref{eq:rssum} hold for all $E\in\pi$.
We say that $\pi$ is {\em realizable over} $K$ if it is realizable with some nonzero
elements of $K$, and we say that $\pi$ is {\em realizable} if it is realizable over some
division ring $K$.
\end{defi}

\begin{remark}\label{rem:singleton}
A realizable partition can have at most one singleton, which would then be $\{(0,0)\}$.
\end{remark}

We order the partitions of a set in the usual way, writing $\pi\le\sigma$ if every element of $\pi$ is a
subset of an element of $\sigma$.
Then $\Gamma_\sigma$ is a quotient of $\Gamma_\pi$ by the map sending canonical generators to their namesakes,
so $\sigma$ nondegenerate implies $\pi$ nondegenerate.

\begin{defi}
Let $K$ be a division ring and 
let $r_0,\ldots,r_{m-1},s_0,\ldots,s_{n-1}\in K\backslash\{0\}$.
A partition $\pi$ of the set~\eqref{eq:set} is
{\em minimally realizable with} $r_0,\ldots,r_{m-1},s_0,\ldots,s_{n-1}$
if it is minimal among all
the partitions that are realizable with $r_0,\ldots,r_{m-1},s_0,\ldots,s_{n-1}$.
We say $\pi$ is {\em minimally realizable over} $K$ if it is minimally realizable
with some choice of $r_0,\ldots,r_{m-1},s_0,\ldots,s_{n-1}\in K\backslash\{0\}$,
and it $\pi$ is simply {\em minimally realizable} if it is minimally realizable over some division ring $K$.
\end{defi}

\begin{remark}\label{rem:pairpartitions}
The partitions of the set~\eqref{eq:set} that
are minimally realizable over the field $\Fb_2$ of two elements are precisely the partitions having
only pairs except for the singleton set $\{(0,0)\}$.
Thus, the existence of such a partition implies that $m$ and $n$ are both odd.
\end{remark}

\begin{remark}\label{rem:minreal}
The notion of being minimally realizable over $K$ is ostensibly different
from being minimal among the partitions that are realizable over $K$,
just as being minimally realizable is different from being minimal among the realizable partitions.
This is because the quality of being minimally realizable is bound up with a particular choice of
field elements $r_0,\ldots,r_{m-1},s_0,\ldots,s_{n-1}$.
We see that this is important, for example
in the proof of Theorem~\ref{thm:KapULIE}.
However, see Remark~\ref{rem:whyminimal} for more on this.
\end{remark}

\begin{defi}
Let $K$ be a division ring and let $m,n\ge2$ be integers.
Let $\ULIE_K(m,n)$ be the set of all groups $\Gamma_\pi$ as in Definition~\ref{def:Gammapi}
as $\pi$ runs over all  partitions $\pi$ of the set~\eqref{eq:set} that are both nondegenerate
and minimally realizable over $K$.
We will say that an {\em ULIE$_K$ group} is one that belongs to the set $\bigcup_{m,n\ge2}\ULIE_K(m,n)$.
Similarly, we let $\ULIE(m,n)$ denote the set of all groups $\Gamma_\pi$ as $\pi$ runs over all partitions
$\pi$ of~\eqref{eq:set} that are both nondegenerate and minimally realizable, and
say that an {\em ULIE group} is one that belongs to the set 
$\bigcup_{m,n\ge2}\ULIE(m,n)$.
(ULIE is an acronym for {\em Universal Left Invertible Element}.)
Finally, we let $\ULIE^{(-)}_K(m,n)\subseteq\ULIE_K(m,n)$ be equal to$\ULIE_K(m,n)$ if $m\ne n$ and,
if $m=n$, we let it consist of the complement of the set of $\ULIE_K(m,n)$--groups $\Gamma_\pi$ for partitions $\pi$
that are minimally realizable for some $r_0,\ldots,r_{m-1},s_0,\ldots,s_{n-1}\in K$ and so that the partition forces equality
$r_0a_0+\cdots+r_{m-1}a_{m-1}=s_0b_0+\cdots+s_{n-1}b_{n-1}$ in the group ring $K[\Gamma_\pi]$.
\end{defi}

\begin{remark}\label{rem:whyminimal}
We have introduced the ULIE groups in order to restrict the class of groups $\Gamma$
that would need to be tested for $K[\Gamma]$ being directly finite, in order to prove Kaplansky's conjecture.
We see this in Theorems~\ref{thm:KapULIE} and~\ref{thm:KapULIE1}
and Corollaries~\ref{cor:ULIEK} and~\ref{cor:ULIEK1} below.
However, there is no harm in increasing the class of groups that are tested.
Keeping this in mind, one may find that it is better not to worry about minimally realizable partitions
$\pi$, but, for example, for a given rank pair
$(m,n)$ to test simply all groups $\Gamma_\pi$ over all nondegenerate partitions $\pi$ of
$\{0,\ldots,m-1\}\times\{0,\ldots,n-1\}$,
rather than first to decide which of them are minimally realizable or even realizable.
Of course, if we restrict to the field $K=\Fb_2$ of two elements, then,
as seen in Remark~\ref{rem:pairpartitions},
the minimally realizable partitions have a particularly simple form.
But for general $K$ this is not clear to us.

In connection with the Invertibles Conjecture,
in Theorem~\ref{thm:InvTorULIE} the implication (i)$\implies$(ii) does depend on taking
realizable partitions, though they need not be minimally realizable.
\end{remark}

\begin{defi}
For a division ring $K$ and for integers $m,n\ge2$,
we will say that {\em Kaplansky's Direct Finiteness Conjecture
holds over $K$ for rank pair} $(m,n)$ if for all groups $G$ and all $a,b\in K[G]$
with rank of $a$ equal to $m$ and rank of $b$ equal to $n$,
$ab=1$ implies $ba=1$.
We will say that {\em Kaplansky's Direct Finiteness Conjecture
holds over $K$ for rank} $m$ if it holds for all rank pairs $(m,n)$ with $n\ge2$, namely, if
for all groups $G$, right invertibility of $a\in K[G]$ with $\rank(a)=m$ implies invertibility of $a$.
We will say that {\em Kaplansky's Direct Finiteness Conjecture
holds over $K$} if it holds over $K$ for all rank pairs, namely, if $K[G]$ is directly
finite for all groups $G$.
\end{defi}

\begin{remark}\label{rem:Gop}
Given a group $G$, we can define the group $G^{\text{\it op}}$ to be the set $G$ equipped with the
opposite binary operation: the product of $g$ and $h$ in $G^{\text{\it op}}$ is defined to be the element
$hg$ of $G$.
Using $G^{\text{\it op}}$, we easily see that Kaplansky's Direct Finiteness Conjecture holds over $K$
for rank pair $(m,n)$ if and only if it holds over $K$ for rank pair $(n,m)$.
Furthermore, this implies that Kaplansky's Direct Finiteness Conjecture holds for rank $m$ if and only if
for any group $G$ and for any $a\in K[G]$ of rank $m$,
one--sided invertibility $a$ of implies invertibility of $a$.
\end{remark}

The idea of the following theorem was explained (and an adequate proof in the case $K=\Fb_2$ was given)
in the introduction.
\begin{thm}\label{thm:KapULIE}
Let $K$ be a division ring and let $m,n\ge2$ be integers.
Suppose that for every group $\Gamma\in\ULIE^{(-)}_K(m,n)$, the group ring $K[\Gamma]$ is
directly finite.
Then Kaplansky's Direct Finiteness Conjecture holds over $K$
for rank pair $(m,n)$.
\end{thm}
\begin{proof}
Let $G$ be any group and let $c$
and $d$ be elements of $K[G]$ having ranks $m$ and $n$, respectively, and assume $cd=1$.
We must show $dc=1$.
We may write $c=r_0c_0+\cdots+r_{m-1}c_{m-1}$
for distinct elements $c_0,\ldots,c_{m-1}$ of $G$,
and $d=s_0d_0+\cdots+s_{n-1}d_{n-1}$
for distinct elements $d_0,\ldots,d_{n-1}$ of $G$
and for nonzero elements
$r_0,\ldots r_{m-1},s_0,\ldots,s_{n-1}$ of $K$.

After renumbering, we may without loss of generality assume $c_0d_0=1$;
replacing $c$ by $c_0^{-1}c$ and $d$ by $dd_0^{-1}$,
we may assume $c_0=d_0=1$.
Let $\sigma$ be the partition of the set~\eqref{eq:set}
that is the cancellation partition for the pair $(c,d)$ with respect to the given
orderings of their supports.
Then there is a group homomorphism $\psi:\Gamma_\sigma\to G$ sending each $a_i$ to $c_i$ and each $b_j$ to $d_j$.
This implies that $\sigma$ is nondegenerate.
Clearly, it is realizable with $r_0,\ldots,r_{m-1},s_0,\ldots,s_{n-1}$.
Let $\pi\le\sigma$ be a partition of the set~\eqref{eq:set}
that is minimally realizable with $r_0,\ldots,r_{m-1},s_0,\ldots,s_{n-1}$.
Thus, taking $a=r_0a_0+\cdots+r_{m-1}a_{m-1}$ and $b=s_0b_0+\cdots+s_{n-1}b_{n-1}$ in $K[\Gamma_\pi]$,
where the $a_i$ and $b_j$ are the named generators in the group $\Gamma_\pi$ from Definition~\ref{def:Gammapi},
by virtue of the defining relations in~\eqref{eq:Gammapi},
we have $ab=1$.
Combining $\psi$ with the natural quotient group homomorphism $\Gamma_\pi\to\Gamma_\sigma$, we get
a group homomorphism $\phi:\Gamma_\pi\to G$ that extends linearly to
a unital ring homomorphism $K[\Gamma_\pi]\to K[G]$
sending $a$ to $c$ and $b$ to $d$.
By hypothesis either $K[\Gamma_\pi]$ is directly finite or we have $a=b$ in $K[\Gamma_\pi]$.
In either case, we conclude $ba=1$ in $K[\Gamma_\pi]$.
Therefore, we have $dc=1$ in $K[G]$.
\end{proof}

\begin{cor}\label{cor:ULIEK}
Let $K$ be a division ring.
Then Kaplansky's conjecture holds over $K$ if and only if for all 
ULIE$_K$ groups $\Gamma$, the group ring $K[\Gamma]$ is directly finite.
\end{cor}

A strategy for testing Kaplansky's Direct Finiteness Conjecture over a division ring $K$
is, thus, to check for direct finiteness of the groups belonging to $\ULIE_K(m,n)$ for various ranks $m$ and $n$.
In fact, as we will now show, if we proceed by starting small and incrementing $m$ and $n$ by only one each time,
then we can restrict
to testing a slightly smaller set of groups.

\begin{defi}
We consider the set~\eqref{eq:set}. 
We view this set as laid out
like an $m\times n$ matrix, with rows numbered $0$ to $m-1$ and columns numbered $0$ to $n-1$.
Given a parition $\pi$ of this set, we let $\simpi$ be the corresponding
equivalence relation, whose equivalence classes are the sets of the partition.
Let $\simr$ be the equivalence relation on $\{0,\ldots,m-1\}$ that is generated by
$\simpi$ under the projection onto the first coordinate, namely, generated by the relations
\[
\{i\simr i'\mid \exists j,j'\text{ with }(i,j)\simpi(i',j')\}.
\]
We say $\pi$ is {\em row connected} if $\simr$ has only one equivalence class.

Similarly,
let $\simc$ be the equivalence relation on $\{0,\ldots,n-1\}$ that is generated by
$\simpi$ under the projection onto the second coordinate, namely, generated by the relations
\[
\{j\simc j'\mid \exists i,i'\text{ with }(i,j)\simpi(i',j')\}.
\]
We say $\pi$ is {\em column connected} if $\simc$ has only one equivalence class.
\end{defi}

\begin{lemma}\label{lem:connected}
Let $K$ be a division ring and let $M,N\ge2$ be integers.
Let $G$ be a group and suppose $c,d\in K[G]$ have ranks $M$ and $N$, respectively,
both have the identity element of $G$ in their supports and satisfy $cd=1$.
Write $c=r_0c_0+\cdots r_{M-1}c_{M-1}$ and $d=s_0d_0+\cdots s_{N-1}d_{N-1}$ for group elements
$c_i$ and $d_j$ and assume $c_0=d_0=1$.
Let $\sigma$ be the
cancellation partition of $(c,d)$ with respect to these orderings of the supports.
\begin{enumerate}[(a)]
\item
Suppose Kaplansky's conjecture holds over $K$ for all rank pairs $(m,N)$, with $2\le m<M$.
Then $\sigma$ is row connected.
Moreover, if $\pi\le\sigma$ is a partition that is realizable
with $r_0,\ldots,r_{M-1},s_0,\ldots,s_{N-1}$, then $\pi$ is row connected.
\item
Suppose Kaplansky's conjecture holds over $K$ for all rank pairs $(M,n)$, with $2\le n<N$.
Then $\sigma$ is column connected.
Moreover, if $\pi\le\sigma$ is a partition that is realizable
with $r_0,\ldots,r_{M-1},s_0,\ldots,s_{N-1}$, then $\pi$ is column connected.
\end{enumerate}
\end{lemma}
\begin{proof}
For part~(a),
let $\Gamma_\pi$ be the group with presentation~\eqref{eq:Gammapi}.
Then there is a group homomorphism $\phi:\Gamma_\pi\to G$ sending each $a_i$ to $c_i$ and each $b_j$ to $d_j$.
Since $c_0,\ldots,c_{M-1}$ are distinct elements of $G$ and $d_0,\ldots,d_{N-1}$ are distinct elements of $G$, it follows
that $\pi$ is nondegenerate.
Letting $a=r_0a_0+\cdots+r_{M-1}a_{M-1}$ and $b=s_0b_0+\cdots+s_{N-1}b_{N-1}$, we also have $ab=1$ in $K[\Gamma_\pi]$.

Suppose, to obtain a contradiction, that $\pi$ is not row connected.
Then, after renumbering if necessary, we may assume there is $\ell\in\{1,\ldots,M-1\}$
such that $\{0,\ldots,\ell-1\}$ is a union of equivalence classes of $\simr$, i.e., such that
$i\not\simr i'$ whenever $0\le i<\ell\le i'<M$.
(Actually, using nondegeneracy and Remark~\ref{rem:degen}, we must have $2\le\ell\le M-2$.)
We have $a=a'+a''$, where
\begin{align*}
a'&=r_0a_0+\cdots+r_{\ell-1}a_{\ell-1} \\
a''&=r_\ell a_\ell+\cdots+r_{M-1}a_{M-1}.
\end{align*}
Using the defining relations of $\Gamma_\pi$ and the fact that $\pi$ is realizable with
$r_0,\ldots,r_{M-1}$, $s_0,\ldots,s_{N-1}$,
we get $a'b=1$ and $a''b=0$.

Since $\rank(a')=\ell<M$ and $\rank(b)=N$, by hypothesis we have $ba'=1$.
But then, in $K[\Gamma_\pi]$, we have
\[
0=0a'=(a''b)a'=a''(ba')=a''.
\]
Since $\pi$ is nondegenerate and all the $r_j$ are nonzero, this gives a contradiction.

The proof of part~(b) is similar:
assuming $\pi$ is not column connected,
we get analogously $b=b'+b''$ with $ab'=1$ and $ab''=0$, and this yields $b'a=1$ and $b''=0$, giving a contradiction.
\end{proof}

\begin{defi}
Let $K$ be a division ring and let $m,n\ge2$ be integers.
Let $\ULIE^{(1)}_K(m,n)$ be the set of all groups $\Gamma_\pi$ as in Definition~\ref{def:Gammapi}
as $\pi$ runs over all  partitions $\pi$ of the set~\eqref{eq:set} that are nondegenerate,
minimally realizable over $K$, row connected and column connected.
We will say that an {\em ULIE$^{(1)}_K$ group} is one that belongs to the set $\bigcup_{m,n\ge2}\ULIE^{(1)}_K(m,n)$.
Similarly, we let $\ULIE^{(1)}(m,n)$ denote the set of all groups $\Gamma_\pi$ as $\pi$ runs over all partitions
$\pi$ of~\eqref{eq:set} that are nondegenerate, minimally realizable, row connected and column connected, and
say that an {\em ULIE$^{(1)}$ group} is one that belongs to the set 
$\bigcup_{m,n\ge2}\ULIE^{(1)}(m,n)$.
Finally, we let $\ULIE^{(1-)}_K(m,n)\subseteq\ULIE^{(1)}_K(m,n)$ be equal to$\ULIE^{(1)}_K(m,n)$ if $m\ne n$ and,
if $m=n$, we let it consist of the complement of the set of $\ULIE^{(1)}_K(m,n)$--groups $\Gamma_\pi$ for partitions $\pi$
that are minimally realizable for some $r_0,\ldots,r_{m-1},s_0,\ldots,s_{n-1}\in K$ and so that the partition forces equality
$r_0a_0+\cdots+r_{m-1}a_{m-1}=s_0b_0+\cdots+s_{n-1}b_{n-1}$ in the group ring $K[\Gamma_\pi]$.
\end{defi}

Now Lemma~\ref{lem:connected} gives the following variant of Theorem~\ref{thm:KapULIE}:
\begin{thm}\label{thm:KapULIE1}
Let $K$ be a division ring and let $M,N\ge2$ be integers.
Suppose that for every group
\begin{equation}\label{eq:ULIE1cup}
\Gamma\in\bigcup_{\substack{2\le m\le M \\2\le n\le N}}\ULIE^{(1-)}_K(m,n),
\end{equation}
the group ring $K[\Gamma]$ is directly finite.
Then Kaplansky's Direct Finiteness Conjecture holds over $K$
for rank pair $(M,N)$.
\end{thm}
\begin{proof}
Arguing first by induction on $M+N$, we may assume that Kaplansky's Direct Finiteness Conjecture holds over $K$
for all
rank pairs $(m,n)$ appearing in~\eqref{eq:ULIE1cup} provided $(m,n)\ne(M,N)$.
We now proceed as in the proof of Theorem~\ref{thm:KapULIE}, with $(M,N)$ replacing $(m,n)$, 
except we note
that the equality $ab=1$ in $K[\Gamma_\pi]$ implies, grace of Lemma~\ref{lem:connected},
that $\pi$ is row connected and column connected.
\end{proof}

\begin{cor}\label{cor:ULIEK1}
Let $K$ be a division ring.
Then Kaplansky's conjecture holds over $K$ if and only if for all 
ULIE$^{(1-)}_K$ groups $\Gamma$, the group ring $K[\Gamma]$ is directly finite.
\end{cor}

From the calculations reported in section~\ref{sec:calcs}, we now have the following:
\begin{prop}\label{prop:DFCcases}
Let $m$ and $n$ be odd integers with either (a) $\min(m,n)=3$ and $\max(m,n)\le11$ or (b) $m=n=5$.
Then Kaplansky's Direct Finiteness Conjecture holds over the field $\Fb_2$ of two elements for rank pair $(m,n)$.
\end{prop}
\begin{proof}
All of the ULIE$^{(-)}_{\Fb_2}(m,n)$ groups $\Gamma_\pi$ have been computed and they are described in section~\ref{sec:calcs}.
In case~(a), the are all amenable while in case~(b), all are amenable except for two, 
which by the main result of~\cite{CD}, \cite{ES11} and \cite{P11}, are seen to be sofic.
Hence, by~\cite{ES04}, each group ring $\Fb_2[\Gamma_\pi]$ is directly finite.
Now Theorem~\ref{thm:KapULIE} applies.
\end{proof}

\section{An infinite family of non-amenable ULIE groups}
\label{sec:InfFam}

We describe infinitely many nondegenerate partitions that yield non-amenable ULIE$_{\Fb_2}^{(1)}$ groups.
These groups are, however, known to be sofic.

For an integer $n\ge2$, we describe a pair partition $\pi$ of the set
\[
\big(\{0,1,\ldots,2n\}\times\{0,1,\ldots,2n\}\big)\backslash\{(0,0)\}.
\]
The pair partition is described on a $(2n+1)\times(2n+1)$ grid (rows and columns numbered from $0$ to $2n$)
by (a) drawing lines between positions and 
(b) writing numbers in the positions.
If there is a straight
line between positions $(i,j)$ and $(k,\ell)$
or if the same number is written in positions $(i,j)$ and $(k,\ell)$,
then this indicates $\{(i,j),(k,\ell)\}\in\pi$.
See Figure~\ref{fig:template}.
\setlength{\unitlength}{2em}
\begin{figure}[b]
\caption{A $(2n+1)\times(2n+1)$ partial pair partition.}
\label{fig:template}
\begin{picture}(11,11.5)(0,-11.5)

\put(1,-1){\line(0,-1){10}} 
\put(2,-1){\line(0,-1){1}}
\put(1,-1){\line(1,0){10}}
\put(1,-2){\line(1,0){1}}

\multiputlist(2.7,-0.5)(1,0)
 {$b_1$,$b_2$,$b_3$,$b_4$,$b_5$,$b_6$,$\cdots$,$b_{2n-1}$,$b_{2n}$} 
\multiputlist(0.2,-2.5)(0,-1)
 {$a_1$,$a_2$,$a_3$,$a_4$,$a_5$,$a_6$,$\vdots$,$a_{2n-1}$,$a_{2n}$} 

\put(1.5,-1.6){$0$} 
\multiputlist(2.7,-1.5)(1,0){$1$,$2$} 
\multiputlist(1.5,-2.5)(0,-1){$1$,$2$} 

\put(2.7,-2.5){\line(1,-1){1}} 
\put(2.7,-3.5){\line(1,1){1}}

\matrixput(4.7,-4.5)(2,0){2}(0,-2){2}{\line(1,-1){1}}
\matrixput(4.7,-5.5)(2,0){2}(0,-2){2}{\line(1,1){1}}
\multiput(8.7,-4.5)(0,-1){4}{$\cdots$}
\multiput(9.7,-4.5)(0,-2){2}{\line(1,-1){1}}
\multiput(9.7,-5.5)(0,-2){2}{\line(1,1){1}}

\multiput(4.7,-8.5)(1,0){4}{$\vdots$}
\multiput(9.7,-8.5)(1,0){2}{$\vdots$}

\multiput(4.7,-9.5)(2,0){2}{\line(1,-1){1}}
\multiput(4.7,-10.5)(2,0){2}{\line(1,1){1}}
\multiput(8.7,-9.5)(0,-1){2}{$\cdots$}
\put(9.7,-9.5){\line(1,-1){1}}
\put(9.7,-10.5){\line(1,1){1}}

\dottedline{0.2}(4.1,-1)(4.1,-11)
\dottedline{0.2}(1,-4.1)(11,-4.1)
\end{picture}
\end{figure}
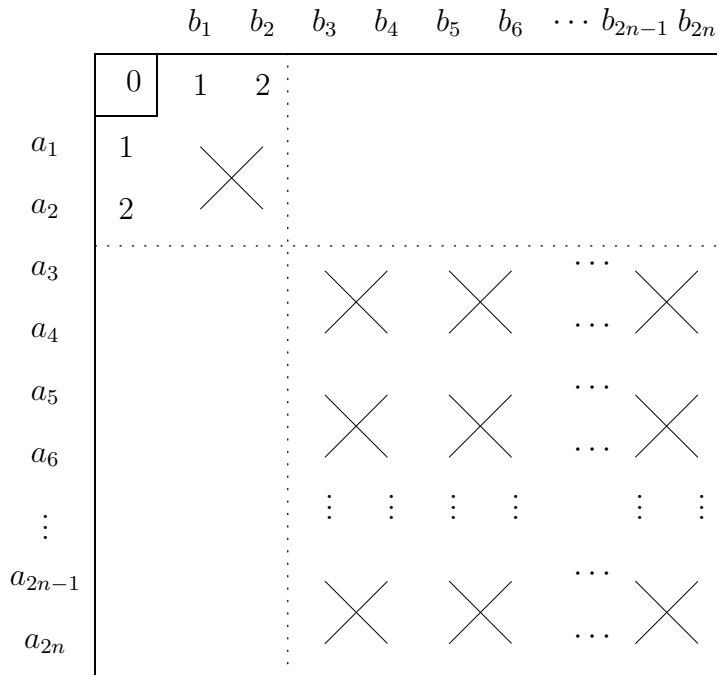
For example, the cross in the upper left corner indicates the pairings $(1,1)\sim(2,2)$ and $(1,2)\sim(2,1)$.
Also, the numbers $1$ and $2$ in the picture indicate,
respectively, the pairings $(0,1)\sim(1,0)$ and $(0,2)\sim(2,0)$.
The long $3\times(2n-2)$ block on the upper right and the tall $(2n-2)\times 3$ block on the lower left,
both partially outlined with dotted lines,
are still to be filled in.
We have written the group elements above and to the left of the grid, to remind us that a pairing between
positions $(i,j)$ and $(k,\ell)$ leads to the relation $a_ib_j=a_kb_\ell$ in the group $\Gamma_\pi$.

The finitely presented group with generators $a_1,\ldots,a_{2n},b_1,\ldots,b_{2n}$
and relations dictated by the pairings indicated in Figure~\ref{fig:template}
is isomorphic to the group with
presentation
\begin{equation}\label{eq:ppmGp}
\langle s,t,a_1,a_3,a_5\ldots,a_{2n-1},b_3,b_5,\ldots,b_{2n-1}\mid
s^2=t^2=[a_1,s]=1\rangle,
\end{equation}
where $[x,y]$ means the multiplicative commutator $xyx^{-1}y^{-1}$,
with the isomorphism implemented by \pagebreak[1]
\begin{alignat}{2}
a_j&\mapsto a_j &\quad &(j\text{ odd}) \notag \\
b_j&\mapsto b_j &&(j\text{ odd},\,j\ge3) \notag  \pagebreak[1] \\
a_2&\mapsto a_1s \label{eq:gpid3} \\
b_2&\mapsto a_1s \label{eq:gpid4} \\
a_k&\mapsto a_{k-1}\,t &&(k\text{ even},\,k\ge4) \label{eq:gpid5} \\
b_k&\mapsto t\,b_{k-1} &&(k\text{ even},\,k\ge4) \label{eq:gpid6} 
\end{alignat}
We relabel the group elements to incorporate the identifications~\eqref{eq:gpid3}--\eqref{eq:gpid6}.
Thus, the top row in Figure~\ref{fig:template} becomes
\[
a_1\quad a_1s\quad b_3\quad tb_3\quad b_5\quad tb_5\quad \cdots\quad b_{2n-1}\quad tb_{2n-1}
\]
while the left--most column in Figure~\ref{fig:template} becomes (the transpose of)
\[
a_1\quad a_1s\quad a_3\quad a_3t\quad a_5\quad a_5t\quad \cdots\quad a_{2n-1}\quad a_{2n-1}t
\]

Now we fill in the remaining pairings to create a complete pair partition.
The upper right--hand $3\times(2n-2)$ block gets filled in as indicated in Figure~\ref{fig:upperrightblock},
\setlength{\unitlength}{3em}
\begin{figure}[b]
\caption{The numbers we put into the upper right $3\times(2n-2)$ block.}
\label{fig:upperrightblock}
\begin{picture}(9.5,4)(0,-4)

\put(1,-1){\line(0,-1){3}} 
\put(1,-1){\line(1,0){8.5}}

\multiputlist(1.7,-0.5)(1.2,0)
 {$b_3$,$tb_3$,$b_5$,$tb_5$,$\cdots$,$b_{2n-1}$,$tb_{2n-1}$} 
\multiputlist(0.5,-2.5)(0,-1)
 {$a_1$,$a_1s$} 

\multiputlist(1.6,-1.5)(1.2,0){$3$,$4$,$5$,$6$,$\cdots$,$2n-1$,$2n$} 
\multiputlist(1.6,-2.5)(1.2,0){$2n+1$,$2n+2$,$2n+3$,$2n+4$,$\cdots$,$4n-3$,$4n-2$} 
\multiputlist(1.6,-3.5)(1.2,0){$4n-1$,$4n$,$4n+1$,$4n+2$,$\cdots$,$6n-5$,$6n-4$} 
\end{picture}
\end{figure}
while the lower left--hand $(2n-2)\times3$ block gets filled in as indicated in Figure~\ref{fig:lowerleftblock}.
\begin{figure}[hb]
\caption{The numbers we put into the lower left $(2n-2)\times3$ block.}
\label{fig:lowerleftblock}
\begin{picture}(4.5,8)(0,-8)

\put(1,-1){\line(1,0){3.5}} 
\put(1,-1){\line(0,-1){7}}

\multiputlist(2.8,-0.5)(1.2,0)
 {$a_1$,$a_1s$} 
\multiputlist(0.5,-1.5)(0,-1)
 {$a_3$,$a_3t$,$a_5$,$a_5t$,$\vdots$,$a_{2n-1}$,$a_{2n-1}t$} 

\multiputlist(1.6,-1.5)(1.2,0){$2n+1$,$4n-1$,$3$}
\multiputlist(1.6,-2.5)(1.2,0){$2n+2$,$4n$,$4$}
\multiputlist(1.6,-3.5)(1.2,0){$2n+3$,$4n+1$,$5$}
\multiputlist(1.6,-4.5)(1.2,0){$2n+4$,$4n+2$,$6$}
\multiputlist(1.6,-5.5)(1.2,0){$\vdots$,$\vdots$,$\vdots$}
\multiputlist(1.6,-6.5)(1.2,0){$4n-3$,$6n-5$,$2n-1$}
\multiputlist(1.6,-7.5)(1.2,0){$4n-2$,$6n-4$,$2n$}
\end{picture}
\end{figure}
This completes the pair partition $\pi$
of $(\{0,1,\ldots,2n\}\times\{0,1,\ldots,2n\})\backslash\{(0,0)\}$.
The group $\Gamma_\pi$ equals the quotient of the group~\eqref{eq:ppmGp}
by the additional relations corresponding the the numbers $3$ to $6n-4$, according
to Figures~\ref{fig:upperrightblock} and~\ref{fig:lowerleftblock}.

Let $R_j$ denote the relation implied by the pairing indicated
by the number $j$ in Figures~\ref{fig:upperrightblock} and~\ref{fig:lowerleftblock}.
We have
\begin{alignat*}{2}
R_3:&\qquad & b_3&=a_3a_1s \\
R_4:& & tb_3&=a_3ta_1s \\
\vdots\quad&&&\quad\vdots \\
R_{2n-1}:& & b_{2n-1}&=a_{2n-1}a_1s \\
R_{2n}:& & tb_{2n-1}&=a_{2n-1}ta_1s.
\end{alignat*}
These are equivalent to the relations
\begin{align}
b_j&=a_ja_1s \label{eq:c1} \\
[t,a_j]&=1 \label{eq:c2}
\end{align}
for all $j$ odd, $3\le j\le 2n-1$.
For this same range of $j$ values, relations $R_{2n+1}$ to $R_{4n-2}$ give us
\begin{align}
a_1b_j&=a_j \label{eq:c8} \\
a_1tb_j&=a_jt, \label{eq:c9}
\end{align}
which, using~\eqref{eq:c1} and~\eqref{eq:c2} and $[a_1,s]=s^2=1$, are seen to be equivalent to
\begin{align}
a_j^{-1}a_1a_j&=a_1^{-1}s \label{eq:c3} \\
[t,a_1]&=1. \label{eq:c4}
\end{align}
Again for the same range of $j$ values, $R_{4n-1}$ to $R_{6n-4}$ give us
\begin{align}
a_1sb_j&=a_ja_1 \label{eq:c5} \\
a_1stb_j&=a_jta_1. \label{eq:c6}
\end{align}
Using~\eqref{eq:c1} and $[a_1,s]=s^2=1$, the first of these is equivalent to 
\begin{equation}\label{eq:c7}
a_j^{-1}a_1sa_j=s,
\end{equation}
while using also~\eqref{eq:c2} and~\eqref{eq:c4}, we see that~\eqref{eq:c6} yields
\[
[t,s]=1.
\]
Taking~\eqref{eq:c3} and~\eqref{eq:c7} together gives
\begin{equation*}
a_j^{-1}sa_j=a_1,
\end{equation*}
which implies $a_1^2=1$.

Therefore, in the group $\Gamma_\pi$, the relations
\begin{align}
s^2&=t^2=[a_1,s]=[a_1,t]=[s,t]=1, \label{eq:r1} \\
&([a_j,t]=1)_{3\le j\le 2n-1,\;j\text{ odd}}, \label{eq:r2} \\
&(a_j^{-1}sa_j=a_1)_{3\le j\le 2n-1,\;j\text{ odd}}, \label{eq:r3} \\
&(a_j^{-1}a_1a_j=a_1s)_{3\le j\le 2n-1,\;j\text{ odd}} \label{eq:r4} 
\end{align}
hold, and we easily see that they imply the relations~\eqref{eq:c1}, \eqref{eq:c2},
\eqref{eq:c8}, \eqref{eq:c9}, \eqref{eq:c5} and~\eqref{eq:c6}.
Thus, $\Gamma_\pi$ has presentation
with generators $s,t,a_1,a_3,\ldots,a_{2n-1}$
and relations~\eqref{eq:r1}--\eqref{eq:r4}.

We see that the relations \eqref{eq:r1}--\eqref{eq:r4} are equivalently described by:
\begin{enumerate}[(i)]
\item $t$ is in the center
\item the subgroup $H$ generated by $s$ and $a_1$ is isomorphic to $\bZ_2\times\bZ_2$
\item conjugation by $a_j$ for every $j\in\{3,5,\ldots,2n-1\}$
implements the same automorphism $\alpha$, of $H$, which is the automorphism
of order $3$ that cycles the nontrivial elements of $H$.
\end{enumerate}
The group $\Gamma_\pi$ is, therefore, isomorphic to
\begin{equation}\label{eq:nonamenGammapi}
\overset{t}{\bZ_2}\times
\bigg((\overset{s}{\bZ_2}\times\overset{a_1}{\bZ_2})\rtimes_{\alpha*\cdots*\alpha}
(\underset{n-1\text{ times}}{\underbrace{\overset{a_3}{\bZ}*\overset{a_5}{\bZ}*\cdots*\overset{a_{2n-1}}{\bZ}}})\big),
\end{equation}
where the symbols appearing above the cyclic groups indicate the corresponding generators of the groups.
When $n\ge3$, this group is non-amenable.
The semidirect product group appearing above is isomorphic to the free product of $n-1$ copies of the amenable group
\[
(\bZ_2\times\bZ_2)\rtimes_\alpha\bZ
\]
with amalgamation over $\bZ_2\times\bZ_2$.
Therefore, the group $\Gamma_\pi$ is sofic, (by the main result of~\cite{CD}, \cite{ES11} and \cite{P11}).

\section{The Invertibles Conjecture}
\label{sec:InvTor}

See Conjecture~\ref{conj:IT} for a statement of the Invertibles Conjecture.
Here are some related finer considerations.

\begin{defi}
Let $K$ be a division ring.
We will say that the {\em Invertibles Conjecture holds over} $K$ if $K[G]$ contains no
one--sided invertible elements of rank $>1$ for all torsion--free groups $G$.
For integers $m,n\ge2$, we will say that  the {\em Invertibles Conjecture holds for rank pair} $(m,n)$ if
for all division rings $K$ and all torsion--free groups $G$, $K[G]$ contains no two elements $a$ and $b$
having ranks $m$ and $n$, respectively, such that $ab=1$.
We will say that the {\em Invertibles Conjecture holds for rank} $m$ if it holds for rank pairs $(m,n)$,
for all integers $n\ge2$, namely, if the existence of a right--invertible element of rank $m$ in a group algebra
$K[G]$ implies $G$ has torsion.
(By the method described in Remark~\ref{rem:Gop},
we may replace ``right--invertible'' by ``one--sided invertible''
in the previous sentence.)
Intersections of these properties (e.g., over $K$ for rank pair $(m,n)$) have the obvious meaning.
\end{defi}

As a consequence of Proposition~\ref{prop:rank2}, we have:
\begin{thm}\label{thm:InvTorrank2}
The Invertibles Conjecture holds for rank $2$.
\end{thm}

We first describe the smallest normal subgroup whose corresponding quotient is torsion--free.
This is surely well known, but it doesn't take long.

\begin{defi}\label{def:Ntor}
Given a group $\Gamma$, let $N_\tor^{(1)}(\Gamma)$ be the smallest normal subgroup of $\Gamma$
that contains all torsion elements of $\Gamma$.
We now recursively define normal subgroups $N_{\tor,n}$ of $\Gamma$, $n\ge1$, by
letting $N_{\tor,1}=N_\tor^{(1)}(\Gamma)$ and,
given $N_{\tor,n}$, letting $\phi_n:\Gamma\to\Gamma/N_{\tor,n}$ be the quotient map and
$N_{\tor,n+1}=\phi_n^{-1}(N_\tor^{(1)}(\Gamma/N_{\tor,n}))$.
Let $N_\tor(\Gamma)=\bigcup_{n=1}^\infty N_{\tor,n}$.
Clearly, $N_\tor(\Gamma)$ is a normal subgroup of $\Gamma$.
\end{defi}

\begin{prop}\label{prop:Ntor}
If $\Gamma/N_\tor(\Gamma)$ is nontrivial, then it is torsion--free.
Moreover, if $N$ is a normal subgroup of $\Gamma$ so that $\Gamma/N$ is torsion--free,
then $N_\tor(\Gamma)\subseteq N$.
\end{prop}
\begin{proof}
If $g\in\Gamma$ and $g^k\in N_\tor(\Gamma)$ for some $k\in\Nats$, then $g^k\in N_{\tor,n}$ for some $n\in\Nats$,
and, consequently, $g\in N_{\tor,n+1}$, so $g\in N_\tor(\Gamma)$.
This implies the first statement.
If $\Gamma/N$ is torsion--free, then clearly $N_\tor^{(1)}(\Gamma)\subseteq N$.
Now for any $g\in\Gamma$ so that $g^k\in N_\tor^{(1)}(\Gamma)$,
if $g\notin N$, then the class of $g$ would have finite order in $\Gamma/N$
contrary to hypothesis; thus, $N_{\tor,2}\subseteq N$.
Continuing in this way, we see by induction that $N_{\tor,n}\subseteq N$ for all $n$.
So $N_\tor(\Gamma)\subseteq N$.
\end{proof}

\begin{thm}\label{thm:InvTorULIE}
Let $K$ be a division ring and let $m,n\ge2$ be integers.
Then the following are equivalent:
\begin{enumerate}[(i)]
\item The Invertibles Conjecture holds over $K$ for rank pair $(m,n)$.
\item for every
$\ULIE_K(m,n)$--group $\Gamma$ with its canonical generators $1=a_0,a_1,\ldots,a_{m-1}$
and $1=b_0,b_1,\ldots,b_{n-1}$,
letting $\phi:\Gamma\to\Gamma/N_\tor(\Gamma)$ be the quotient map,
we have
$\phi(a_i)=\phi(a_{i'})$ for some $0\le i<i'\le m-1$ or $\phi(b_j)=\phi(b_{j'})$
for some $0\le j<j'\le n-1$.
\end{enumerate}
\end{thm}
\begin{proof}
For (ii)$\implies$(i), 
suppose the Invertibles Conjecture over $K$ for rank pair $(m,n)$ fails.
Then there is a torsion--free group $G$ such that 
$K[G]$ contains elements $\at$ of rank $m$ and $\bt$ of rank $n$ such that $\at\bt=1$.
After allowable modifications, we may without loss of generality write
$\at=1+r_1\at_1+\cdots r_{m-1}\at_{m-1}$
and $\bt=s_01+s_1\bt_1+\cdots+s_{n-1}\bt_{n-1}$
for some $\at_1,\ldots,\at_{m-1}$ distinct, nontrivial elements
of $G$ and $\bt_1,\ldots,\bt_{n-1}$ distinct, nontrivial elements of $G$ and for
$r_1,\ldots,r_{m-1},s_0,\ldots,s_{n-1}\in K\backslash\{0\}$.
Letting $\sigma$ be the cancellation partition for $\at\bt=1$ and taking $\pi\le\sigma$ that
is minimally realizable for $r_0,\ldots,r_{m-1},s_0,\ldots,s_{n-1}$, we have a group homomorphism
from the $\ULIE_K(m,n)$ group $\Gamma_\pi$ into
$G$ 
that sends canonical generators $a_i$ to $\at_i$ and $b_i$ to $\bt_i$.
Of, course, we have $ab=1$ in $K[\Gamma_\pi]$, where $a=r_0a_0+\cdots+r_{m-1}a_{m-1}$ and $b=s_0b_0+\cdots+s_{n-1}b_{n-1}$.
Since $G$ is torsion--free, by Proposition~\ref{prop:Ntor}, the kernel of the above homomorphism contains
$N_\tor(\Gamma_\pi)$.
Since $1,\at_1,\ldots,\at_{m-1}$ are distinct and $1,\bt_1,\ldots,\bt_{n-1}$ are distinct it follows that the
images of $1,a_1,\ldots,a_{m-1}$ in the quotient $\Gamma/N_\tor(\Gamma_\pi)$ are distinct,
as are the images of of $1,b_1,\ldots,b_{n-1}$,
and (ii) fails.

For (i)$\implies$(ii), suppose that for some $\ULIE_K(m,n)$--group $\Gamma=\Gamma_\pi$
and for $\phi$ the quotient map to $\Gamma/N_\tor(\Gamma)$,
the elements $1,\phi(a_1),\ldots,\phi(a_{m-1})$ are distinct
and $1,\phi(b_1),\ldots,\phi(b_{n-1})$ are distinct.
Now the partition $\pi$ is realizable with some $r_0,\ldots,r_{m-1},s_0,\ldots,s_{n-1}\in K\backslash\{0\}$,
so letting $a=r_01+r_1a_1+\cdots+r_{m-1}a_{m-1}$ and $b=s_01+s_1b_1+\cdots+s_{n-1}b_{n-1}$ in
$K[\Gamma]$, we have $ab=1$.
Extending the quotient map $\phi$ linearly to a ring homomorphism $K[\Gamma]\to K[\Gamma/N_\tor(\Gamma)]$,
we get that $\phi(a)$ has rank $m$ and $\phi(b)$ has rank $n$ and $\phi(a)\phi(b)=1$.
In particular, $\Gamma/N_\tor(\Gamma)$ is nontrivial.
By Proposition~\ref{prop:Ntor}, it is torsion--free.
So the Invertibles Conjecture fails over $K$ for rank pair $(m,n)$.
\end{proof}

\begin{remark}\label{rem:ICabelian}
It is well known and easy to show that for a torsion--free abelian group $G$ and $K$ a division ring,
$K[G]$ has no invertible elements of rank strictly greater than $1$.
Thus, the setting of Theorem~\ref{thm:InvTorULIE}, if $\Gamma/N_\tor(\Gamma)$ is abelian,
then $1,\phi(a_1),\ldots,\phi(a_{m-1})$ cannot be distinct.
\end{remark}

\begin{example}
For the non-amenable ULIE groups $\Gamma_\pi$ considered at~\eqref{eq:nonamenGammapi} in Section~\ref{sec:InfFam},
we easily see $\Gamma/N_\tor(\Gamma)$ is a copy of the free group on $n-1$ generators, but 
the quotient map sends $a_1$ to the identity.
\end{example}

If, instead of considering each rank pair $(m,n)$ individually, we start small and increase one rank
at a time, then we can get away with considering a smaller set of partitions and corresponding ULIE groups.

Keeping in mind the special role of $(0,0)$ in
\begin{equation}\label{eq:set2nd}
\{0,\ldots,m-1\}\times\{0,\ldots,n-1\}
\end{equation}
as pertains to ULIE groups, we make the following definition.
\begin{defi}\label{def:simple}
Let $m,n\ge2$ be integers and let
$\pi$ be a partition of~\eqref{eq:set2nd}.
An {\em invariant subgrid} of $\pi$ is a pair $(R,C)$ with $0\in R\subset\{0,1,\ldots,m-1\}$
and $0\in C\subset\{0,1,\ldots,n-1\}$, $|R|\ge2$ and $|C|\ge2$, so that whenever $(i,j)\in R\times C$
and $(i,j)\simpi(i',j')\in\{0,\ldots,m-1\}\times\{0,\ldots,n-1\}$, then  $(i',j')\in R\times C$.
The subgrid $(R,C)$ is {\em proper} if either $|R|<m$ or $|C|<n$.
\end{defi}

Note that partitions without proper invariant subgrids must be row and column connected.

\begin{lemma}\label{lem:simple}
Let $K$ be a division ring and let $M,N\ge2$ be integers.
Suppose $G$ is a torsion--free group and suppose $c,d\in K[G]$ having ranks $M$ and $N$, respectively,
both have the identity element of $G$ in their supports and satisfy $cd=1$.
Write $c=r_0c_0+\cdots r_{M-1}c_{M-1}$ and $d=s_0d_0+\cdots s_{N-1}d_{N-1}$ for group elements
$c_i$ and $d_j$ and assume $c_0=d_0=1$.
Let $\sigma$ be the
cancellation partition of $(c,d)$ with respect to these orderings of the supports
and let $\pi\le\sigma$ be any partition that is realizable with
$r_0,\ldots,r_{M-1},s_0,\ldots,s_{N-1}$.
Suppose the Invertibles Conjecture holds over $K$ for all rank pairs $(m,n)$, with
$2\le m\le M$ and $2\le n\le N$ and $(m,n)\ne(M,N)$.
Then $\pi$ has no proper invariant subgrids.
\end{lemma}
\begin{proof}
Suppose for contradiction that $\pi$ has a proper invariant subgrid $(R,C)$.
Without loss of generality we may suppose
$R=\{0,\ldots,m-1\}$ and $C=\{0,\ldots,n-1\}$.
Let $\pi'$ be the restriction of $\pi$ to $R\times C$.
Then $\pi'$ is realizable with $r_0,\ldots,r_{m-1}$
and $s_0,\ldots,s_{n-1}$.
Let $\pi''\le\pi'$ be a partition of $R\times C$ that is minimally realizable with $r_0,\ldots,r_{m-1}$
and $s_0,\ldots,s_{n-1}$.
Let $\Gamma_{\pi''}$ be the corresponding $\ULIE_K(m,n)$--group with its canonical generators
$a_0,\ldots,a_{m-1}$ and $b_0,\ldots,b_{n-1}$.
Then there is a group homomorphism $\psi:\Gamma_{\pi''}\to G$ so that $\psi(a_i)=c_i$ and $\psi(b_j)=d_j$.
Since $G$ is torsion--free, by Proposition~\ref{prop:Ntor}, $N_\tor(\Gamma_{\pi''})\subseteq\ker\psi$.
By hypothesis, the Invertibles Conjecture holds over $K$ for rank pair $(m,n)$.
Thus, by Theorem~\ref{thm:InvTorULIE}, the mapping $\psi$
must identify either two distinct $a_i$ and $a_{i'}$ with each other or two distinct $b_j$ and $b_{j'}$
with each other, which contradicts that $c_0,\ldots,c_{m-1}$ are distinct and $d_0,\ldots,d_{n-1}$
are distinct.
\end{proof}

\begin{defi}\label{def:ULIE2}
Let $K$ be a division ring and let $m,n\ge2$ be integers.
Let $\ULIE^{(2)}_K(m,n)$ be the set of all groups $\Gamma_\pi$ as in Definition~\ref{def:Gammapi}
as $\pi$ runs over all  partitions $\pi$ of the set~\eqref{eq:set2nd} that are nondegenerate,
minimally realizable over $K$ and have no proper invariant subgrids.
We will say that an {\em ULIE$^{(2)}_K$ group} is one that belongs to the set $\bigcup_{m,n\ge2}\ULIE^{(2)}_K(m,n)$.
Similarly, we let $\ULIE^{(2)}(m,n)$ denote the set of all groups $\Gamma_\pi$ as $\pi$ runs over all partitions
$\pi$ of~\eqref{eq:set2nd} that are nondegenerate, minimally realizable (over some division ring) and
have no proper invariant subgrids,
and say that an {\em ULIE$^{(2)}$ group} is one that belongs to the set 
$\bigcup_{m,n\ge2}\ULIE^{(2)}(m,n)$.
\end{defi}

Now Lemma~\ref{lem:simple} gives the following variant of Theorem~\ref{thm:InvTorULIE}:
\begin{thm}\label{thm:InvTorULIE2}
Let $K$ be any nonzero field or division ring and let $M,N\ge2$ be integers.
Suppose that for every group
\begin{equation}\label{eq:ULIE2cup}
\Gamma\in\bigcup_{\substack{2\le m\le M \\2\le n\le N}}\ULIE^{(2)}_K(m,n)
\end{equation}
with its canonical generators $1=a_0,a_1,\ldots,a_{m-1}$
and $1=b_0,b_1,\ldots,b_{n-1}$,
letting $\phi:\Gamma\to\Gamma/N_\tor(\Gamma)$ be the quotient map,
we have
$\phi(a_i)=\phi(a_{i'})$ for some $0\le i<i'\le m-1$ or $\phi(b_j)=\phi(b_{j'})$
for some $0\le j<j'\le n-1$.
Then the Invertibles Conjecture holds over $K$
for rank pair $(M,N)$.
\end{thm}
\begin{proof}
Arguing first by induction on $M+N$, we may assume the Invertibles Conjecture holds over $K$ for all
rank pairs $(m,n)$ appearing in~\eqref{eq:ULIE2cup} provided $(m,n)\ne(M,N)$.
Now we proceed as in the proof of (ii)$\implies$(i) in Theorem~\ref{thm:InvTorULIE},
but using $(M,N)$ instead of $(m,n)$, except we note
that the equality $ab=1$ in $K[\Gamma_\pi]$ implies, grace of Lemma~\ref{lem:simple},
that $\pi$ has no proper invariant subgrids.
\end{proof}

From the calculations reported in section~\ref{sec:calcs}, we now have the following:
\begin{prop}\label{prop:ICcases}
Let $m$ and $n$ be odd integers with either (a) $\min(m,n)=3$ and $\max(m,n)\le11$ or (b) $m=n=5$.
Then the Invertibles Conjecture holds over the field $\Fb_2$ of two elements for rank pair $(m,n)$.
\end{prop}
\begin{proof}
All of the ULIE$_{\Fb_2}(m,n)$ groups $\Gamma_\pi$ in the cases~(a)
and all the ULIE$^{(-)}(5,5)$ groups in case~(b) 
have been computed and they are described in section~\ref{sec:calcs}.
In all cases, one easily verifies that the quotient groups $\Gamma_\pi/N_\tor(\Gamma_\pi)$ are abelian.
As described in Remark~\ref{rem:ICabelian}, it follows that the quotient map $\phi:\Gamma\to\Gamma_\pi/N_\tor(\Gamma_\pi)$
fails to be one--to--one on $\{1,a_1,\ldots,a_{m-1}\}$.
For the ULIE$_{\Fb_2}(5,5)$--groups that are not in ULIE$^{(-)}_{\Fb_2}(5,5)$, this lack of injectivity of $\phi$
on $\{1,a_1,\ldots,a_{m-1}\}$ is verified directly in Subsection~\ref{subsec:5x5a=b}.
Now Theorem~\ref{thm:InvTorULIE} applies.
\end{proof}

\section{A procedure for computations of ULIE groups over $\Fb_2$}
\label{sec:alg}

Our aim is, for certain $m$ and $n$, to compute the ULIE groups of all nondegenerate
pairings of the $(2m+1)\times(2n+1)$ grid
\begin{equation}\label{eq:E}
E=(\{0,1,\ldots,2m\}\times\{0,1,\ldots,2n\})\backslash\{(0,0)\}.
\end{equation}
Our strategy is to use a C++ code to enumerate the pairings and for each pairing $\pi$
to call GAP~\cite{GAP} to compute the finitely presented
group $\Gamma_\pi$ as in equation~\eqref{eq:Gammapi} and to determine whether the pairing $\pi$ degenerates.
One difficulty with this strategy is that the Knuth--Bendix procedure employed by GAP to try to decide when
a given word is equivalent to the identity in a finitely presented group
may not terminate in
a reasonable amount of time and is not even guaranteed to ever terminate.
Any groups for which this approach fails to determine degeneracy and/or to decide the $ba=1$ question,
must be handled separately. We handle this by having a timeout routine inside the C++ code that aborts the GAP computation after a couple of seconds and marks this pairing matrix for manual analysis.
We also separate the construction of valid pairings from the degeneracy analysis in GAP for performance reasons: the C++ code can enumerate pairings much more efficiently.

In order to limit the number of costly calls into GAP, we have considered a natural equivalence relation on pairings, and
run GAP on only one pairing from each equivalence class. All relevant pairings are constructed in a
recursive procedure by filling in entries one by one in a pairing matrix. This procedure forms a tree with pairings being the leaves of the tree. Branches that can not produce valid pairing matrices (because they are handled in a different branch due to the equivalence relation, or because they will never produce a valid pairing matrix later) are skipped as soon as it is known.

For permutations $\sigma$ and $\tau$ of $\{0,1,\ldots,2m\}$ and $\{0,1,\ldots,2n\}$, respectively, both of which fix $0$,
let $\pit$ be the image of $\pi$ under the permutation $\sigma\times\tau$,
and let $\at$ and $\bt$ be the elements of $\Gamma_\pit$ that are analogous to $a$ and $b$.
Then $\Gamma_\pi$ is degenerate if and only if $\Gamma_\pit$ is degenerate, and
$ba=1$ in $\Fb_2[\Gamma_\pi]$ if and only if $\bt\at=1$ in $\Fb_2[\Gamma_\pit]$.
Thus, the equivalence relation on the set of pairings that we use is the one induced  by this natural action
of $S_{2m}\times S_{2n}$.

We will encode pairings as $(2m+1)\times(2n+1)$ matrices, as described below.
Keeping with the convention that the elements in the support of $a$ are numbered starting with $a_0$,
and similarly for $b$, we will index the entries of a $(2m+1)\times(2n+1)$ matrix $A$ as $a_{ij}$ with $0\le i\le2m$
and $0\le j\le 2n$.

\begin{defi}
A $(2m+1)\times(2n+1)$ {\em pairing matrix} is a $(2m+1)\times(2n+1)$ matrix $A$ with
\begin{itemize}
\item[$\bullet$] $-1$ in the $(0,0)$ entry
\item[$\bullet$] all other entries of $A$ coming from the set $\{1,2,\ldots,2mn+m+n\}$
\item[$\bullet$] each element of $\{1,2,\ldots,2mn+m+n\}$ appearing in exactly two entries of $A$
\item[$\bullet$] no row or column of $A$ containing a repeated value.
\end{itemize}
\end{defi}

\begin{defi}
A $(2m+1)\times(2n+1)$ {\em partial pairing matrix} is a $(2m+1)\times(2n+1)$ matrix $A$ with
\begin{itemize}
\item[$\bullet$] $-1$ in the $(0,0)$ entry
\item[$\bullet$] all other entries of $A$ coming from the set $\{0,1,2,\ldots,2mn+m+n\}$
\item[$\bullet$] no element of $\{1,2,\ldots,2mn+m+n\}$ appearing in more than two entries of $A$
\item[$\bullet$] no row or column of $A$ containing a repeated nonzero value.
\end{itemize}
\end{defi}

We think about traversing an $(2m+1)\times(2n+1)$ matrix, starting at the $(0,1)$ entry, by proceeding towards the right
until we reach the $(0,2n)$ entry, then taking the next row, starting at the $(1,0)$ entry, moving
from left to right until the $(1,2n)$, and so on, row after row, until we reach the $(2m,2n)$
entry.
In fact, we will eventually construct pairing matrices by filling in the entries in this order.
We say that a partial pairing matrix $A$ is {\em stacked} if, when we traverse the matrix as described above, if we once encounter
a zero, then all the following entries are zero.
This is expressed more precisely below.

\begin{defi}
A partial pairing matrix $A=(a_{ij})_{0\le i\le2m,\,0\le j\le2n}$ is {\em stacked} if $a_{ij}=0$ implies
$a_{i\ell}=0$ for all $\ell>j$ and $a_{k\ell}=0$ for all $k>i$ and all $0\le\ell\le2n$.
\end{defi}

\begin{defi}
We say that a partial pairing matrix $A=(a_{ij})_{0\le i\le2m,\,0\le j\le2n}$ is {\em consecutively numbered} if for the list
\begin{equation}\label{eq:aseq}
\begin{aligned}
a_{0,1},a_{0,2},&\ldots,a_{0,2n},\,
a_{1,0},a_{1,1},\ldots,a_{1,2n}, \\
a_{2,0},&a_{2,1},\ldots,a_{2,2n},\,\ldots, \\
&a_{2m,0},a_{2m,1},\ldots,a_{2m,2n},
\end{aligned}
\end{equation}
when relabeled as $b_1,b_2,\ldots,b_{4mn+2m+2n}$,
the set $\{b_1,\ldots,b_q\}$ of every initial segment (ignoring repeats and rearranging)
is equal to a set of the form $\{0,1,\ldots,r\}$ or $\{1,2,\ldots,r\}$ for some non-negative integer $r$.
\end{defi}

A partial pairing matrix $A=(a_{ij})_{0\le i\le2m,\,0\le j\le2n}$ yields an equivalence relation on the set $E$ as
in~\eqref{eq:E} given by
\[
(i,j)\sim(k,\ell)\quad\Longleftrightarrow\quad a_{ij}=a_{k\ell}\ne0,
\]
and the equivalence classes are all singletons or pairs.
They are all pairs if and only if $A$ is a pairing matrix,
and we will call such an equivalence relation a {\em restricted pairing} of $E$.
Note that the restricted pairings are precisely the partitions of the set $E$
into pairs and singletons so that no entry is paired
with another in the same row or column.

Given a partial pairing matrix $A$, there is a unique partial pairing matrix $B$ that yields the same
equivalence relation as $A$ and such that $B$ is consecutively numbered.
We call $B$ the {\em consecutive renumbering} of $A$. To compute the consecutive renumbering $B$
one has to traverse through all the entries of $A$ row-wise and replace the numbers according to a
map that gets created during the traversal. If an entry for a certain number already exists in the mapping,
it is used, otherwise the smallest unused positive number will be taken.

We will consider the action $\alpha$ of the product $S_{2m}\times S_{2n}$ of symmetric groups
on the set of all $(2m+1)\times(2n+1)$ partial pairing matrices, by permutations of rows numbered $1,2,\ldots,2m$
and columns numbered $1,2,\ldots,2n$.
Restricting this action to the set of all pairing matrices, 
it descends to an action $\beta$ of $S_{2m}\times S_{2n}$ on the set of all restricted pairings of the set $E$ in~\eqref{eq:E}.

\smallskip
We will now describe an algorithm that will generate a set $R_{m,n}$ consisting of
one consecutively numbered pairing matrix for each orbit of $\beta$
(i.e., whose restricted pairing belongs to the given orbit of $\beta$).
We will begin with the matrix $A_0$ which has zero in every entry except for a $-1$ in the $(0,0)$ entry,
and we proceed via a branching process, filling in the entries of the matrix, one after the other, so that we
generate a tree, $T_{m,n}$, of stacked, consecutively numbered $(2m+1)\times(2n+1)$
pairing matrices rooted at $A_0$, and so that all the branches
flowing from a given matrix $A$ are obtained from $A$ be replacing a single zero entry by a nonzero entry.
If the first zero entry of $A$ is in the $(i,j)$ entry, then the branches at this node are determined
by the set $V$ of possible nonzero values to place in the $(i,j)$ position.
Clearly, $V$ will be a subset of the union $H\cup N$, where $H$ is the set of strictly positive integers
that appear in exactly one entry of $A$ (the so-called half--pairs of $A$) and do not appear in the $i$th row
or $j$th column of $A$, and $N$ is either the singleton set $\{\ell+1\}$, where $\ell$ is the largest value 
that appears as an entry of $A$ or, if $\ell=2mn+m+n$, then $N$ is the empty set.

In order to specify $V$, consider the total ordering $<$ on the set of all $(2m+1)\times(2n+1)$ partial pairing matrices, which is defined as the lexicographic ordering on the sequences~\eqref{eq:aseq} associated to partial pairing matrices $A=(a_{ij})_{1\le i\le m,\,1\le j\le n}$.
Then $V=H'\cup N$, where $H'$ is the set of all $h\in H$ such that, if $A'$ is matrix obtained from $A$
by setting the first zero entry (i.e., the $(i,j)$ entry) to be $h$,
then whenever
\begin{equation}\label{eq:sigrho}
(\sigma,\rho)\in S_{2m}\times S_{2n}
\end{equation}
is such that $\alpha((\sigma,\rho),A')$ is a stacked
partial pairing matrix and $B$ is the consecutive renumbering of it, we do not have $B<A'$.

In other words, we branch off at $(i,j)$ with entry $v$ only if there is no permutation
with a smaller consecutive numbering than the current partial pairing matrix. Otherwise
this case is already handled in a different branch of the tree and we would
generate duplicate results.

Using the above algorithm, we will construct a tree, call it $\Tt_{m,n}$, that may have leaves that are not pairing matrices,
i.e., have zeros in them.
For example, in the case $m=n=1$, $\left(\begin{smallmatrix}-1&1&2\\1&2&3\\3&4&0\end{smallmatrix}\right)$
is such a leaf.
We will call such a leaf {\em stunted}.
We obtain the tree $T_{m,n}$ by pruning $\Tt_{m,n}$, lopping off all stunted leaves and all branches that end in only
stunted leaves.
Finally, the set $R_{m,n}$ consists of the matrices found at the leaves of $T_{m,n}$.

Note that the algorithm described here has several good properties:
First, no duplicate pairings will be created. This simplifies the analysis and speeds up the computation as described before.
Second, there is no global state to be kept around. All the operations are done with the current partial pairing matrix. We can determine if a branch has been done already using the ordering and without keeping track what we already touched.
Third, the branching allows us to do parallel computations without any communication between processes. In a pre-process we can run the algorithm where we stop the recursion after the partial pairing matrix has $k$ entries for some $k<(2m+1)(2n+1)$. If we write out these partial pairing matrices, we can start independent jobs for each of those matrices in parallel.

The code that implements this algorithm and the raw output of it are included in the directory
\verb=ULIE.computations= that was submited to (and is retrievable from)
the arXiv with this article as part of the source code.

\newpage

\section{Lists of ULIE groups}
\label{sec:calcs}

Here we list all the ULIE groups of sizes $3\times k$ for $k$ odd, $3\le k\le9$ and the most interesting ones
of sizes $3\times 11$ and $5\times 5$.
The conclusions drawn from these regarding Kaplansky's Direct Finiteness Conjecture and the Invertibles Conjecture
have already been mentioned in Propositions~\ref{prop:DFCcases} and~\ref{prop:ICcases}.

\begin{notation}
\begin{enumerate}[(i)]
\item
$\bZ_n=\bZ/n\bZ$ is the cyclic group of of order $n$.
\item
In presentations of groups, the identity element will be written as $1$.
\item
$\Dih_n\cong\langle x,x\mid x^n=y^2=1,\,yxy=x^{n-1}\rangle$ is the dihedral group of order $2n$.
\item
$H\rtimes_{\alpha}G$ denotes the semidirect product of group $H$ by group $G$ and
the action $\alpha:G\to\operatorname{Aut}(H)$, typically written $G\ni g\mapsto\alpha_g\in\operatorname{Aut}(H)$.
If $G=\bZ$ or $G=\bZ_n$, then by abuse of notation we will write simply $\alpha$ instead of $\alpha_1$.
\end{enumerate}
\end{notation}

The following remark is easy to prove.
\begin{remark}\label{rem:Gammac}
Let $H$ and $G$ be groups
and suppose
\begin{equation}\label{eq:GammaCentExt}
1\to H\to \Gamma\overset{\pi}{\to} G\to1
\end{equation}
is a central extension
and choose any $\phi:G\to\Gamma$ so that $\pi\circ\phi=\id_G$ and $\phi(1)=1$.
Then, the cocycle $\Gamma$, $c:G\times G\to H$ defined by
\begin{equation}\label{eq:cphi}
c(g_1,g_2)=\phi(g_1)\phi(g_2)\phi(g_1g_2)^{-1}
\end{equation}
satisfies the identity
\begin{equation}\label{eq:cidentity}
c(g_1g_2,g_3)c(g_1,g_2)=c(g_1,g_2g_3)c(g_2,g_3).
\end{equation}
Conversely, given any abelian group $H$ and group $G$ and any map $c$ satisfying $c(1,g)=c(g,1)=1$ for all $g\in G$
and the identity~\eqref{eq:cidentity},
there is a uniquely determined
central extension
$\Gamma$ as in~\eqref{eq:GammaCentExt}
with a right inverse $\phi$ for the quotient map $\pi$ so that~\eqref{eq:cphi} holds.
Indeed, we may define $\Gamma$ to be the set $H\times G$ endowed with the multiplication $(h_1,g_1)\cdot(h_2,g_2)=(h_1h_2c(g_1,g_2),g_1g_2)$
and then define $\phi(g)=(1,g)$.
Furthermore, $H$ is equal to the center of $\Gamma$ if and only if for every nontrivial element $g$ of the center of $G$,
there is $g'\in G$ such that $c(g,g')\ne c(g',g)$.
\end{remark}

\begin{remark}
We describe some finitely presented infinite nonabelian groups $\Gamma$ by
describing the centers $H=Z(\Gamma)$ and quotients $G=\Gamma/Z(\Gamma)$
of these groups.
Whenever we make such an assertion, we have verified it by performing the steps indicated in the above remark.
This procedure is carried out in more detail for Group~\ref{gp:3x11.1} below, but in other cases the details are omitted.
\end{remark}

\subsection{Case 3 by 3. All groups are  Abelian.}{$\;$}\newline

There are 3 equivalence classes of nondegenerate pairing matrices, representatives for which are listed below.

\[
\begin{array}{c|c|l}
\text{matrix}&\text{group}&\text{distinguished generators} \\ \hline
\left(\begin{smallmatrix} \times& 1& 2 \\ 1& 3& 4 \\ 2& 4& 3\end{smallmatrix}\right)
&
\langle x,y:y^2=1,\,xy=yx\rangle\cong\bZ\times\bZ_2
&
(a_1,a_2)=(b_1,b_2)=(x,xy) \\[1.5ex] \hline
\left(\begin{smallmatrix} \times& 1& 2 \\ 1& 3& 4 \\ 4& 2& 3\end{smallmatrix}\right)
&
\langle x:x^4=1 \rangle\cong\bZ_4
&
(a_1,a_2)=(x^2,x),\;(b_1,b_2)=(x^2,x^3) \\[1.5ex] \hline
\left(\begin{smallmatrix} \times& 1& 2 \\ 3& 4& 1 \\ 4& 2& 3\end{smallmatrix}\right)
&
\langle x: x^5=1\rangle\cong\bZ_5
&
(a_1,a_2)=(x^3,x^2),\;(b_1,b_2)=(x,x^4).
\end{array}
\]

\subsection{Case 3 by 5. All groups are finite}{$\;$}\newline

There are 9 equivalence classes of nondegenerate pairing matrices, representatives for which are listed below.

\[
\begin{array}{c|c|l}
\text{matrix}&\text{group}&\text{distinguished generators} \\ \hline
\left(\begin{smallmatrix} \times& 1& 2& 3&  4 \\ 1& 2& 5& 6& 7 \\  6& 4& 7& 5& 3 \end{smallmatrix}\right)
&
\langle x:x^8=1 \rangle\cong\bZ_8
&
\begin{array}{l}
(a_1,a_2)=(x,x^6),\\
(b_1,b_2,b_3,b_4)=(x,x^2,x^5,x^7)
\end{array} \\[1ex] \hline
\left(\begin{smallmatrix} \times& 1& 2& 3&  4 \\ 1& 5& 3& 6& 7 \\ 4& 7& 5& 2& 6  \end{smallmatrix}\right)
&
\langle x: x^7=1\rangle\cong\bZ_7
&
\begin{array}{l}
(a_1,a_2)=(x,x^6), \\
(b_1,\ldots,b_4)=(x,x^3,x^4,x^6) \\[1ex] \hline
\end{array} \\[1ex] \hline
\left(\begin{smallmatrix} \times& 1& 2& 3&  4 \\ 1& 5& 3& 6& 7 \\ 6& 2& 4& 7& 5  \end{smallmatrix}\right)
&
\langle x : x^9=1 \rangle\cong\bZ_9
&
\begin{array}{l}
(a_1,a_2)=(x^3,x^7), \\
(b_1,\ldots,b_4)=(x^3,x,x^4,x^8) \\[1ex] \hline
\end{array} \\[1ex] \hline
\left(\begin{smallmatrix} \times& 1& 2& 3&  4 \\ 1& 5& 3& 6& 7 \\ 7& 2& 5& 4& 6  \end{smallmatrix}\right)
&
\langle x : x^7=1 \rangle\cong\bZ_7
&
\begin{array}{l}
(a_1,a_2)=(x,x^4), \\
(b_1,\ldots,b_4)=(x,x^5,x^6,x^3) \\[1ex] \hline
\end{array} \\[1ex] \hline
\left(\begin{smallmatrix} \times& 1& 2& 3&  4 \\ 5& 2& 1& 6& 7 \\ 6& 5& 7& 4& 3  \end{smallmatrix}\right)
&
\begin{array}{c}
\langle x,y: x^4=y^2=1,\,yxy=x^{-1} \rangle \\
\cong\Dih_4 \\
\text{the dihedral group of order }8
\end{array}
&
\begin{array}{l}
(a_1,a_2)=(y,xy), \\
(b_1,\ldots,b_4)=(x,x^3y,x^3,x^2y)
\end{array} \\[1ex] \hline
\left(\begin{smallmatrix} \times& 1& 2& 3&  4 \\ 5& 2& 3& 6& 7 \\ 6& 4& 7& 1& 5  \end{smallmatrix}\right)
&
\langle x: x^{10}=1 \rangle\cong\bZ_{10}
&
\begin{array}{l}
(a_1,a_2)=(x,x^8), \\
(b_1,\ldots,b_4)=(x^5,x^6,x^7,x^3)
\end{array} \\[1ex] \hline
\left(\begin{smallmatrix} \times& 1& 2& 3&  4 \\ 5& 2& 3& 6& 7 \\ 6& 5& 7& 4& 1  \end{smallmatrix}\right)
&
\begin{array}{c}
\langle x,y: x^4=y^4=(xy)^4=1, \\
\quad yxy=x,\,xyx=y \rangle\cong Q_8, \\
\text{the quaternion group}
\end{array}
&
\begin{array}{l}
(a_1,a_2)=(x,x^3y), \\
(b_1,\ldots,b_4)=(y,xy,x^2y,x^3)
\end{array} \\[1ex] \hline
\left(\begin{smallmatrix} \times& 1& 2& 3&  4 \\ 5& 2& 3& 6& 7 \\ 7& 6& 5& 4& 1  \end{smallmatrix}\right)
&
\langle x : x^8=1\rangle\cong\bZ_8
&
\begin{array}{l}
(a_1,a_2)=(x,x^3), \\
(b_1,\ldots,b_4)=(x^5,x^6,x^7,x^2)
\end{array} \\[1ex] \hline
\left(\begin{smallmatrix} \times& 1& 2& 3&  4 \\ 5& 2& 6& 4& 7 \\ 6& 7& 1& 5& 3  \end{smallmatrix}\right)
&
\langle x: x^8=1 \rangle\cong\bZ_8 
&
\begin{array}{l}
(a_1,a_2)=(x,x^7), \\
(b_1,\ldots,b_4)=(x^5,x^6,x^2,x^3)
\end{array} \\[1ex] \hline
\end{array}
\]

\subsection{Case 3 by 7.}{$\;$}\newline

There are 18 equivalence classes of nondegenerate pairing matrices, representatives for which are listed below,
where the distinguished generators are $(a_1,a_2)$ and $(b_1,\ldots,b_6)$.

\subsubsection{Infinite groups.}{$\;$}\newline

\[
\begin{array}{c|c|l}
\text{matrix}&\text{group}&\text{distinguished generators} \\ \hline
\left(\begin{smallmatrix}
\times& 1& 2& 3& 4& 5& 6  \\  
1& 2& 7& 4& 8& 9& 10 \\
3& 5& 8& 6& 7& 10& 9
\end{smallmatrix}\right)
&
\begin{array}{c}
\begin{aligned}
\langle x,y,z:y^2=z^2=1,\,xy=yx,& \\
yz=zy,\,zxz=xy&\rangle
 \end{aligned} \\
\cong(\bZ\times\bZ_2)\rtimes_\alpha\bZ_2 \\
\begin{aligned}
\text{where }&\alpha:(1,0)\mapsto(1,1), \\
&\alpha:(0,1)\mapsto(0,1)
\end{aligned}
\end{array}
&
\begin{array}{l}
(x,zx), \\
(x,x^2,zx,zx^2y,zx^2,x^2y)
\end{array} \\ \hline

\left(\begin{smallmatrix}
\times& 1& 2& 3& 4& 5& 6  \\  
1& 2& 7& 4& 8& 9& 10 \\
3& 5& 10& 6& 9& 8& 7
\end{smallmatrix}\right)
&
\begin{array}{c}
\begin{aligned}
\langle x,y,z:x^2=y^2=1,\,xy=yx,& \\
zxz^{-1}=y,\,zyz^{-1}=xy&\rangle
\end{aligned} \\
\cong(\bZ_2\times\bZ_2)\rtimes_\alpha\bZ  \\
\begin{aligned}
\text{where }&\alpha:(1,0)\mapsto(0,1), \\
&\alpha:(0,1)\mapsto(1,1)
\end{aligned}
\end{array}
&
\begin{array}{l}
(yz,z), \\
(yz,xz^2,z,yz^2,xyz^2,z^2)
\end{array}
\end{array}
\]

\subsubsection{Finite groups.}{$\;$}\newline

\[
\begin{array}{c|c|l}
\text{matrix}&\text{group}&\text{distinguished generators} \\ \hline

\left(\begin{smallmatrix}\times& 1& 2& 3& 4& 5& 6 \\ 1& 2& 7& 4& 8& 9& 10 \\ 5& 9& 3& 6& 10& 8& 7 \end{smallmatrix}\right)
&
\langle x:x^{11}=1\rangle\cong\bZ_{11}
&
\begin{array}{l}
(x^3,x) \\
(x^3,x^6,x^7,x^{10},x,x^8)
\end{array} \\ \hline

\left(\begin{smallmatrix}\times& 1& 2& 3& 4& 5& 6 \\ 1& 2& 7& 4& 8& 9& 10 \\ 5& 10& 8& 6& 9& 7& 3 \end{smallmatrix}\right)
&
\begin{aligned}
\langle x,y,z:x^2=y^2=z^3=1,& \\
xy=yx,\,zxz^{-1}=y,\,zyz^{-1}=xy&\rangle \\
\cong(\bZ_2\times\bZ_2)\rtimes_\alpha\bZ_3\cong&A_4
\end{aligned}
&
\begin{array}{l}
(z,x), \\
(z,z^2,y,xyz,x,xy)
\end{array} \\ \hline

\left(\begin{smallmatrix}\times& 1& 2& 3& 4& 5& 6 \\ 1& 2& 7& 4& 8& 9& 10 \\ 8& 5& 9& 6& 10& 3& 7 \end{smallmatrix}\right)
&
\langle x:x^{14}=1\rangle\cong\bZ_{14}
&
\begin{array}{l}
(x,x^{11}),\\
(x,x^2,x^9,x^{10},x^{12},x^6)
\end{array} \\ \hline

\left(\begin{smallmatrix}\times& 1& 2& 3& 4& 5& 6 \\ 1& 2& 7& 4& 8& 9& 10 \\ 9& 6& 10& 7& 5& 3& 8 \end{smallmatrix}\right)
&
\langle x:x^{11}=1\rangle\cong\bZ_{11}
&
\begin{array}{l}
(x,x^5), \\
(x,x^2,x^9,x^{10},x^4,x^6)
\end{array} \\ \hline

\left(\begin{smallmatrix}\times& 1& 2& 3& 4& 5& 6 \\ 1& 7& 3& 2& 8& 9& 10 \\ 4& 9& 5& 6& 10& 7& 8 \end{smallmatrix}\right)
&
\begin{aligned}
\langle x,y:x^5=y^2=1,\,yxy=x^{-1}&\rangle \\
\cong&\Dih_5
\end{aligned}
&
\begin{array}{l}
(y,x), \\
(y,x^3,x^2y,x,x^4,x^3y)
\end{array} \\ \hline

\left(\begin{smallmatrix}\times& 1& 2& 3& 4& 5& 6 \\ 1& 7& 3& 2& 8& 9& 10 \\ 4& 9& 8& 10& 7& 6& 5 \end{smallmatrix}\right)
&
\langle x,y:x^5=y^2=1,\,yxy=x^{-1}\rangle\cong\Dih_5
&
\begin{array}{l}
(y,xy),\\
(y,x^2y,x^3,xy,x^4y,x^2)
\end{array} \\ \hline

\left(\begin{smallmatrix}\times& 1& 2& 3& 4& 5& 6 \\ 1& 7& 3& 4& 8& 9& 10 \\ 5& 9& 6& 10& 2& 8& 7 \end{smallmatrix}\right)
&
\langle x:x^{13}=1\rangle\cong\bZ_{13}
&
\begin{array}{l}
(x,x^{11}),\\
(x,x^6,x^7,x^8,x^{11},x^4)
\end{array} \\ \hline

\left(\begin{smallmatrix}\times& 1& 2& 3& 4& 5& 6 \\ 1& 7& 3& 4& 8& 9& 10 \\ 9& 2& 6& 10& 5& 7& 8 \end{smallmatrix}\right)
&
\langle x:x^{11}=1\rangle\cong\bZ_{11}
&
\begin{array}{l}
(x^2,x^3), \\
(x^2,x^5,x^7,x^9,x,x^8)
\end{array} \\ \hline 

\left(\begin{smallmatrix}\times& 1& 2& 3& 4& 5& 6 \\ 1& 7& 3& 8& 5& 9& 10 \\ 6& 10& 7& 2& 8& 4& 9 \end{smallmatrix}\right)
&
\langle x:x^{10}=1\rangle\cong\bZ_{10}
&
\begin{array}{l}
(x,x^9), \\
(x,x^3,x^4,x^6,x^7,x^9)
\end{array} \\ \hline

\left(\begin{smallmatrix}\times& 1& 2& 3& 4& 5& 6 \\ 1& 7& 3& 8& 5& 9& 10 \\ 8& 2& 9& 4& 6& 10& 7 \end{smallmatrix}\right)
&
\langle x:x^{12}=1\rangle\cong\bZ_{12}
&
\begin{array}{l}
(x^4,x^9), \\
(x^4,x,x^5,x^2,x^6,x^{11})
\end{array} \\ \hline

\left(\begin{smallmatrix}\times& 1& 2& 3& 4& 5& 6 \\ 1& 7& 3& 8& 5& 9& 10 \\ 10& 2& 7& 4& 8& 6& 9 \end{smallmatrix}\right)
&
\langle x:x^{10}=1\rangle\cong\bZ_{10}
&
\begin{array}{l}
(x^2,x), \\
(x^2,x^3,x^5,x^6,x^8,x^9)
\end{array} \\ \hline

\end{array}
\]

\[
\begin{array}{c|c|l}
\hline

\left(\begin{smallmatrix}\times& 1& 2& 3& 4& 5& 6 \\ 7& 2& 3& 4& 8& 9& 10 \\ 8& 5& 9& 6& 10& 7& 1 \end{smallmatrix}\right)
&
\langle x:x^{16}=1\rangle\cong\bZ_{16}
&
\begin{array}{l}
(x^7,x), \\
(x^5,x^{12},x^3,x^{10},x^6,x^4)
\end{array} \\ \hline

\left(\begin{smallmatrix}\times& 1& 2& 3& 4& 5& 6 \\ 7& 2& 3& 4& 8& 9& 10 \\ 9& 6& 10& 7& 1& 8& 5 \end{smallmatrix}\right)
&
\langle x:x^{13}=1\rangle\cong\bZ_{13}
&
\begin{array}{l}
(x^{10},x^9), \\
(x^7,x^4,x,x^{11},x^{12},x^3)
\end{array} \\ \hline 

\left(\begin{smallmatrix}\times& 1& 2& 3& 4& 5& 6 \\ 7& 2& 3& 8& 5& 9& 10 \\ 9& 6& 10& 4& 7& 1& 8 \end{smallmatrix}\right)
&
\langle x:x^{12}=1\rangle\cong\bZ_{12}
&
\begin{array}{l}
(x^{10},x^9), \\
(x^8,x^6,x^4,x,x^{11},x^5)
\end{array} \\ \hline

\left(\begin{smallmatrix}\times& 1& 2& 3& 4& 5& 6 \\ 7& 2& 3& 8& 5& 9& 10 \\ 10& 9& 6& 4& 1& 7& 8 \end{smallmatrix}\right)
&
\begin{array}{c}
\langle x,y:x^2=y^2,\,x^4=1,\,(xy)^4=y^2\rangle \\
\cong Q_{16}, \\
\text{the generalized quaternion group}
\end{array}
&
\begin{array}{l}
(x,y), \\
(yxyx,yxy,xyxy, \\
\qquad\qquad xyx,yx^3,xy^3)
\end{array} \\ \hline 

\left(\begin{smallmatrix}\times& 1& 2& 3& 4& 5& 6 \\ 7& 2& 8& 4& 9& 6& 10 \\ 8& 9& 1& 10& 3& 7& 5 \end{smallmatrix}\right)
&
\langle x:x^{11}=1\rangle\cong\bZ_{11}
&
\begin{array}{l}
(x,x^{10}), \\
(x^8,x^9,x^5,x^6,x^2,x^3)
\end{array} \\ \hline
\end{array}
\]

\subsection{Case 3 by 9.}{$\;$}\newline

There are 24 equivalence classes of nondegenerate pairing matrices, representatives for which are listed below,
where the distinguished generators are $(a_1,a_2)$ and $(b_1,\ldots,b_8)$.

\subsubsection{Infinite groups.}{$\;$}\newline

\[
\begin{array}{c|c|l}
\text{matrix}&\text{group}&\text{distinguished generators} \\ \hline
\left(\begin{smallmatrix}
\times& 1& 2& 3& 4& 5& 6& 7& 8 \\
1& 2& 3& 9& 5& 10& 11& 12& 13  \\ 
6& 11& 7& 12& 8& 13& 4& 10& 9
\end{smallmatrix}\right)
&
\begin{array}{c}
\langle x,y:y^4=1,\,xy=yx \rangle \\
\cong\bZ\times\bZ_4
\end{array}
&
\begin{array}{l}
(x,xy), \\
(x,x^2,x^3,x^2y^2, \\
\;x^3y^2,xy,x^3y,x^3y^3)
\end{array} \\ \hline

\left(\begin{smallmatrix}
\times& 1& 2& 3& 4& 5& 6& 7& 8 \\  
1& 9& 3& 4& 10& 6& 11& 12& 13 \\
2& 5& 9& 7& 13& 8& 12& 10& 11
\end{smallmatrix}\right)
&
\begin{array}{c}
\Gamma=\langle x,y:x^2=y^2,\\
\qquad\qquad(xy)^2=(yx)^2\rangle \\
\text{has center }Z(\Gamma)\cong\bZ\times\bZ_2 \\
\text{generated by }\{x^2,(xy)^2x^{-4}\} \\
\text{with }\Gamma/Z(\Gamma)\cong\bZ_2\times\bZ_2
\end{array}
&
\begin{array}{l}
(x,y), \\
(x,y,xy,y^3, \\
\quad yx,xyx,yxy,x^3)
\end{array} \\ \hline
\end{array}
\]

\subsubsection{Finite  groups.}

\[
\begin{array}{c|c|l}
\text{matrix}&\text{group}&\text{generators} \\ \hline

\left(\begin{smallmatrix}
\times& 1& 2& 3& 4& 5& 6& 7& 8 \\
1& 2& 3& 9& 5& 10& 11& 12& 13  \\ 
6& 11& 7& 12& 8& 13& 10& 4& 9
\end{smallmatrix}\right)
&
\begin{array}{c}
\langle x,y:x^8=y^2=1,\,xy=yx \rangle \\
\cong\bZ_8\times\bZ_2
\end{array}
&
\begin{array}{l}
(x^2y,x) \\
(x^2y,x^4,x^6y,x^6, \\
\;y,x,x^5,x^7)
\end{array} \\ \hline

\left(\begin{smallmatrix}
\times& 1& 2& 3& 4& 5& 6& 7& 8 \\
1& 2& 3& 9& 5& 10& 11& 12& 13  \\ 
11& 6& 12& 7& 13& 8& 4& 10& 9
\end{smallmatrix}\right)
&
\begin{array}{c}
\langle x,y:x^8=y^2=1,\,yxy=x^5 \rangle,\\
\text{the modular group of order }16
\end{array}
&
\begin{array}{l}
(x^6y,x^3y) \\
(x^6y,x^4,x^2y,y, \\
\;x^6,x,x^5,xy)
\end{array} \\ \hline

\left(\begin{smallmatrix}
\times& 1& 2& 3& 4& 5& 6& 7& 8 \\
1& 2& 3& 9& 5& 10& 11& 12& 13  \\ 
11& 6& 12& 7& 13& 8& 10& 4& 9
\end{smallmatrix}\right)
&
\begin{array}{c}
\langle x,y:x^4=y^4=1,\,yxy^{-1}=x^{-1} \rangle,\\
\cong\bZ_4\rtimes_{\alpha}\bZ_4, \\
\text{where }\alpha(1)=3
\end{array}
&
\begin{array}{l}
(x,xy),\\
(x,x^2,x^3,x^3y^2, \\
\;y^2,y,x^2y,xy^3)
\end{array} \\ \hline

\end{array}
\]

\[
\begin{array}{c|c|l}
\hline

\left(\begin{smallmatrix}
\times& 1& 2& 3& 4& 5& 6& 7& 8 \\
1& 2& 3& 9& 5& 10& 11& 12& 13  \\ 
11& 7& 12& 6& 8& 13& 4& 10& 9
\end{smallmatrix}\right)
&
\langle x:x^{16}=1 \rangle\cong\bZ_{16}
&
\begin{array}{l}
(x^4,x) \\
(x^4,x^8,x^{12},x^{14}, \\
\;x^2,x^{13},x^5,x^{15})
\end{array} \\ \hline

\left(\begin{smallmatrix}
\times& 1& 2& 3& 4& 5& 6& 7& 8 \\
1& 2& 3& 9& 5& 10& 11& 12& 13  \\ 
11& 7& 12& 6& 8& 13& 10& 4& 9
\end{smallmatrix}\right)
&
\langle x:x^{16}=1 \rangle\cong\bZ_{16}
&
\begin{array}{l}
(x^{12},x) \\
(x^{12},x^8,x^4,x^{14}, \\
\;x^{10},x^5,x^{13},x^{15})
\end{array} \\ \hline

\left(\begin{smallmatrix}
\times& 1& 2& 3& 4& 5& 6& 7& 8 \\
1& 2& 9& 4& 5& 10& 11& 12& 13  \\ 
11& 7& 12& 8& 13& 3& 10& 6& 9
\end{smallmatrix}\right)
&
\langle x:x^{15}=1 \rangle\cong\bZ_{15}
&
\begin{array}{l}
(x,x^{13}) \\
(x,x^2,x^7,x^8, \\
\;x^9,x^{12},x^{14},x^5)
\end{array} \\ \hline

\left(\begin{smallmatrix}
\times& 1& 2& 3& 4& 5& 6& 7& 8 \\
1& 2& 9& 4& 5& 10& 11& 12& 13  \\ 
11& 7& 12& 9& 8& 13& 10& 3& 6
\end{smallmatrix}\right)
&
\langle x:x^{15}=1 \rangle\cong\bZ_{15}
&
\begin{array}{l}
(x^{12},x^{13}) \\
(x^{12},x^9,x^8,x^5, \\
\;x^2,x,x^{10},x^3)
\end{array} \\ \hline

\left(\begin{smallmatrix}
\times& 1& 2& 3& 4& 5& 6& 7& 8 \\
1& 2& 9& 4& 10& 6& 11& 12& 13  \\ 
10& 7& 12& 9& 5& 8& 13& 11& 3
\end{smallmatrix}\right)
&
\langle x:x^{17}=1 \rangle\cong\bZ_{17}
&
\begin{array}{l}
(x^{12},x^{13}) \\
(x^{12},x^7,x^6,x, \\
\;x^{14},x^9,x^8,x^{10})
\end{array} \\ \hline

\left(\begin{smallmatrix}
\times& 1& 2& 3& 4& 5& 6& 7& 8 \\
1& 9& 3& 4& 5& 10& 11& 12& 13  \\ 
6& 11& 7& 12& 8& 13& 10& 9& 2
\end{smallmatrix}\right)
&
\begin{array}{c}
\langle x,y:x^{10}=y^2=1,\,xy=yx \rangle \\
\cong\bZ_{10}\times\bZ_2
\end{array}
&
\begin{array}{l}
(x,x^9y) \\
(x,x^4,x^5.x^6,x^7, \\
\;x^9y,x^3y,x^5y)
\end{array} \\ \hline

\left(\begin{smallmatrix}
\times& 1& 2& 3& 4& 5& 6& 7& 8 \\
1& 9& 3& 4& 5& 10& 11& 12& 13  \\ 
10& 11& 7& 13& 9& 12& 8& 2& 6
\end{smallmatrix}\right)
&
\langle x,y:x^8=y^2=1,\,yxy=x^3 \rangle
&
\begin{array}{l}
(x,x^2y) \\
(x,x^6y,x^7y,y, \\
\;xy,x^4y,x^4,x^6)
\end{array} \\ \hline

\left(\begin{smallmatrix}
\times& 1& 2& 3& 4& 5& 6& 7& 8 \\
1& 9& 3& 4& 10& 6& 11& 12& 13  \\ 
7& 13& 11& 8& 5& 2& 12& 9& 10
\end{smallmatrix}\right)
&
\begin{array}{c}
\langle x,y:x^5=y^4=1,\,yxy^{-1}=x^{-1} \rangle \\
\cong\bZ_5\rtimes_\alpha\bZ_4\\
\text{where }\alpha(1)=4
\end{array}
&
\begin{array}{l}
(xy^3,y) \\
(xy^3,x^3y,x^3,x^3y^3, \\
\;x^2,x^4y^3,y,x^2y)
\end{array} \\ \hline

\left(\begin{smallmatrix}
\times& 1& 2& 3& 4& 5& 6& 7& 8 \\
1& 9& 3& 4& 10& 6& 11& 12& 13  \\ 
10& 5& 9& 7& 12& 8& 13& 11& 2
\end{smallmatrix}\right)
&
\langle x:x^{18}=1 \rangle\cong\bZ_{18}
&
\begin{array}{l}
(x^2,x^5) \\
(x^2,x^{17},x,x^3, \\
\;x^7,x^9,x^6,x^{12})
\end{array} \\ \hline

\left(\begin{smallmatrix}
\times& 1& 2& 3& 4& 5& 6& 7& 8 \\
1& 9& 3& 4& 10& 6& 11& 12& 13  \\ 
11& 2& 7& 12& 5& 8& 13& 9& 10
\end{smallmatrix}\right)
&
\langle x:x^{16}=1 \rangle\cong\bZ_{16}
&
\begin{array}{l}
(x^3,x) \\
(x^3,x^4,x^7,x^{10}, \\
\;x^{11},x^{14},x^5,x^{12})
\end{array} \\ \hline

\left(\begin{smallmatrix}
\times& 1& 2& 3& 4& 5& 6& 7& 8 \\
1& 9& 3& 10& 5& 11& 7& 12& 13  \\ 
8& 13& 9& 2& 10& 4& 11& 6& 12
\end{smallmatrix}\right)
&
\langle x:x^{13}=1 \rangle\cong\bZ_{13}
&
\begin{array}{l}
(x^9,x^4) \\
(x^9,x,x^{10},x^2, \\
\;x^{11},x^3,x^{12},x^4)
\end{array} \\ \hline

\left(\begin{smallmatrix}
\times& 1& 2& 3& 4& 5& 6& 7& 8 \\
1& 9& 3& 10& 5& 11& 7& 12& 13  \\ 
10& 2& 11& 4& 12& 6& 8& 13& 9
\end{smallmatrix}\right)
&
\langle x:x^{15}=1 \rangle\cong\bZ_{15}
&
\begin{array}{l}
(x^5,x^{11}) \\
(x^5,x,x^6,x^2, \\
\;x^7,x^3,x^8,x^{14})
\end{array} \\ \hline

\left(\begin{smallmatrix}
\times& 1& 2& 3& 4& 5& 6& 7& 8 \\
1& 9& 3& 10& 5& 11& 7& 12& 13  \\ 
13& 2& 9& 4& 10& 6& 11& 8& 12
\end{smallmatrix}\right)
&
\langle x:x^{13}=1 \rangle\cong\bZ_{13}
&
\begin{array}{l}
(x^2,x) \\
(x^2,x^3,x^5,x^6, \\
\;x^8,x^9,x^{11},x^{12})
\end{array} \\ \hline

\left(\begin{smallmatrix}
\times& 1& 2& 3& 4& 5& 6& 7& 8 \\
9& 2& 1& 4& 3& 10& 11& 12& 13  \\ 
10& 9& 11& 12& 13& 7& 8& 5& 6
\end{smallmatrix}\right)
&
\begin{array}{c}
\langle x,y:x^7=y^2=1,\,yxy=x^{-1} \rangle \\
\cong\Dih_7
\end{array}
&
\begin{array}{l}
(y,xy) \\
(x,x^6y,x^3y,x^4, \\
\;x^6,x^5y,x^2y,x^3)
\end{array} \\ \hline
\end{array}
\]

\[
\begin{array}{c|c|l}
\hline

\left(\begin{smallmatrix}
\times& 1& 2& 3& 4& 5& 6& 7& 8 \\
9& 2& 3& 4& 5& 10& 11& 12& 13  \\ 
11& 7& 12& 8& 13& 9& 1& 10& 6
\end{smallmatrix}\right)
&
\langle x:x^{19}=1 \rangle\cong\bZ_{19}
&
\begin{array}{l}
(x,x^{12}) \\
(x^4,x^5,x^6,x^7, \\
\;x^8,x^{11},x^{16},x^{18})
\end{array} \\ \hline

\left(\begin{smallmatrix}
\times& 1& 2& 3& 4& 5& 6& 7& 8 \\
9& 2& 3& 4& 10& 6& 11& 12& 13  \\ 
10& 7& 12& 8& 13& 9& 1& 5& 11
\end{smallmatrix}\right)
&
\langle x:x^{19}=1 \rangle\cong\bZ_{19}
&
\begin{array}{l}
(x^4,x^5) \\
(x^8,x^{12},x^{16},x, \\
\;x^{18},x^3,x^{13},x^2)
\end{array} \\ \hline

\left(\begin{smallmatrix}
\times& 1& 2& 3& 4& 5& 6& 7& 8 \\
9& 2& 3& 10& 5& 6& 11& 12& 13  \\ 
10& 7& 12& 1& 8& 13& 4& 11& 9
\end{smallmatrix}\right)
&
\langle x:x^{17}=1 \rangle\cong\bZ_{17}
&
\begin{array}{l}
(x^{11},x^{12}) \\
(x^{13},x^7,x,x^4, \\
\;x^{15},x^9,x^8,x^{16})
\end{array} \\ \hline

\left(\begin{smallmatrix}
\times& 1& 2& 3& 4& 5& 6& 7& 8 \\
9& 2& 3& 10& 5& 6& 11& 12& 13  \\ 
12& 8& 13& 7& 10& 9& 1& 4& 11
\end{smallmatrix}\right)
&
\langle x:x^{15}=1 \rangle\cong\bZ_{15}
&
\begin{array}{l}
(x^3,x) \\
(x^6,x^9,x^{12},x^{14}, \\
\;x^2,x^5,x^{13},x^7)
\end{array} \\ \hline

\left(\begin{smallmatrix}
\times& 1& 2& 3& 4& 5& 6& 7& 8 \\
9& 2& 10& 4& 11& 6& 12& 8& 13  \\ 
10& 11& 1& 12& 3& 13& 5& 9& 7
\end{smallmatrix}\right)
&
\langle x:x^{14}=1 \rangle\cong\bZ_{14}
&
\begin{array}{l}
(x,x^{13}) \\
(x^{11},x^{12},x^8,x^9, \\
\;x^5,x^6,x^2,x^3)
\end{array} \\ \hline
\end{array}
\]

\subsection{Case 3 by 11.}{$\;$}\newline

There are 29 equivalence classes of nondegenerate pairing matrices, representatives for which are listed below.

\subsubsection{Infinite groups.}{$\;$}\newline

\begin{gp}\label{gp:3x11.1}
The pairing matrix
\[
\left(\begin{smallmatrix} \times & 1 & 2 & 3 & 4 & 5 & 6 & 7 & 8 & 9 & 10 \\ 1 & 2 & 3 & 11 & 5 & 12 & 7 & 13 & 14 & 15 & 16 \\ 4 & 6 & 12 & 13 & 8 & 9 & 14 & 15 & 10 & 16 & 11 \end{smallmatrix}\right)
\]
Gives rise to the group
\[
\Gamma=\langle a,b\mid a^2b=ba^2,\,b^2a=ab^2,\,(ba)^2=(ab)^2\rangle.
\]
(We have reappropriated the symbols $a$ and $b$ to be group elements, rather than elements of the group ring.)
This group has central elements $\rt:=a^2$, $\st:=a^2b^{-2}$, $\ttilde:=ab^{-1}ab^{-1}$ satisfying $\st^2=\ttilde^2=1$.
Modulo the subgroup generated by these elements, we have that the images $x$ and $y$ of $a$ and $b$ in the quotient of $\Gamma$ by the subgroup
generated by $\{\rt,\st,\ttilde\}$ commute and satisfy $x^2=y^2=1$.
Taking $G=\langle x,y\mid x^2=y^2=1,\,xy=yx\rangle\cong\Ints_2\times\Ints_2$, and the map $\phi:G\to\Gamma$ given
by $\phi(1)=1$, $\phi(x)=a$, $\phi(y)=b$ and $\phi(xy)=ab$,
letting $H=\Ints\times\Ints_2\times\Ints_2$ have generators $r=(1,0,0)$, $s=(0,1,0)$ and $t=(0,0,1)$,
and using $\phi$ for guidance, we let $c:G\times G\to H$ be defined by $c(1,g)=1=c(g,1)$ and as shown in the following table,
where $g_1$ is the left column, $g_2$ the top row and $c(g_1,g_2)$ the corresponding entry:

\vskip2ex
\noindent
\hskip3em
\begin{tabular}{r||c|c|c}
& $x$ & $y$ & $xy$ \\
\hline\hline
$x$ & $r$ & $1$ & $r$ \\ \hline
$y$ & $st$ & $rs$ & $rt$ \\ \hline
$xy$ & $rst$ & $rs$ & $r^2t$
\end{tabular}
\vskip2ex

\noindent
(Indeed, 
we express $\phi(g_1)\phi(g_2)\phi(g_1g_2)^{-1}$ in terms of $\rt$, $\st$ and $\ttilde$ and define $c(g_1,g_2)$ to be the corresponding element of $H$.)
One easily verifies that the identity~\eqref{eq:cidentity} holds.
Let $\Gamma_c$ be the group that is an extension of $H$ by $G$ defined using $c$, as described in Remark~\ref{rem:Gammac}.
From this, one defines a map $\Gamma_c\to\Gamma$
by
\[
(r^is^jt^k,g)\mapsto\rt^i\st^j\ttilde^k\phi(g),\qquad(i\in\Ints,\,j,k\in\{0,1\},\,g\in G)
\]
and the choices made ensure that this is a group homomorphism.
On the other hand, the elements $(1,x)$ and $(1,y)$ in $\Gamma_c$ satisfy the defining relations (in $a$ and $b$)
of $\Gamma$, so there is a group homormorphism
$\Gamma\to\Gamma_c$ given by $a\mapsto (1,x)$ and $b\mapsto (1,y)$.
These morphisms are inverses of each other, so $\Gamma\cong\Gamma_c$.

Finally, the criterion mentioned in Remark~\ref{rem:Gammac} for $H$ to be the center of $\Gamma_c$ is fulfilled.
We (partially) summarize this situation by stating that $\Gamma$ has center $Z(\Gamma)$
isomorphic to $\Ints\times\Ints_2\times\Ints_2$ and generated by $\{a^2,a^2b^{-2},ab^{-1}ab^{-1}\}$,
and with quotient $\Gamma/Z(\Gamma)\cong\Ints_2\times\Ints_2$.
In particular, the group is amenable, infinite and nonabelian.
\end{gp}

\begin{gp}
The pairing matrix
\[
\left(\begin{smallmatrix} \times & 1 & 2 & 3 & 4 & 5 & 6 & 7 & 8 & 9 & 10 \\ 1 & 2 & 11 & 4 & 5 & 12 & 7 & 13 & 14 & 15 & 16 \\ 3 & 8 & 9 & 6 & 11 & 16 & 14 & 12 & 10 & 13 & 15 \end{smallmatrix}\right)
\]
gives rise to the group
\[
\Gamma=\langle a,b\mid aba=b^3,\,bab=a^3,\,b^2a^{-2}=ab^{-1}a^{-1}b\rangle.
\]
(We have reappropriated the symbols $a$ and $b$ to be group elements, rather than elements of the group ring.)
We have the resulting relations $b^4=abab=baba=a^4$, so $r:=a^4=b^4$ lies in the center of $\Gamma$.
Also, $s:=ba^{-1}ba^{-1}$ lies in the center of $\Gamma$ and $s^2=1$.
We find that $Z(\Gamma)\cong\Ints\times\Ints_2$ and 
\[
\Gamma/Z(\Gamma)\cong\langle x,y\mid x^4=y^4=1,\,(xy)^2=(xy^{-1})^2=1\rangle
\]
is a nonabelian group of order $16$.
Thus, $\Gamma$ is amenable, infinite and nonabelian.
\end{gp}

\subsubsection{Finite  groups.}{$\;$}\newline

All the other  nondegenerate pairing matrices give
rise to finite ULIE groups.
A list of representative of the equivalence clases of these is below:

\begin{alignat*}{2}
\left(\begin{smallmatrix} \times & 1 & 2 & 3 & 4 & 5 & 6 & 7 & 8 & 9 & 10 \\ 1 & 2 & 3 & 4 & 11 & 6 & 12 & 13 & 14 & 15 & 16 \\ 13 & 8 & 14 & 9 & 15 & 10 & 16 & 5 & 12 & 7 & 11 \end{smallmatrix}\right)
&\quad
\left(\begin{smallmatrix} \times & 1 & 2 & 3 & 4 & 5 & 6 & 7 & 8 & 9 & 10 \\ 1 & 2 & 3 & 4 & 11 & 6 & 12 & 13 & 14 & 15 & 16 \\ 13 & 8 & 14 & 9 & 15 & 10 & 16 & 11 & 5 & 12 & 7 \end{smallmatrix}\right)
&\quad
\left(\begin{smallmatrix} \times & 1 & 2 & 3 & 4 & 5 & 6 & 7 & 8 & 9 & 10 \\ 1 & 2 & 3 & 11 & 5 & 12 & 7 & 13 & 14 & 15 & 16 \\ 12 & 13 & 4 & 6 & 8 & 9 & 14 & 15 & 10 & 16 & 11 \end{smallmatrix}\right)
\\[2ex]
\left(\begin{smallmatrix} \times & 1 & 2 & 3 & 4 & 5 & 6 & 7 & 8 & 9 & 10 \\ 1 & 2 & 11 & 4 & 5 & 6 & 12 & 13 & 14 & 15 & 16 \\ 13 & 8 & 14 & 9 & 15 & 10 & 16 & 11 & 3 & 12 & 7 \end{smallmatrix}\right)
&\quad
\left(\begin{smallmatrix} \times & 1 & 2 & 3 & 4 & 5 & 6 & 7 & 8 & 9 & 10 \\ 1 & 2 & 11 & 4 & 5 & 12 & 7 & 13 & 14 & 15 & 16 \\ 8 & 14 & 3 & 9 & 15 & 6 & 10 & 16 & 13 & 11 & 12 \end{smallmatrix}\right)
&\quad
\left(\begin{smallmatrix} \times & 1 & 2 & 3 & 4 & 5 & 6 & 7 & 8 & 9 & 10 \\ 1 & 2 & 11 & 4 & 5 & 12 & 7 & 13 & 14 & 15 & 16 \\ 8 & 14 & 6 & 11 & 9 & 15 & 10 & 16 & 12 & 13 & 3 \end{smallmatrix}\right)
\\[2ex]
\left(\begin{smallmatrix} \times & 1 & 2 & 3 & 4 & 5 & 6 & 7 & 8 & 9 & 10 \\ 1 & 2 & 11 & 4 & 5 & 12 & 7 & 13 & 14 & 15 & 16 \\ 12 & 8 & 14 & 9 & 15 & 6 & 10 & 16 & 13 & 11 & 3 \end{smallmatrix}\right)
&\quad
\left(\begin{smallmatrix} \times & 1 & 2 & 3 & 4 & 5 & 6 & 7 & 8 & 9 & 10 \\ 1 & 2 & 11 & 4 & 5 & 12 & 7 & 13 & 14 & 15 & 16 \\ 14 & 9 & 15 & 11 & 10 & 16 & 12 & 8 & 6 & 13 & 3 \end{smallmatrix}\right)
&\quad
\left(\begin{smallmatrix} \times & 1 & 2 & 3 & 4 & 5 & 6 & 7 & 8 & 9 & 10 \\ 1 & 11 & 3 & 2 & 5 & 4 & 12 & 13 & 14 & 15 & 16 \\ 6 & 13 & 7 & 8 & 9 & 10 & 15 & 11 & 16 & 12 & 14 \end{smallmatrix}\right)
\displaybreak[2] \\[2ex]
\left(\begin{smallmatrix} \times & 1 & 2 & 3 & 4 & 5 & 6 & 7 & 8 & 9 & 10 \\ 1 & 11 & 3 & 2 & 5 & 4 & 12 & 13 & 14 & 15 & 16 \\ 6 & 13 & 12 & 14 & 15 & 16 & 11 & 9 & 10 & 7 & 8 \end{smallmatrix}\right)
&\quad
\left(\begin{smallmatrix} \times & 1 & 2 & 3 & 4 & 5 & 6 & 7 & 8 & 9 & 10 \\ 1 & 11 & 3 & 4 & 5 & 12 & 7 & 13 & 14 & 15 & 16 \\ 8 & 14 & 9 & 15 & 10 & 16 & 11 & 2 & 12 & 6 & 13 \end{smallmatrix}\right)
&\quad
\left(\begin{smallmatrix} \times & 1 & 2 & 3 & 4 & 5 & 6 & 7 & 8 & 9 & 10 \\ 1 & 11 & 3 & 4 & 5 & 12 & 7 & 13 & 14 & 15 & 16 \\ 12 & 2 & 8 & 14 & 9 & 15 & 10 & 16 & 6 & 13 & 11 \end{smallmatrix}\right)
\\[2ex]
\left(\begin{smallmatrix} \times & 1 & 2 & 3 & 4 & 5 & 6 & 7 & 8 & 9 & 10 \\ 1 & 11 & 3 & 4 & 12 & 6 & 7 & 13 & 14 & 15 & 16 \\ 8 & 14 & 9 & 15 & 2 & 10 & 16 & 5 & 12 & 13 & 11 \end{smallmatrix}\right)
&\quad
\left(\begin{smallmatrix} \times & 1 & 2 & 3 & 4 & 5 & 6 & 7 & 8 & 9 & 10 \\ 1 & 11 & 3 & 4 & 12 & 6 & 7 & 13 & 14 & 15 & 16 \\ 14 & 2 & 9 & 15 & 5 & 10 & 16 & 8 & 11 & 12 & 13 \end{smallmatrix}\right)
&\quad
\left(\begin{smallmatrix} \times & 1 & 2 & 3 & 4 & 5 & 6 & 7 & 8 & 9 & 10 \\ 1 & 11 & 3 & 12 & 5 & 13 & 7 & 14 & 9 & 15 & 16 \\ 10 & 16 & 11 & 2 & 12 & 4 & 13 & 6 & 14 & 8 & 15 \end{smallmatrix}\right)
\displaybreak[2] \\[2ex]
\left(\begin{smallmatrix} \times & 1 & 2 & 3 & 4 & 5 & 6 & 7 & 8 & 9 & 10 \\ 1 & 11 & 3 & 12 & 5 & 13 & 7 & 14 & 9 & 15 & 16 \\ 12 & 2 & 13 & 4 & 14 & 6 & 15 & 8 & 10 & 16 & 11 \end{smallmatrix}\right)
&\quad
\left(\begin{smallmatrix} \times & 1 & 2 & 3 & 4 & 5 & 6 & 7 & 8 & 9 & 10 \\ 1 & 11 & 3 & 12 & 5 & 13 & 7 & 14 & 9 & 15 & 16 \\ 16 & 2 & 11 & 4 & 12 & 6 & 13 & 8 & 14 & 10 & 15 \end{smallmatrix}\right)
&\quad
\left(\begin{smallmatrix} \times & 1 & 2 & 3 & 4 & 5 & 6 & 7 & 8 & 9 & 10 \\ 11 & 2 & 3 & 4 & 5 & 12 & 7 & 13 & 14 & 15 & 16 \\ 12 & 8 & 14 & 9 & 15 & 1 & 10 & 16 & 6 & 13 & 11 \end{smallmatrix}\right)
\\[2ex]
\left(\begin{smallmatrix} \times & 1 & 2 & 3 & 4 & 5 & 6 & 7 & 8 & 9 & 10 \\ 11 & 2 & 3 & 4 & 5 & 12 & 7 & 13 & 14 & 15 & 16 \\ 13 & 8 & 14 & 9 & 15 & 6 & 10 & 16 & 12 & 11 & 1 \end{smallmatrix}\right)
&\quad
\left(\begin{smallmatrix} \times & 1 & 2 & 3 & 4 & 5 & 6 & 7 & 8 & 9 & 10 \\ 11 & 2 & 3 & 4 & 5 & 12 & 7 & 13 & 14 & 15 & 16 \\ 14 & 9 & 15 & 10 & 16 & 8 & 12 & 11 & 1 & 6 & 13 \end{smallmatrix}\right)
&\quad
\left(\begin{smallmatrix} \times & 1 & 2 & 3 & 4 & 5 & 6 & 7 & 8 & 9 & 10 \\ 11 & 2 & 3 & 4 & 12 & 6 & 7 & 13 & 14 & 15 & 16 \\ 12 & 14 & 8 & 15 & 9 & 11 & 16 & 10 & 13 & 5 & 1 \end{smallmatrix}\right)
\displaybreak[2] \\[2ex]
\left(\begin{smallmatrix} \times & 1 & 2 & 3 & 4 & 5 & 6 & 7 & 8 & 9 & 10 \\ 11 & 2 & 3 & 4 & 12 & 6 & 7 & 13 & 14 & 15 & 16 \\ 13 & 8 & 14 & 9 & 15 & 10 & 16 & 1 & 11 & 5 & 12 \end{smallmatrix}\right)
&\quad
\left(\begin{smallmatrix} \times & 1 & 2 & 3 & 4 & 5 & 6 & 7 & 8 & 9 & 10 \\ 11 & 2 & 3 & 4 & 12 & 6 & 7 & 13 & 14 & 15 & 16 \\ 14 & 13 & 11 & 9 & 15 & 10 & 16 & 8 & 5 & 1 & 12 \end{smallmatrix}\right)
&\quad
\left(\begin{smallmatrix} \times & 1 & 2 & 3 & 4 & 5 & 6 & 7 & 8 & 9 & 10 \\ 11 & 2 & 3 & 4 & 12 & 6 & 13 & 8 & 14 & 15 & 16 \\ 13 & 12 & 9 & 15 & 7 & 10 & 16 & 11 & 5 & 14 & 1 \end{smallmatrix}\right)
\\[2ex]
\left(\begin{smallmatrix} \times & 1 & 2 & 3 & 4 & 5 & 6 & 7 & 8 & 9 & 10 \\ 11 & 2 & 3 & 12 & 5 & 6 & 13 & 8 & 14 & 15 & 16 \\ 12 & 14 & 15 & 7 & 11 & 10 & 9 & 1 & 16 & 4 & 13 \end{smallmatrix}\right)
&\quad
\left(\begin{smallmatrix} \times & 1 & 2 & 3 & 4 & 5 & 6 & 7 & 8 & 9 & 10 \\ 11 & 2 & 3 & 12 & 5 & 6 & 13 & 8 & 14 & 15 & 16 \\ 14 & 9 & 15 & 4 & 10 & 16 & 7 & 11 & 1 & 12 & 13 \end{smallmatrix}\right)
&\quad
\left(\begin{smallmatrix} \times & 1 & 2 & 3 & 4 & 5 & 6 & 7 & 8 & 9 & 10 \\ 11 & 2 & 12 & 4 & 13 & 6 & 14 & 8 & 15 & 10 & 16 \\ 12 & 13 & 1 & 14 & 3 & 15 & 5 & 16 & 7 & 11 & 9 \end{smallmatrix}\right)
\end{alignat*}


\subsection{Case 5 by 5, $a=b$.}{$\;$}\newline
\label{subsec:5x5a=b}

There are $100$ inequivalent pairing matrices of the form
\begin{equation}\label{eq:mat1234}
\left(
\begin{matrix}
\times & 1 & 2 & 3 & 4 \\
1 & * & * & * & * \\
2 & * & * & * & * \\
3 & * & * & * & * \\
4 & * & * & * & *
\end{matrix}
\right)
\end{equation}
that were not found by GAP to be degenerate.
GAP decided that $79$ of them are nondegenerate, while for $21$ of them, GAP did not 
decide (in reasonable time) equality of words in the group
and so didn't decide on degeneracy.
In the corresponding group algebras, the resulting relations imply
$a:=1+a_1+\cdots+a_4=1+b_1+\cdots+b_4=:b$, which trivially gives $ba=1$.
Although these cases are not interesting for the direct finiteness conjecture, they are of interest for the invertibles
conjecture, so we list all the pairing matrices and analyze the groups, without describing them in detail.
Thus, throughout we have $a_i=b_i$ for all $i$, and the group is generated by $a_1,a_2,a_3,a_4$.

\begin{matcases}
Any pairing $\pi$ with identifications
\begin{equation}\label{eq:ijkl}
(i,k)\sim(j,\ell)\text{ and }(i,\ell)\sim(j,k),
\end{equation}
where $i\ne j$ and $k\ne\ell$, yields the relations
\[
a_{\ell}a_k^{-1}=a_j^{-1}a_i=a_ka_\ell^{-1}
\]
which implies that $a_\ell a_k^{-1}$ has order $2$.
Thus, $a_k$ and $a_\ell$ are identified in $\Gamma/N_\tor(\Gamma)$.
There were $18$ inequivalent pairing matrices whose pairings included identifications of the form~\eqref{eq:ijkl}, but which our GAP
algorithm failed to decide about degeneracy.
These are:
\begin{alignat*}{4}
\left(\begin{smallmatrix} \times& 1& 2& 3& 4 \\[0.5ex] 1& 5& 6& 7& 8 \\[0.5ex] 2& 6& 5& 8& 7 \\[0.5ex] 3& 9& 10& 11& 12 \\[0.5ex] 4& 10& 9& 12& 11 \end{smallmatrix}\right)
&\quad
\left(\begin{smallmatrix} \times& 1& 2& 3& 4 \\[0.5ex] 1& 5& 6& 7& 8 \\[0.5ex] 2& 6& 5& 9& 10 \\[0.5ex] 3& 7& 9& 11& 12 \\[0.5ex] 4& 10& 8& 12& 11 \end{smallmatrix}\right)
&\quad
\left(\begin{smallmatrix} \times& 1& 2& 3& 4 \\[0.5ex] 1& 5& 6& 7& 8 \\[0.5ex] 2& 6& 5& 9& 10 \\[0.5ex] 3& 8& 10& 11& 12 \\[0.5ex] 4& 9& 7& 12& 11 \end{smallmatrix}\right)
&\quad
\left(\begin{smallmatrix} \times& 1& 2& 3& 4 \\[0.5ex] 1& 5& 6& 7& 8 \\[0.5ex] 2& 6& 9& 8& 7 \\[0.5ex] 3& 10& 11& 5& 12 \\[0.5ex] 4& 11& 10& 12& 9 \end{smallmatrix}\right)
&\quad
\left(\begin{smallmatrix} \times& 1& 2& 3& 4 \\[0.5ex] 1& 5& 6& 7& 8 \\[0.5ex] 2& 6& 9& 8& 7 \\[0.5ex] 3& 10& 11& 12& 5 \\[0.5ex] 4& 11& 10& 9& 12 \end{smallmatrix}\right) \displaybreak[2]  \\[2ex]
\left(\begin{smallmatrix} \times& 1& 2& 3& 4 \\[0.5ex] 1& 5& 6& 7& 8 \\[0.5ex] 2& 6& 9& 10& 11 \\[0.5ex] 3& 8& 7& 12& 5 \\[0.5ex] 4& 11& 10& 9& 12 \end{smallmatrix}\right)
&\quad
\left(\begin{smallmatrix} \times& 1& 2& 3& 4 \\[0.5ex] 1& 5& 6& 7& 8 \\[0.5ex] 2& 7& 8& 5& 9 \\[0.5ex] 3& 6& 10& 11& 12 \\[0.5ex] 4& 11& 12& 9& 10 \end{smallmatrix}\right)
&\quad
\left(\begin{smallmatrix} \times& 1& 2& 3& 4 \\[0.5ex] 1& 5& 6& 7& 8 \\[0.5ex] 2& 7& 8& 5& 9 \\[0.5ex] 3& 9& 10& 11& 12 \\[0.5ex] 4& 11& 12& 6& 10 \end{smallmatrix}\right)
&\quad
\left(\begin{smallmatrix} \times& 1& 2& 3& 4 \\[0.5ex] 1& 5& 6& 7& 8 \\[0.5ex] 2& 7& 8& 9& 6 \\[0.5ex] 3& 10& 5& 11& 12 \\[0.5ex] 4& 11& 12& 10& 9 \end{smallmatrix}\right)
&\quad
\left(\begin{smallmatrix} \times& 1& 2& 3& 4 \\[0.5ex] 1& 5& 6& 7& 8 \\[0.5ex] 2& 7& 8& 9& 10 \\[0.5ex] 3& 6& 5& 11& 12 \\[0.5ex] 4& 11& 12& 10& 9 \end{smallmatrix}\right) \displaybreak[2] \\[2ex]
\left(\begin{smallmatrix} \times& 1& 2& 3& 4 \\[0.5ex] 1& 5& 6& 7& 8 \\[0.5ex] 2& 9& 5& 8& 7 \\[0.5ex] 3& 10& 11& 12& 6 \\[0.5ex] 4& 11& 10& 9& 12 \end{smallmatrix}\right)
&\quad
\left(\begin{smallmatrix} \times& 1& 2& 3& 4 \\[0.5ex] 1& 5& 6& 7& 8 \\[0.5ex] 2& 9& 5& 8& 10 \\[0.5ex] 3& 10& 11& 12& 9 \\[0.5ex] 4& 11& 7& 6& 12 \end{smallmatrix}\right)
&\quad
\left(\begin{smallmatrix} \times& 1& 2& 3& 4 \\[0.5ex] 1& 5& 6& 7& 8 \\[0.5ex] 2& 9& 5& 8& 10 \\[0.5ex] 3& 10& 11& 12& 9 \\[0.5ex] 4& 12& 7& 6& 11 \end{smallmatrix}\right)
&\quad
\left(\begin{smallmatrix} \times& 1& 2& 3& 4 \\[0.5ex] 1& 5& 6& 7& 8 \\[0.5ex] 2& 9& 5& 10& 11 \\[0.5ex] 3& 11& 8& 12& 6 \\[0.5ex] 4& 10& 7& 9& 12 \end{smallmatrix}\right)
&\quad
\left(\begin{smallmatrix} \times& 1& 2& 3& 4 \\[0.5ex] 1& 5& 6& 7& 8 \\[0.5ex] 2& 9& 5& 10& 11 \\[0.5ex] 3& 11& 8& 12& 6 \\[0.5ex] 4& 10& 12& 9& 7 \end{smallmatrix}\right) \\[2ex]
\left(\begin{smallmatrix} \times& 1& 2& 3& 4 \\[0.5ex] 1& 5& 6& 7& 8 \\[0.5ex] 2& 9& 5& 10& 11 \\[0.5ex] 3& 11& 12& 8& 9 \\[0.5ex] 4& 12& 7& 6& 10 \end{smallmatrix}\right)
&\quad
\left(\begin{smallmatrix} \times& 1& 2& 3& 4 \\[0.5ex] 1& 5& 6& 7& 8 \\[0.5ex] 2& 9& 7& 10& 5 \\[0.5ex] 3& 11& 8& 12& 6 \\[0.5ex] 4& 10& 12& 9& 11 \end{smallmatrix}\right)
&\quad
\left(\begin{smallmatrix} \times& 1& 2& 3& 4 \\[0.5ex] 1& 5& 6& 7& 8 \\[0.5ex] 2& 9& 7& 10& 11 \\[0.5ex] 3& 12& 8& 11& 6 \\[0.5ex] 4& 10& 5& 9& 12 \end{smallmatrix}\right)
\end{alignat*}
There were $19$ inequivalent pairing matrices whose pairings included identifications of the form~\eqref{eq:ijkl} and which our GAP
algorithm showed are nondegenerate.
These are:
\begin{alignat*}{4}
\left(\begin{smallmatrix} \times& 1& 2& 3& 4 \\[0.5ex] 1& 5& 6& 7& 8 \\[0.5ex] 2& 6& 5& 8& 9 \\[0.5ex] 3& 7& 10& 11& 12 \\[0.5ex] 4& 10& 9& 12& 11 \end{smallmatrix}\right) 
&\quad
\left(\begin{smallmatrix} \times& 1& 2& 3& 4 \\[0.5ex] 1& 5& 6& 7& 8 \\[0.5ex] 2& 6& 5& 8& 9 \\[0.5ex] 3& 9& 10& 11& 12 \\[0.5ex] 4& 10& 7& 12& 11 \end{smallmatrix}\right)
&\quad
\left(\begin{smallmatrix} \times& 1& 2& 3& 4 \\[0.5ex] 1& 5& 6& 7& 8 \\[0.5ex] 2& 6& 5& 8& 9 \\[0.5ex] 3& 10& 7& 11& 12 \\[0.5ex] 4& 9& 10& 12& 11 \end{smallmatrix}\right) 
&\quad
\left(\begin{smallmatrix} \times& 1& 2& 3& 4 \\[0.5ex] 1& 5& 6& 7& 8 \\[0.5ex] 2& 6& 5& 9& 10 \\[0.5ex] 3& 7& 8& 11& 12 \\[0.5ex] 4& 9& 10& 12& 11 \end{smallmatrix}\right)
&\quad
\left(\begin{smallmatrix} \times& 1& 2& 3& 4 \\[0.5ex] 1& 5& 6& 7& 8 \\[0.5ex] 2& 6& 5& 9& 10 \\[0.5ex] 3& 7& 9& 11& 12 \\[0.5ex] 4& 8& 10& 12& 11 \end{smallmatrix}\right)  \displaybreak[2] \\[2ex]
\left(\begin{smallmatrix} \times& 1& 2& 3& 4 \\[0.5ex] 1& 5& 6& 7& 8 \\[0.5ex] 2& 6& 5& 9& 10 \\[0.5ex] 3& 7& 9& 11& 12 \\[0.5ex] 4& 8& 12& 10& 11 \end{smallmatrix}\right)
&\quad
\left(\begin{smallmatrix} \times& 1& 2& 3& 4 \\[0.5ex] 1& 5& 6& 7& 8 \\[0.5ex] 2& 6& 5& 9& 10 \\[0.5ex] 3& 8& 7& 11& 12 \\[0.5ex] 4& 10& 9& 12& 11 \end{smallmatrix}\right) 
&\quad
\left(\begin{smallmatrix} \times& 1& 2& 3& 4 \\[0.5ex] 1& 5& 6& 7& 8 \\[0.5ex] 2& 6& 5& 9& 10 \\[0.5ex] 3& 8& 10& 11& 12 \\[0.5ex] 4& 7& 9& 12& 11 \end{smallmatrix}\right)
&\quad
\left(\begin{smallmatrix} \times& 1& 2& 3& 4 \\[0.5ex] 1& 5& 6& 7& 8 \\[0.5ex] 2& 6& 5& 9& 10 \\[0.5ex] 3& 10& 8& 11& 12 \\[0.5ex] 4& 9& 7& 12& 11 \end{smallmatrix}\right)
&\quad
\left(\begin{smallmatrix} \times& 1& 2& 3& 4 \\[0.5ex] 1& 5& 6& 7& 8 \\[0.5ex] 2& 6& 7& 9& 10 \\[0.5ex] 3& 11& 9& 5& 12 \\[0.5ex] 4& 12& 10& 8& 11 \end{smallmatrix}\right)  \displaybreak[2] \\[2ex]
\left(\begin{smallmatrix} \times& 1& 2& 3& 4 \\[0.5ex] 1& 5& 6& 7& 8 \\[0.5ex] 2& 6& 9& 8& 10 \\[0.5ex] 3& 11& 10& 12& 9 \\[0.5ex] 4& 7& 11& 5& 12 \end{smallmatrix}\right)
&\quad
\left(\begin{smallmatrix} \times& 1& 2& 3& 4 \\[0.5ex] 1& 5& 6& 7& 8 \\[0.5ex] 2& 6& 9& 10& 11 \\[0.5ex] 3& 8& 11& 12& 5 \\[0.5ex] 4& 7& 10& 9& 12 \end{smallmatrix}\right)
&\quad
\left(\begin{smallmatrix} \times& 1& 2& 3& 4 \\[0.5ex] 1& 5& 6& 7& 8 \\[0.5ex] 2& 7& 8& 5& 6 \\[0.5ex] 3& 9& 10& 11& 12 \\[0.5ex] 4& 11& 12& 9& 10 \end{smallmatrix}\right) 
&\quad
\left(\begin{smallmatrix} \times& 1& 2& 3& 4 \\[0.5ex] 1& 5& 6& 7& 8 \\[0.5ex] 2& 7& 8& 5& 9 \\[0.5ex] 3& 10& 11& 6& 12 \\[0.5ex] 4& 9& 12& 10& 11 \end{smallmatrix}\right)
&\quad
\left(\begin{smallmatrix} \times& 1& 2& 3& 4 \\[0.5ex] 1& 5& 6& 7& 8 \\[0.5ex] 2& 7& 8& 5& 9 \\[0.5ex] 3& 10& 11& 9& 12 \\[0.5ex] 4& 6& 12& 10& 11 \end{smallmatrix}\right)  \displaybreak[2] \\[2ex]
\left(\begin{smallmatrix} \times& 1& 2& 3& 4 \\[0.5ex] 1& 5& 6& 7& 8 \\[0.5ex] 2& 7& 8& 9& 10 \\[0.5ex] 3& 11& 12& 10& 9 \\[0.5ex] 4& 6& 5& 11& 12 \end{smallmatrix}\right)
&\quad
\left(\begin{smallmatrix} \times& 1& 2& 3& 4 \\[0.5ex] 1& 5& 6& 7& 8 \\[0.5ex] 2& 7& 9& 5& 10 \\[0.5ex] 3& 8& 11& 10& 12 \\[0.5ex] 4& 6& 12& 9& 11 \end{smallmatrix}\right)
&\quad
\left(\begin{smallmatrix} \times& 1& 2& 3& 4 \\[0.5ex] 1& 5& 6& 7& 8 \\[0.5ex] 2& 7& 9& 5& 10 \\[0.5ex] 3& 9& 11& 10& 12 \\[0.5ex] 4& 6& 12& 8& 11 \end{smallmatrix}\right) 
&\quad
\left(\begin{smallmatrix} \times& 1& 2& 3& 4 \\[0.5ex] 1& 5& 6& 7& 8 \\[0.5ex] 2& 7& 9& 5& 10 \\[0.5ex] 3& 10& 11& 8& 12 \\[0.5ex] 4& 9& 12& 6& 11 \end{smallmatrix}\right)
\end{alignat*}
\end{matcases}

\begin{matcases}
For the pairings $\pi$ corresponding to the following pairing matrices,
GAP determined that they are nondegenerate and that the corresponding ULIE groups $\Gamma_\pi$
are abelian.
Thus, (see Remark~\ref{rem:ICabelian}), the images of the support elements of $a$ in $\Gamma_\pi/N_\tor(\Gamma_\pi)$
are not distinct.
\begin{alignat*}{4}
\left(\begin{smallmatrix} \times& 1& 2& 3& 4 \\[0.5ex] 1& 5& 6& 7& 8 \\[0.5ex] 2& 6& 7& 8& 5 \\[0.5ex] 3& 9& 10& 11& 12 \\[0.5ex] 4& 10& 11& 12& 9  \end{smallmatrix}\right)
&\quad
\left(\begin{smallmatrix} \times& 1& 2& 3& 4 \\[0.5ex] 1& 5& 6& 7& 8 \\[0.5ex] 2& 6& 7& 8& 9 \\[0.5ex] 3& 10& 11& 5& 12 \\[0.5ex] 4& 11& 9& 12& 10  \end{smallmatrix}\right)
&\quad
\left(\begin{smallmatrix} \times& 1& 2& 3& 4 \\[0.5ex] 1& 5& 6& 7& 8 \\[0.5ex] 2& 6& 7& 8& 9 \\[0.5ex] 3& 10& 11& 9& 12 \\[0.5ex] 4& 11& 5& 12& 10  \end{smallmatrix}\right)
&\quad
\left(\begin{smallmatrix} \times& 1& 2& 3& 4 \\[0.5ex] 1& 5& 6& 7& 8 \\[0.5ex] 2& 6& 7& 9& 5 \\[0.5ex] 3& 10& 8& 11& 12 \\[0.5ex] 4& 9& 11& 12& 10  \end{smallmatrix}\right)
&\quad
\left(\begin{smallmatrix} \times& 1& 2& 3& 4 \\[0.5ex] 1& 5& 6& 7& 8 \\[0.5ex] 2& 6& 7& 9& 5 \\[0.5ex] 3& 10& 9& 11& 12 \\[0.5ex] 4& 8& 11& 12& 10  \end{smallmatrix}\right)  \displaybreak[2] \\[2ex]
\left(\begin{smallmatrix} \times& 1& 2& 3& 4 \\[0.5ex] 1& 5& 6& 7& 8 \\[0.5ex] 2& 6& 7& 9& 10 \\[0.5ex] 3& 11& 8& 5& 12 \\[0.5ex] 4& 9& 10& 12& 11  \end{smallmatrix}\right)
&\quad
\left(\begin{smallmatrix} \times& 1& 2& 3& 4 \\[0.5ex] 1& 5& 6& 7& 8 \\[0.5ex] 2& 6& 7& 9& 10 \\[0.5ex] 3& 11& 8& 10& 12 \\[0.5ex] 4& 9& 5& 12& 11  \end{smallmatrix}\right)
&\quad
\left(\begin{smallmatrix} \times& 1& 2& 3& 4 \\[0.5ex] 1& 5& 6& 7& 8 \\[0.5ex] 2& 6& 7& 9& 10 \\[0.5ex] 3& 11& 9& 5& 12 \\[0.5ex] 4& 8& 10& 12& 11  \end{smallmatrix}\right)
&\quad
\left(\begin{smallmatrix} \times& 1& 2& 3& 4 \\[0.5ex] 1& 5& 6& 7& 8 \\[0.5ex] 2& 6& 7& 9& 10 \\[0.5ex] 3& 11& 9& 10& 12 \\[0.5ex] 4& 8& 5& 12& 11  \end{smallmatrix}\right)
&\quad
\left(\begin{smallmatrix} \times& 1& 2& 3& 4 \\[0.5ex] 1& 5& 6& 7& 8 \\[0.5ex] 2& 6& 9& 8& 10 \\[0.5ex] 3& 7& 11& 9& 12 \\[0.5ex] 4& 11& 10& 12& 5  \end{smallmatrix}\right)  \displaybreak[2] \\[2ex]
\left(\begin{smallmatrix} \times& 1& 2& 3& 4 \\[0.5ex] 1& 5& 6& 7& 8 \\[0.5ex] 2& 6& 9& 8& 10 \\[0.5ex] 3& 7& 11& 10& 12 \\[0.5ex] 4& 11& 5& 12& 9  \end{smallmatrix}\right)
&\quad
\left(\begin{smallmatrix} \times& 1& 2& 3& 4 \\[0.5ex] 1& 5& 6& 7& 8 \\[0.5ex] 2& 6& 9& 8& 10 \\[0.5ex] 3& 7& 11& 12& 9 \\[0.5ex] 4& 11& 10& 5& 12  \end{smallmatrix}\right)
&\quad
\left(\begin{smallmatrix} \times& 1& 2& 3& 4 \\[0.5ex] 1& 5& 6& 7& 8 \\[0.5ex] 2& 6& 9& 8& 10 \\[0.5ex] 3& 9& 11& 10& 12 \\[0.5ex] 4& 11& 5& 12& 7  \end{smallmatrix}\right)
&\quad
\left(\begin{smallmatrix} \times& 1& 2& 3& 4 \\[0.5ex] 1& 5& 6& 7& 8 \\[0.5ex] 2& 6& 9& 8& 10 \\[0.5ex] 3& 10& 11& 12& 5 \\[0.5ex] 4& 11& 7& 9& 12  \end{smallmatrix}\right)
&\quad
\left(\begin{smallmatrix} \times& 1& 2& 3& 4 \\[0.5ex] 1& 5& 6& 7& 8 \\[0.5ex] 2& 6& 9& 8& 10 \\[0.5ex] 3& 10& 11& 12& 9 \\[0.5ex] 4& 11& 7& 5& 12  \end{smallmatrix}\right)  \displaybreak[2] \\[2ex]
\left(\begin{smallmatrix} \times& 1& 2& 3& 4 \\[0.5ex] 1& 5& 6& 7& 8 \\[0.5ex] 2& 6& 9& 10& 11 \\[0.5ex] 3& 9& 8& 11& 12 \\[0.5ex] 4& 10& 5& 12& 7  \end{smallmatrix}\right)
&\quad
\left(\begin{smallmatrix} \times& 1& 2& 3& 4 \\[0.5ex] 1& 5& 6& 7& 8 \\[0.5ex] 2& 9& 7& 8& 5 \\[0.5ex] 3& 10& 11& 12& 6 \\[0.5ex] 4& 11& 12& 9& 10  \end{smallmatrix}\right)
&\quad
\left(\begin{smallmatrix} \times& 1& 2& 3& 4 \\[0.5ex] 1& 5& 6& 7& 8 \\[0.5ex] 2& 9& 7& 8& 5 \\[0.5ex] 3& 10& 11& 12& 9 \\[0.5ex] 4& 11& 12& 6& 10  \end{smallmatrix}\right)
&\quad
\left(\begin{smallmatrix} \times& 1& 2& 3& 4 \\[0.5ex] 1& 5& 6& 7& 8 \\[0.5ex] 2& 9& 7& 8& 10 \\[0.5ex] 3& 11& 12& 10& 6 \\[0.5ex] 4& 12& 5& 9& 11  \end{smallmatrix}\right)
&\quad
\left(\begin{smallmatrix} \times& 1& 2& 3& 4 \\[0.5ex] 1& 5& 6& 7& 8 \\[0.5ex] 2& 9& 7& 8& 10 \\[0.5ex] 3& 11& 12& 10& 9 \\[0.5ex] 4& 12& 5& 6& 11  \end{smallmatrix}\right) \\[2ex]
\left(\begin{smallmatrix} \times& 1& 2& 3& 4 \\[0.5ex] 1& 5& 6& 7& 8 \\[0.5ex] 2& 9& 7& 10& 5 \\[0.5ex] 3& 11& 8& 12& 9 \\[0.5ex] 4& 10& 12& 6& 11  \end{smallmatrix}\right)
\end{alignat*}
\end{matcases}

\begin{matcases}
For the pairings $\pi$ of the followig pairing matrices,
GAP determined that they are nondegenerate and that in the corresponding ULIE groups $\Gamma_\pi$,
the element $a_1a_2^{-1}$ has order $2$:
\begin{alignat*}{4}
\left(\begin{smallmatrix} \times& 1& 2& 3& 4 \\[0.5ex] 1& 5& 6& 7& 8 \\[0.5ex] 2& 6& 9& 8& 10 \\[0.5ex] 3& 11& 7& 5& 12 \\[0.5ex] 4& 10& 11& 12& 9  \end{smallmatrix}\right)
&\quad
\left(\begin{smallmatrix} \times& 1& 2& 3& 4 \\[0.5ex] 1& 5& 6& 7& 8 \\[0.5ex] 2& 6& 9& 8& 10 \\[0.5ex] 3& 11& 7& 9& 12 \\[0.5ex] 4& 10& 11& 12& 5  \end{smallmatrix}\right)
&\quad
\left(\begin{smallmatrix} \times& 1& 2& 3& 4 \\[0.5ex] 1& 5& 6& 7& 8 \\[0.5ex] 2& 6& 9& 10& 11 \\[0.5ex] 3& 8& 7& 5& 12 \\[0.5ex] 4& 11& 10& 12& 9  \end{smallmatrix}\right)
&\quad
\left(\begin{smallmatrix} \times& 1& 2& 3& 4 \\[0.5ex] 1& 5& 6& 7& 8 \\[0.5ex] 2& 6& 9& 10& 11 \\[0.5ex] 3& 8& 7& 9& 12 \\[0.5ex] 4& 11& 10& 12& 5  \end{smallmatrix}\right)
&\quad
\left(\begin{smallmatrix} \times& 1& 2& 3& 4 \\[0.5ex] 1& 5& 6& 7& 8 \\[0.5ex] 2& 6& 9& 10& 11 \\[0.5ex] 3& 8& 11& 5& 12 \\[0.5ex] 4& 7& 10& 12& 9  \end{smallmatrix}\right)   \displaybreak[2] \\[2ex]
\left(\begin{smallmatrix} \times& 1& 2& 3& 4 \\[0.5ex] 1& 5& 6& 7& 8 \\[0.5ex] 2& 6& 9& 10& 11 \\[0.5ex] 3& 11& 8& 5& 12 \\[0.5ex] 4& 10& 7& 12& 9  \end{smallmatrix}\right)
&\quad
\left(\begin{smallmatrix} \times& 1& 2& 3& 4 \\[0.5ex] 1& 5& 6& 7& 8 \\[0.5ex] 2& 6& 9& 10& 11 \\[0.5ex] 3& 11& 8& 12& 5 \\[0.5ex] 4& 10& 7& 9& 12  \end{smallmatrix}\right)
&\quad
\left(\begin{smallmatrix} \times& 1& 2& 3& 4 \\[0.5ex] 1& 5& 6& 7& 8 \\[0.5ex] 2& 7& 8& 9& 10 \\[0.5ex] 3& 11& 12& 6& 5 \\[0.5ex] 4& 10& 9& 11& 12  \end{smallmatrix}\right)
&\quad
\left(\begin{smallmatrix} \times& 1& 2& 3& 4 \\[0.5ex] 1& 5& 6& 7& 8 \\[0.5ex] 2& 7& 8& 9& 10 \\[0.5ex] 3& 11& 12& 6& 9 \\[0.5ex] 4& 10& 5& 11& 12  \end{smallmatrix}\right)
&\quad
\left(\begin{smallmatrix} \times& 1& 2& 3& 4 \\[0.5ex] 1& 5& 6& 7& 8 \\[0.5ex] 2& 7& 9& 10& 6 \\[0.5ex] 3& 8& 11& 12& 10 \\[0.5ex] 4& 12& 5& 9& 11  \end{smallmatrix}\right)
\end{alignat*}
\end{matcases}

\begin{matcases}
For the pairings $\pi$ of the followig pairing matrices,
GAP determined that they are nondegenerate and that in the corresponding ULIE groups $\Gamma_\pi$,
the element $a_1a_2^{-1}$ has order $3$:
\[
\left(\begin{smallmatrix} \times& 1& 2& 3& 4 \\[0.5ex] 1& 5& 6& 7& 8 \\[0.5ex] 2& 7& 9& 10& 6 \\[0.5ex] 3& 9& 11& 8& 12 \\[0.5ex] 4& 12& 5& 11& 10 \end{smallmatrix}\right)
\quad
\left(\begin{smallmatrix} \times& 1& 2& 3& 4 \\[0.5ex] 1& 5& 6& 7& 8 \\[0.5ex] 2& 7& 9& 10& 6 \\[0.5ex] 3& 11& 5& 12& 10 \\[0.5ex] 4& 9& 12& 8& 11 \end{smallmatrix}\right)
\quad
\left(\begin{smallmatrix} \times& 1& 2& 3& 4 \\[0.5ex] 1& 5& 6& 7& 8 \\[0.5ex] 2& 7& 9& 10& 11 \\[0.5ex] 3& 12& 11& 6& 5 \\[0.5ex] 4& 9& 8& 12& 10 \end{smallmatrix}\right)
\]
\end{matcases}

\begin{matcases}
For the pairings $\pi$ of the followig pairing matrices,
GAP determined that they are nondegenerate and that in the corresponding ULIE groups $\Gamma_\pi$,
the element $a_1a_2^{-1}$ has order $4$:
\begin{alignat*}{4}
\left(\begin{smallmatrix} \times& 1& 2& 3& 4 \\[0.5ex] 1& 5& 6& 7& 8 \\[0.5ex] 2& 6& 9& 8& 10 \\[0.5ex] 3& 10& 11& 5& 12 \\[0.5ex] 4& 11& 7& 12& 9 \end{smallmatrix}\right)
&\quad
\left(\begin{smallmatrix} \times& 1& 2& 3& 4 \\[0.5ex] 1& 5& 6& 7& 8 \\[0.5ex] 2& 6& 9& 8& 10 \\[0.5ex] 3& 10& 11& 9& 12 \\[0.5ex] 4& 11& 7& 12& 5 \end{smallmatrix}\right)
&\quad
\left(\begin{smallmatrix} \times& 1& 2& 3& 4 \\[0.5ex] 1& 5& 6& 7& 8 \\[0.5ex] 2& 6& 9& 8& 10 \\[0.5ex] 3& 11& 7& 12& 9 \\[0.5ex] 4& 10& 11& 5& 12 \end{smallmatrix}\right)
&\quad
\left(\begin{smallmatrix} \times& 1& 2& 3& 4 \\[0.5ex] 1& 5& 6& 7& 8 \\[0.5ex] 2& 6& 9& 8& 10 \\[0.5ex] 3& 11& 10& 12& 5 \\[0.5ex] 4& 7& 11& 9& 12 \end{smallmatrix}\right)
&\quad
\left(\begin{smallmatrix} \times& 1& 2& 3& 4 \\[0.5ex] 1& 5& 6& 7& 8 \\[0.5ex] 2& 6& 9& 10& 11 \\[0.5ex] 3& 8& 7& 12& 9 \\[0.5ex] 4& 11& 10& 5& 12 \end{smallmatrix}\right)   \displaybreak[2] \\[2ex]
\left(\begin{smallmatrix} \times& 1& 2& 3& 4 \\[0.5ex] 1& 5& 6& 7& 8 \\[0.5ex] 2& 7& 5& 8& 6 \\[0.5ex] 3& 9& 10& 11& 12 \\[0.5ex] 4& 10& 12& 9& 11 \end{smallmatrix}\right)
&\quad
\left(\begin{smallmatrix} \times& 1& 2& 3& 4 \\[0.5ex] 1& 5& 6& 7& 8 \\[0.5ex] 2& 7& 5& 8& 9 \\[0.5ex] 3& 6& 10& 11& 12 \\[0.5ex] 4& 10& 12& 9& 11 \end{smallmatrix}\right)
&\quad
\left(\begin{smallmatrix} \times& 1& 2& 3& 4 \\[0.5ex] 1& 5& 6& 7& 8 \\[0.5ex] 2& 7& 5& 8& 9 \\[0.5ex] 3& 9& 10& 11& 12 \\[0.5ex] 4& 10& 12& 6& 11 \end{smallmatrix}\right)
&\quad
\left(\begin{smallmatrix} \times& 1& 2& 3& 4 \\[0.5ex] 1& 5& 6& 7& 8 \\[0.5ex] 2& 7& 5& 9& 6 \\[0.5ex] 3& 10& 8& 11& 12 \\[0.5ex] 4& 9& 12& 10& 11 \end{smallmatrix}\right)
&\quad
\left(\begin{smallmatrix} \times& 1& 2& 3& 4 \\[0.5ex] 1& 5& 6& 7& 8 \\[0.5ex] 2& 7& 5& 9& 10 \\[0.5ex] 3& 6& 8& 11& 12 \\[0.5ex] 4& 9& 12& 10& 11 \end{smallmatrix}\right)   \displaybreak[2] \\[2ex]
\left(\begin{smallmatrix} \times& 1& 2& 3& 4 \\[0.5ex] 1& 5& 6& 7& 8 \\[0.5ex] 2& 7& 5& 9& 10 \\[0.5ex] 3& 10& 8& 11& 12 \\[0.5ex] 4& 9& 12& 6& 11 \end{smallmatrix}\right)
&\quad
\left(\begin{smallmatrix} \times& 1& 2& 3& 4 \\[0.5ex] 1& 5& 6& 7& 8 \\[0.5ex] 2& 7& 9& 8& 6 \\[0.5ex] 3& 10& 11& 5& 12 \\[0.5ex] 4& 11& 12& 10& 9 \end{smallmatrix}\right)
&\quad
\left(\begin{smallmatrix} \times& 1& 2& 3& 4 \\[0.5ex] 1& 5& 6& 7& 8 \\[0.5ex] 2& 7& 9& 8& 6 \\[0.5ex] 3& 10& 11& 9& 12 \\[0.5ex] 4& 11& 12& 10& 5 \end{smallmatrix}\right)
&\quad
\left(\begin{smallmatrix} \times& 1& 2& 3& 4 \\[0.5ex] 1& 5& 6& 7& 8 \\[0.5ex] 2& 7& 9& 8& 10 \\[0.5ex] 3& 6& 11& 9& 12 \\[0.5ex] 4& 11& 12& 10& 5 \end{smallmatrix}\right)
&\quad
\left(\begin{smallmatrix} \times& 1& 2& 3& 4 \\[0.5ex] 1& 5& 6& 7& 8 \\[0.5ex] 2& 7& 9& 8& 10 \\[0.5ex] 3& 10& 11& 5& 12 \\[0.5ex] 4& 11& 12& 6& 9 \end{smallmatrix}\right)   \displaybreak[2] \\[2ex]
\left(\begin{smallmatrix} \times& 1& 2& 3& 4 \\[0.5ex] 1& 5& 6& 7& 8 \\[0.5ex] 2& 7& 9& 8& 10 \\[0.5ex] 3& 10& 11& 9& 12 \\[0.5ex] 4& 11& 12& 6& 5 \end{smallmatrix}\right)
&\quad
\left(\begin{smallmatrix} \times& 1& 2& 3& 4 \\[0.5ex] 1& 5& 6& 7& 8 \\[0.5ex] 2& 7& 9& 10& 6 \\[0.5ex] 3& 11& 5& 8& 12 \\[0.5ex] 4& 9& 12& 11& 10 \end{smallmatrix}\right)
&\quad
\left(\begin{smallmatrix} \times& 1& 2& 3& 4 \\[0.5ex] 1& 5& 6& 7& 8 \\[0.5ex] 2& 7& 9& 10& 6 \\[0.5ex] 3& 11& 8& 5& 12 \\[0.5ex] 4& 10& 12& 11& 9 \end{smallmatrix}\right)
&\quad
\left(\begin{smallmatrix} \times& 1& 2& 3& 4 \\[0.5ex] 1& 5& 6& 7& 8 \\[0.5ex] 2& 7& 9& 10& 6 \\[0.5ex] 3& 11& 8& 9& 12 \\[0.5ex] 4& 10& 12& 11& 5 \end{smallmatrix}\right)
&\quad
\left(\begin{smallmatrix} \times& 1& 2& 3& 4 \\[0.5ex] 1& 5& 6& 7& 8 \\[0.5ex] 2& 7& 9& 10& 11 \\[0.5ex] 3& 11& 5& 8& 12 \\[0.5ex] 4& 9& 12& 6& 10 \end{smallmatrix}\right) \\[2ex]
\left(\begin{smallmatrix} \times& 1& 2& 3& 4 \\[0.5ex] 1& 5& 6& 7& 8 \\[0.5ex] 2& 7& 9& 10& 11 \\[0.5ex] 3& 11& 8& 5& 12 \\[0.5ex] 4& 10& 12& 6& 9 \end{smallmatrix}\right)
&\quad
\left(\begin{smallmatrix} \times& 1& 2& 3& 4 \\[0.5ex] 1& 5& 6& 7& 8 \\[0.5ex] 2& 7& 9& 10& 11 \\[0.5ex] 3& 11& 8& 9& 12 \\[0.5ex] 4& 10& 12& 6& 5 \end{smallmatrix}\right)
&\quad
\left(\begin{smallmatrix} \times& 1& 2& 3& 4 \\[0.5ex] 1& 5& 6& 7& 8 \\[0.5ex] 2& 9& 5& 8& 10 \\[0.5ex] 3& 10& 11& 12& 6 \\[0.5ex] 4& 11& 7& 9& 12 \end{smallmatrix}\right)
\end{alignat*}
\end{matcases}

\begin{matcases}
For the pairings $\pi$ of the followig pairing matrices,
GAP determined that they are nondegenerate and that in the corresponding ULIE groups $\Gamma_\pi$,
the element $a_1a_2^{-1}$ has order $7$:
\[
\left(\begin{smallmatrix} \times& 1& 2& 3& 4 \\[0.5ex] 1& 5& 6& 7& 8 \\[0.5ex] 2& 7& 5& 9& 10 \\[0.5ex] 3& 6& 10& 11& 12 \\[0.5ex] 4& 12& 9& 8& 11 \end{smallmatrix}\right)
\quad
\left(\begin{smallmatrix} \times& 1& 2& 3& 4 \\[0.5ex] 1& 5& 6& 7& 8 \\[0.5ex] 2& 7& 9& 8& 10 \\[0.5ex] 3& 10& 11& 12& 6 \\[0.5ex] 4& 9& 12& 5& 11 \end{smallmatrix}\right)
\quad
\left(\begin{smallmatrix} \times& 1& 2& 3& 4 \\[0.5ex] 1& 5& 6& 7& 8 \\[0.5ex] 2& 7& 9& 10& 5 \\[0.5ex] 3& 11& 12& 8& 9 \\[0.5ex] 4& 10& 11& 6& 12 \end{smallmatrix}\right)
\]
\end{matcases}

\begin{matcases}
For the pairing $\pi$ corresponding to the pairing matrix
\[
\left(\begin{smallmatrix} \times& 1& 2& 3& 4 \\[0.5ex] 1& 5& 6& 7& 8 \\[0.5ex] 2& 6& 9& 8& 10 \\[0.5ex] 3& 11& 7& 12& 5 \\[0.5ex] 4& 10& 11& 9& 12 \end{smallmatrix}\right),
\]
our algorithm in GAP was not able to decide on degeneracy, but the relations imply that in the corresponding ULIE
group $\Gamma_\pi$, we have $(a_1a_3^{-1})^2=1$.
Indeed this folllows from only the relations implied by the elements of the matrix labelled with $5$--$8$, as we now show.
These relations are
\begin{align}
a_1^2&=a_3a_4 \label{eq:mat1rel5} \\
a_1a_2&=a_2a_1 \label{eq:mat1rel6} \\
a_1a_3&=a_3a_2 \label{eq:mat1rel7} \\
a_1a_4&=a_2a_3\,. \label{eq:mat1rel8}
\end{align}
Using~\eqref{eq:mat1rel5} to solve for $a_4$ and substituteing in~\eqref{eq:mat1rel8} yields
$a_1a_3^{-1}a_1^2=a_2a_3$ which with~\eqref{eq:mat1rel6} gives
\begin{equation}\label{eq:mat1*}
a_1a_2^{-1}a_3^{-1}a_1^2a_3^{-1}=1.
\end{equation}
From~\eqref{eq:mat1rel7} we get $a_2=a_3^{-1}a_1a_3$, and 
using $a_2^{-1}=a_3^{-1}a_1^{-1}a_3$ in~\eqref{eq:mat1*} yields
\[
a_1a_3^{-1}a_1a_3^{-1}=1,
\]
as desired.
\end{matcases}

\begin{matcases}
For the pairing corresponding to the pairing matrix
\[
\left(\begin{smallmatrix} \times& 1& 2& 3& 4 \\[0.5ex] 1& 5& 6& 7& 8 \\[0.5ex] 2& 6& 9& 10& 11 \\[0.5ex] 3& 10& 7& 12& 5 \\[0.5ex] 4& 11& 8& 9& 12 \end{smallmatrix}\right)
\]
our algorithm in GAP was not able to decide on degeneracy, but the relations imply that in the corresponding ULIE
group $\Gamma_\pi$, we have $(a_1a_4^{-1})^4=1$.
Indeed this folllows from only the relations implied by the elements of the matrix labelled with $5$--$8$ and $12$, as we now show.
These relations are
\begin{align}
a_1^2&=a_3a_4 \label{eq:mat2rel5} \\
a_1a_2&=a_2a_1 \label{eq:mat2rel6} \\
a_1a_3&=a_3a_2 \label{eq:mat2rel7} \\
a_1a_4&=a_4a_2\,. \label{eq:mat2rel8} \\
a_3^2&=a_4^2\,. \label{eq:mat2rel12}
\end{align}
Using~\eqref{eq:mat2rel5} to solve for $a_3$ and substituting into~\eqref{eq:mat2rel7} yields $a_2=a_4a_1a_4^{-1}$,
while from~\eqref{eq:mat2rel8} we get $a_2=a_4^{-1}a_1a_4$.
Thus, $a_4^2$ commutes with $a_1$ and, hence, also commutes with $a_2$.
Now
\[
(a_1a_4^{-1})^2=a_1a_4^{-1}a_1a_4^{-1}=a_1a_2a_4^{-2}
\]
and
\begin{equation}\label{eq:mat2*}
(a_1a_4^{-1})^4=a_1^2a_2^2a_4^{-4}.
\end{equation}
But using~\eqref{eq:mat2rel12} and again~\eqref{eq:mat2rel5} to solve for $a_3$, we get
\[
a_4^2=a_3^2=a_4^{-1}a_1^2a_4^{-1}a_1^2=(a_4^{-1}a_1a_4)^2a_4^{-1}a_1^2=a_2^2a_4^{-2}a_1^2,
\]
so $a_4^{-4}=a_1^{-2}a_2^{-2}$.
Using this in~\eqref{eq:mat2*} yields $(a_1a_4^{-1})^4=1$, as desired.
\end{matcases}

\begin{matcases}
For the pairing corresponding to the pairing matrix
\[
\left(\begin{smallmatrix} \times& 1& 2& 3& 4 \\[0.5ex] 1& 5& 6& 7& 8 \\[0.5ex] 2& 7& 8& 9& 10 \\[0.5ex] 3& 10& 5& 11& 12 \\[0.5ex] 4& 11& 12& 6& 9 \end{smallmatrix}\right),
\]
our algorithm in GAP was not able to decide on degeneracy, but the relations imply that in the corresponding ULIE
group $\Gamma_\pi$, we have $(a_1a_3^{-1})^4=1$.
Indeed, this follows from only the relations implied by the elements of the matrix labelled with $5$--$7$, $9$ and $11$, as we now show.
These relations are
\begin{align}
a_1^2&=a_3a_2 \label{eq:mat3rel5} \\
a_1a_2&=a_4a_3 \label{eq:mat3rel6} \\
a_1a_3&=a_2a_1 \label{eq:mat3rel7} \\
a_2a_3&=a_4^2. \label{eq:mat3rel9} \\
a_3^2&=a_4a_1. \label{eq:mat3rel11}
\end{align}
Using~\eqref{eq:mat3rel7} to solve for $a_2$ and ~\eqref{eq:mat3rel11} to solve for $a_4$ and substituting into the other relations,
we get, respectively
\begin{align}
a_1^3&=a_3a_1a_3 \label{eq:mat3rel5'} \\
a_1^2a_3a_1^{-1}&=a_3^2a_1^{-1}a_3 \label{eq:mat3rel6''} \\
a_1a_3a_1^{-1}a_3&=a_3^2a_1^{-1}a_3^2a_1^{-1}. \label{eq:mat3rel9''}
\end{align}
From~\eqref{eq:mat3rel6''} we get
$a_1^2a_3^2=a_3^2a_1^{-1}a_3a_1a_3$ so using~\eqref{eq:mat3rel5'} yields
\begin{equation}\label{eq:mat3rel13}
a_1^2a_3^2=a_3^2a_1^2.
\end{equation}
From~\eqref{eq:mat3rel9''} we get
$(a_3a_1a_3)a_1^{-1}(a_3a_1a_3)=a_3^3a_1^{-1}a_3^3$ and using~\eqref{eq:mat3rel5'} gives
\begin{equation}\label{eq:mat3rel14}
a_1^5=a_3^3a_1^{-1}a_3^3.
\end{equation}
In fact, one can show that $\Gamma_\pi$ is the finitely presented group with generators $a_1$ and $a_3$
and relations~\eqref{eq:mat3rel5'}, \eqref{eq:mat3rel13} and~\eqref{eq:mat3rel14}, but we will not bother with this.
Now using~\eqref{eq:mat3rel5'}, \eqref{eq:mat3rel13} and~\eqref{eq:mat3rel14}, we obtain
\begin{multline}\label{eq:mat3rel15}
a_1a_3^{-1}a_1a_3^{-1}=a_1a_3^{-1}(a_3^{-1}a_1^3a_3^{-1})a_3^{-1}
=a_1a_3^{-2}a_1^3a_3^{-2}=a_1^{-1}a_3^{-2}a_1^5a_3^{-2} \\
=a_1^{-1}a_3(a_3^{-3}a_1^5a_3^{-3})a_3=a_1^{-1}a_3a_1^{-1}a_3.
\end{multline}
Therefore, using~\eqref{eq:mat3rel15}, \eqref{eq:mat3rel5'} and~\eqref{eq:mat3rel14}, we get
\[
(a_1a_3^{-1})^4 
=a_1a_3^{-1}(a_1^{-1}a_3a_1^{-1}a_3)a_1a_3^{-1}
=a_1(a_3^{-1}a_1^{-1}a_3^{-1})a_3^2a_1^{-1}(a_3a_1a_3)a_3^{-2}
=a_1^{-2}a_3^2a_1^2a_3^{-2}
\]
and the last word is the identity element by~\eqref{eq:mat3rel13}.
\end{matcases}


\subsection{Case 5 by 5, $a\ne b$.}{$\;$}\newline

We now list all of the non--abelian infinite groups that appeared in the $5\times 5$ case
from pairing matrices that are not of the form~\eqref{eq:mat1234}
and we also present partial information about the other groups (abelian and/or finite) that appeared.

\subsubsection{Non--Amenable groups.}{$\;$}\newline

\begin{gp}
$\langle x,y:x^4=1,\,x^2y=yx^2\rangle\cong\bZ_4*_{\bZ_2}(\bZ\times\bZ_2),$
the amalgamated free product of $\bZ_4\cong\langle x:x^4=1\rangle$
and $\bZ\times\bZ_2\cong\langle y,z:z^2=1,\,yz=zy\rangle$
over $\bZ_2$ by the identification $x^2=z$,
 from the pairing matrix
\[
\left(\begin{smallmatrix}
\times& 1& 2& 3& 4 \\[0.5ex]
1& 5& 3& 2& 6 \\[0.5ex]
6& 4& 7& 8& 5  \\[0.5ex]
9& 10& 11& 12& 7 \\[0.5ex]
10& 9& 12& 11& 8
\end{smallmatrix}\right)
\]
\end{gp}

\begin{gp}
$\langle x,y,z:x^2=z^2=1,\,xy=yx,\,xz=zx\rangle\cong\bZ_2\times(\bZ*\bZ_2)$,
 from the pairing matrix
\[
\left(\begin{smallmatrix} \times& 1& 2& 3& 4 \\[0.5ex] 1& 5& 3& 2& 6 \\[0.5ex] 4& 6& 7& 8& 5 \\[0.5ex] 
  9& 10& 11& 12& 7 \\[0,5ex] 10& 9& 12& 11& 8 \end{smallmatrix}\right)
\]
\end{gp}

\subsubsection{Infinite amenable non--abelian groups}{$\;$}\newline

\begin{gp}
The infinite dihedral group
\[
\Dih_\infty=\bZ\rtimes_\alpha\bZ_2=\langle x,y:y^2=1,\,yxy=x^{-1}\rangle\cong\bZ_2*\bZ_2,
\]
where $\alpha(1)=-1$,
from the following pairing matrices:
\begin{alignat*}{4}
\left(\begin{smallmatrix} \times& 1& 2& 3& 4 \\[0.5ex] 1& 5& 3& 2& 6 \\[0.5ex] 7& 6& 4& 8& 5 \\[0.5ex] 
  9& 10& 11& 12& 7 \\[0.5ex] 10& 9& 12& 11& 8 \end{smallmatrix}\right)
&\quad
\left(\begin{smallmatrix} \times& 1& 2& 3& 4 \\[0.5ex] 1& 5& 3& 2& 6 \\[0.5ex] 7& 6& 4& 8& 9 \\[0.5ex] 
  10& 11& 5& 12& 7 \\[0.5ex] 11& 10& 12& 9& 8 \end{smallmatrix}\right)
&\quad
\left(\begin{smallmatrix} \times& 1& 2& 3& 4 \\[0.5ex] 1& 5& 3& 2& 6 \\[0.5ex] 7& 6& 4& 8& 9 \\[0.5ex] 
  10& 11& 9& 12& 7 \\[0.5ex] 11& 10& 12& 5& 8 \end{smallmatrix}\right)
&\quad
\left(\begin{smallmatrix} \times& 1& 2& 3& 4 \\[0.5ex] 1& 5& 3& 2& 6 \\[0.5ex] 4& 7& 6& 8& 5 \\[0.5ex] 
  9& 10& 11& 12& 7 \\[0.5ex] 10& 9& 12& 11& 8 \end{smallmatrix}\right)
&\quad
\left(\begin{smallmatrix} \times& 1& 2& 3& 4 \\[0.5ex] 1& 5& 3& 2& 6 \\[0.5ex] 4& 7& 6& 8& 9 \\[0.5ex] 
  10& 11& 5& 12& 7 \\[0.5ex] 11& 10& 12& 9& 8 \end{smallmatrix}\right) \displaybreak[2] \\[2ex]
\left(\begin{smallmatrix} \times& 1& 2& 3& 4 \\[0.5ex] 1& 5& 3& 2& 6 \\[0.5ex] 4& 7& 6& 8& 9 \\[0.5ex] 
  10& 11& 9& 12& 7 \\[0.5ex] 11& 10& 12& 5& 8 \end{smallmatrix}\right)
&\quad
\left(\begin{smallmatrix} \times& 1& 2& 3& 4 \\[0.5ex] 1& 5& 3& 6& 7 \\[0.5ex] 2& 8& 9& 10& 11 \\[0.5ex] 
  4& 11& 7& 12& 5 \\[0.5ex] 8& 6& 10& 9& 12 \end{smallmatrix}\right)
&\quad
\left(\begin{smallmatrix} \times& 1& 2& 3& 4 \\[0.5ex] 1& 5& 3& 6& 7 \\[0.5ex] 2& 8& 9& 10& 11 \\[0.5ex] 
  4& 11& 7& 12& 9 \\[0.5ex] 8& 6& 10& 5& 12 \end{smallmatrix}\right)
&\quad
\left(\begin{smallmatrix} \times& 1& 2& 3& 4 \\[0.5ex] 1& 5& 3& 6& 7 \\[0.5ex] 2& 8& 9& 10& 11 \\[0.5ex] 
  8& 6& 10& 5& 12 \\[0.5ex] 12& 7& 11& 4& 9 \end{smallmatrix}\right)
&\quad
\left(\begin{smallmatrix} \times& 1& 2& 3& 4 \\[0.5ex] 1& 5& 3& 6& 7 \\[0.5ex] 2& 8& 9& 10& 11 \\[0.5ex] 
  8& 6& 10& 9& 12 \\[0.5ex] 12& 7& 11& 4& 5 \end{smallmatrix}\right) \displaybreak[2] \\[2ex]
\left(\begin{smallmatrix} \times& 1& 2& 3& 4 \\[0.5ex] 1& 5& 3& 6& 7 \\[0.5ex] 4& 8& 7& 9& 5 \\[0.5ex] 
  6& 10& 11& 12& 8 \\[0.5ex] 10& 2& 12& 11& 9 \end{smallmatrix}\right)
&\quad
\left(\begin{smallmatrix} \times& 1& 2& 3& 4 \\[0.5ex] 1& 5& 3& 6& 7 \\[0.5ex] 4& 8& 7& 9& 10 \\[0.5ex] 
  6& 11& 5& 12& 8 \\[0.5ex] 11& 2& 12& 10& 9 \end{smallmatrix}\right)
&\quad
\left(\begin{smallmatrix} \times& 1& 2& 3& 4 \\[0.5ex] 1& 5& 3& 6& 7 \\[0.5ex] 4& 8& 7& 9& 10 \\[0.5ex] 
  6& 11& 10& 12& 8 \\[0.5ex] 11& 2& 12& 5& 9 \end{smallmatrix}\right)
&\quad
\left(\begin{smallmatrix} \times& 1& 2& 3& 4 \\[0.5ex] 1& 5& 3& 6& 7 \\[0.5ex] 6& 8& 5& 9& 10 \\[0.5ex] 
  8& 2& 9& 11& 12 \\[0.5ex] 12& 7& 10& 4& 11 \end{smallmatrix}\right)
&\quad
\left(\begin{smallmatrix} \times& 1& 2& 3& 4 \\[0.5ex] 1& 5& 3& 6& 7 \\[0.5ex] 6& 8& 9& 10& 11 \\[0.5ex] 
  8& 2& 10& 5& 12 \\[0.5ex] 12& 7& 11& 4& 9 \end{smallmatrix}\right) \displaybreak[2] \\[2ex]
\left(\begin{smallmatrix} \times& 1& 2& 3& 4 \\[0.5ex] 1& 5& 3& 6& 7 \\[0.5ex] 6& 8& 9& 10& 11 \\[0.5ex] 
  8& 2& 10& 9& 12 \\[0.5ex] 12& 7& 11& 4& 5 \end{smallmatrix}\right)
\end{alignat*}
\end{gp}

\begin{gp}
$\bZ_2\times\Dih_\infty$
from the pairing matrices
$\left(\begin{smallmatrix} \times& 1& 2& 3& 4 \\[0.5ex] 1& 5& 6& 7& 8 \\[0.5ex] 9& 6& 5& 4& 3 \\[0.5ex] 
10& 7& 11& 12& 2 \\[0.5ex] 11& 8& 10& 9& 12 \end{smallmatrix}\right)$
and
$\left(\begin{smallmatrix} \times& 1& 2& 3& 4 \\[0.5ex] 1& 5& 6& 7& 8 \\[0.5ex] 9& 6& 5& 4& 10 \\[0.5ex] 
10& 7& 11& 12& 9 \\[0.5ex] 11& 8& 3& 2& 12 \end{smallmatrix}\right)$.
\end{gp}

\begin{gp}
$\bZ\times\Dih_\infty$
from the pairing matrix
$\left(\begin{smallmatrix} \times& 1& 2& 3& 4 \\[0.5ex] 1& 5& 6& 7& 8 \\[0.5ex] 2& 6& 9& 10& 11 \\[0.5ex] 
  7& 3& 10& 5& 12 \\[0.5ex] 8& 4& 11& 12& 9 \end{smallmatrix}\right)$
\end{gp}

\begin{gp}
The semidirect product $\bZ_3\rtimes_\alpha\bZ=\langle x,y:y^3=1,\,xyx^{-1}=y^2\rangle$,
where $\alpha(1)=2$,
 from the following pairing matrices:
\begin{alignat*}{4}
\left(\begin{smallmatrix} \times& 1& 2& 3& 4 \\[0.5ex] 1& 2& 5& 4& 6 \\[0.5ex] 7& 5& 8& 6& 3 \\[0.5ex] 
  9& 10& 11& 8& 12 \\[0.5ex] 10& 11& 9& 12& 7 \end{smallmatrix}\right)
&\quad
\left(\begin{smallmatrix} \times& 1& 2& 3& 4 \\[0.5ex] 1& 2& 5& 4& 6 \\[0.5ex] 7& 8& 9& 5& 10 \\[0.5ex] 
  9& 7& 8& 11& 12 \\[0.5ex] 10& 12& 11& 6& 3 \end{smallmatrix}\right)
&\quad
\left(\begin{smallmatrix} \times& 1& 2& 3& 4 \\[0.5ex] 1& 2& 5& 4& 6 \\[0.5ex] 7& 8& 9& 6& 3 \\[0.5ex] 
  10& 11& 12& 8& 7 \\[0.5ex] 12& 10& 11& 9& 5 \end{smallmatrix}\right)
&\quad
\left(\begin{smallmatrix} \times& 1& 2& 3& 4 \\[0.5ex] 1& 5& 3& 6& 7 \\[0.5ex] 4& 7& 6& 2& 8 \\[0.5ex] 
  9& 10& 8& 11& 12 \\[0.5ex] 10& 12& 11& 5& 9 \end{smallmatrix}\right)
&\quad
\left(\begin{smallmatrix} \times& 1& 2& 3& 4 \\[0.5ex] 1& 5& 3& 6& 7 \\[0.5ex] 4& 8& 6& 2& 9 \\[0.5ex] 
  10& 11& 7& 5& 12 \\[0.5ex] 12& 10& 9& 8& 11 \end{smallmatrix}\right) \displaybreak[2] \\[2ex]
\left(\begin{smallmatrix} \times& 1& 2& 3& 4 \\[0.5ex] 1& 5& 3& 6& 7 \\[0.5ex] 4& 8& 6& 2& 9 \\[0.5ex] 
  10& 11& 8& 5& 12 \\[0.5ex] 12& 10& 9& 7& 11 \end{smallmatrix}\right)
&\quad
\left(\begin{smallmatrix} \times& 1& 2& 3& 4 \\[0.5ex] 1& 5& 3& 6& 7 \\[0.5ex] 5& 7& 6& 2& 8 \\[0.5ex] 
  9& 10& 8& 11& 12 \\[0.5ex] 10& 12& 11& 4& 9 \end{smallmatrix}\right)
&\quad
\left(\begin{smallmatrix} \times& 1& 2& 3& 4 \\[0.5ex] 1& 5& 3& 6& 7 \\[0.5ex] 5& 8& 6& 2& 9 \\[0.5ex] 
  10& 11& 7& 4& 12 \\[0.5ex] 12& 10& 9& 8& 11 \end{smallmatrix}\right)
&\quad
\left(\begin{smallmatrix} \times& 1& 2& 3& 4 \\[0.5ex] 1& 5& 3& 6& 7 \\[0.5ex] 5& 8& 6& 2& 9 \\[0.5ex] 
  10& 11& 8& 4& 12 \\[0.5ex] 12& 10& 9& 7& 11 \end{smallmatrix}\right)
&\quad
\left(\begin{smallmatrix} \times& 1& 2& 3& 4 \\[0.5ex] 5& 2& 3& 1& 6 \\[0.5ex] 7& 4& 5& 6& 8 \\[0.5ex] 
  8& 9& 10& 11& 12 \\[0.5ex] 12& 11& 9& 10& 7 \end{smallmatrix}\right) \displaybreak[2] \\[2ex]
\left(\begin{smallmatrix} \times& 1& 2& 3& 4 \\[0.5ex] 5& 2& 3& 1& 6 \\[0.5ex] 7& 4& 5& 8& 9 \\[0.5ex] 
  9& 6& 10& 11& 12 \\[0.5ex] 12& 11& 8& 10& 7 \end{smallmatrix}\right)
&\quad
\left(\begin{smallmatrix} \times& 1& 2& 3& 4 \\[0.5ex] 5& 2& 3& 1& 6 \\[0.5ex] 7& 4& 5& 8& 9 \\[0.5ex] 
  9& 8& 10& 11& 12 \\[0.5ex] 12& 11& 6& 10& 7 \end{smallmatrix}\right)
&\quad
\left(\begin{smallmatrix} \times& 1& 2& 3& 4 \\[0.5ex] 5& 2& 3& 1& 6 \\[0.5ex] 7& 4& 8& 6& 9 \\[0.5ex] 
  9& 10& 11& 5& 12 \\[0.5ex] 12& 8& 10& 11& 7 \end{smallmatrix}\right)
&\quad
\left(\begin{smallmatrix} \times& 1& 2& 3& 4 \\[0.5ex] 5& 2& 3& 1& 6 \\[0.5ex] 7& 4& 8& 6& 9 \\[0.5ex] 
  9& 10& 11& 8& 12 \\[0.5ex] 12& 5& 10& 11& 7 \end{smallmatrix}\right)
&\quad
\left(\begin{smallmatrix} \times& 1& 2& 3& 4 \\[0.5ex] 5& 2& 3& 1& 6 \\[0.5ex] 7& 4& 8& 9& 10 \\[0.5ex] 
  10& 6& 11& 5& 12 \\[0.5ex] 12& 8& 9& 11& 7 \end{smallmatrix}\right) \displaybreak[2] \\[2ex]
\left(\begin{smallmatrix} \times& 1& 2& 3& 4 \\[0.5ex] 5& 2& 3& 1& 6 \\[0.5ex] 7& 4& 8& 9& 10 \\[0.5ex] 
  10& 6& 11& 8& 12 \\[0.5ex] 12& 5& 9& 11& 7 \end{smallmatrix}\right)
&\quad
\left(\begin{smallmatrix} \times& 1& 2& 3& 4 \\[0.5ex] 5& 2& 3& 1& 6 \\[0.5ex] 7& 4& 8& 9& 10 \\[0.5ex] 
  10& 9& 11& 5& 12 \\[0.5ex] 12& 8& 6& 11& 7 \end{smallmatrix}\right)
&\quad
\left(\begin{smallmatrix} \times& 1& 2& 3& 4 \\[0.5ex] 5& 2& 3& 1& 6 \\[0.5ex] 7& 4& 8& 9& 10 \\[0.5ex] 
  10& 9& 11& 8& 12 \\[0.5ex] 12& 5& 6& 11& 7 \end{smallmatrix}\right)
&\quad
\left(\begin{smallmatrix} \times& 1& 2& 3& 4 \\[0.5ex] 5& 2& 6& 4& 7 \\[0.5ex] 8& 3& 7& 9& 10 \\[0.5ex] 
  9& 11& 12& 10& 8 \\[0.5ex] 12& 6& 1& 11& 5 \end{smallmatrix}\right)
&\quad
\left(\begin{smallmatrix} \times& 1& 2& 3& 4 \\[0.5ex] 5& 2& 6& 4& 7 \\[0.5ex] 8& 3& 9& 10& 11 \\[0.5ex] 
  10& 7& 12& 11& 8 \\[0.5ex] 12& 6& 1& 9& 5 \end{smallmatrix}\right) \displaybreak[2] \\[2ex]
\left(\begin{smallmatrix} \times& 1& 2& 3& 4 \\[0.5ex] 5& 2& 6& 4& 7 \\[0.5ex] 8& 3& 9& 10& 11 \\[0.5ex] 
  10& 9& 12& 11& 8 \\[0.5ex] 12& 6& 1& 7& 5 \end{smallmatrix}\right)
&\quad
\left(\begin{smallmatrix} \times& 1& 2& 3& 4 \\[0.5ex] 5& 2& 6& 4& 7 \\[0.5ex] 8& 5& 7& 9& 10 \\[0.5ex] 
  9& 11& 12& 10& 8 \\[0.5ex] 12& 6& 1& 11& 3 \end{smallmatrix}\right)
&\quad
\left(\begin{smallmatrix} \times& 1& 2& 3& 4 \\[0.5ex] 5& 2& 6& 4& 7 \\[0.5ex] 8& 5& 9& 10& 11 \\[0.5ex] 
  10& 7& 12& 11& 8 \\[0.5ex] 12& 6& 1& 9& 3 \end{smallmatrix}\right)
&\quad
\left(\begin{smallmatrix} \times& 1& 2& 3& 4 \\[0.5ex] 5& 2& 6& 4& 7 \\[0.5ex] 8& 5& 9& 10& 11 \\[0.5ex] 
  10& 9& 12& 11& 8 \\[0.5ex] 12& 6& 1& 7& 3 \end{smallmatrix}\right)
&\quad
\left(\begin{smallmatrix} \times& 1& 2& 3& 4 \\[0.5ex] 5& 2& 6& 7& 8 \\[0.5ex] 6& 7& 8& 4& 5 \\[0.5ex] 
  9& 3& 10& 11& 12 \\[0.5ex] 10& 11& 12& 1& 9 \end{smallmatrix}\right) \displaybreak[2] \\[2ex]
\left(\begin{smallmatrix} \times& 1& 2& 3& 4 \\[0.5ex] 5& 2& 6& 7& 8 \\[0.5ex] 6& 9& 7& 5& 3 \\[0.5ex] 
  10& 4& 11& 8& 12 \\[0.5ex] 11& 12& 9& 10& 1 \end{smallmatrix}\right)
&\quad
\left(\begin{smallmatrix} \times& 1& 2& 3& 4 \\[0.5ex] 5& 2& 6& 7& 8 \\[0.5ex] 6& 9& 7& 5& 3 \\[0.5ex] 
  10& 4& 11& 9& 12 \\[0.5ex] 11& 12& 8& 10& 1 \end{smallmatrix}\right)
\end{alignat*}
\end{gp}

\begin{gp}
The semidirect product
\[
\bZ\rtimes_\alpha(\bZ_2\times\bZ_2)=\langle x,y,z:x^2=y^2=1,\,xy=yx,\,xzx=z^{-1},\,yzy=z^{-1}\rangle,
\]
from the pairing matrix
$\left(\begin{smallmatrix} \times& 1& 2& 3& 4 \\[0.5ex] 1& 5& 3& 2& 6 \\[0.5ex] 4& 6& 7& 8& 9 \\[0.5ex] 
  10& 11& 5& 12& 7 \\[0.5ex] 11& 10& 12& 9& 8 \end{smallmatrix}\right)$
\end{gp}

\begin{gp}
The semidirect product
\[
\bZ\rtimes_\alpha\Dih_4=\langle x,y,z:x^2=y^4=1,\,xyx=y^{-1},\,xzx=z^{-1},\,yzy=z^{-1}\rangle,
\]
from the pairing matrix
$\left(\begin{smallmatrix} \times& 1& 2& 3& 4 \\[0.5ex] 1& 5& 6& 7& 8 \\[0.5ex] 9& 6& 5& 10& 11 \\[0.5ex] 
  10& 8& 3& 2& 12 \\[0.5ex] 11& 7& 4& 12& 9 \end{smallmatrix}\right)$
\end{gp}

\begin{gp}
The semidirect product
$\langle x,y:y^4=1,\,yxy^{-1}=x^{-1}\rangle\cong\bZ\rtimes_\alpha\bZ_4$, where $\alpha(1)=-1$,
from the pairing matrices:
\begin{alignat*}{4}
\left(\begin{smallmatrix} \times& 1& 2& 3& 4 \\[0.5ex] 1& 5& 3& 2& 6 \\[0.5ex] 6& 4& 7& 8& 9 \\[0.5ex] 
  10& 11& 5& 12& 8 \\[0.5ex] 11& 10& 12& 9& 7 \end{smallmatrix}\right)
&\quad
\left(\begin{smallmatrix} \times& 1& 2& 3& 4 \\[0.5ex] 1& 5& 3& 6& 7 \\[0.5ex] 2& 8& 5& 9& 10 \\[0.5ex] 
  8& 6& 9& 11& 12 \\[0.5ex] 12& 7& 10& 4& 11 \end{smallmatrix}\right)
&\quad
\left(\begin{smallmatrix} \times& 1& 2& 3& 4 \\[0.5ex] 1& 5& 3& 6& 7 \\[0.5ex] 2& 8& 9& 10& 11 \\[0.5ex] 
  7& 4& 11& 12& 5 \\[0.5ex] 8& 6& 10& 9& 12 \end{smallmatrix}\right)
&\quad
\left(\begin{smallmatrix} \times& 1& 2& 3& 4 \\[0.5ex] 5& 2& 1& 6& 7 \\[0.5ex] 6& 8& 9& 7& 3 \\[0.5ex] 
  10& 4& 11& 8& 12 \\[0.5ex] 12& 11& 5& 9& 10 \end{smallmatrix}\right)
&\quad
\left(\begin{smallmatrix} \times& 1& 2& 3& 4 \\[0.5ex] 5& 2& 1& 6& 7 \\[0.5ex] 6& 8& 9& 10& 3 \\[0.5ex] 
  11& 4& 7& 8& 12 \\[0.5ex] 12& 10& 5& 9& 11 \end{smallmatrix}\right) \displaybreak[2] \\[2ex]
\left(\begin{smallmatrix} \times& 1& 2& 3& 4 \\[0.5ex] 5& 2& 1& 6& 7 \\[0.5ex] 6& 8& 9& 10& 3 \\[0.5ex] 
  11& 4& 10& 8& 12 \\[0.5ex] 12& 7& 5& 9& 11 \end{smallmatrix}\right)
\end{alignat*}
\end{gp}

\begin{gp}
The semidirect product
$\langle x,y:x^4=1,\,yxy^{-1}=x^3\rangle\cong\bZ_4\rtimes_\alpha\bZ$
where $\alpha(1)=3$, from pairing matrices
\[
\left(\begin{smallmatrix} \times& 1& 2& 3& 4 \\[0.5ex] 1& 5& 3& 6& 7 \\[0.5ex] 2& 8& 9& 10& 11 \\[0.5ex] 
  8& 6& 10& 5& 12 \\[0.5ex] 11& 7& 4& 12& 9 \end{smallmatrix}\right)
\quad
\left(\begin{smallmatrix} \times& 1& 2& 3& 4 \\[0.5ex] 1& 5& 3& 6& 7 \\[0.5ex] 6& 8& 9& 10& 11 \\[0.5ex] 
  7& 4& 11& 12& 5 \\[0.5ex] 8& 2& 10& 9& 12 \end{smallmatrix}\right)
\]
\end{gp}

\begin{gp}
$\bZ_2\times(\bZ_4\rtimes_\alpha\bZ)=\langle x,y,z:x^2=1,\,xy=yx,\,xz=zx,\,z^4=1,\,yzy^{-1}=z^{-1}\rangle$,
where $\alpha(1)=3$,
from the pairing matrix
\[
\left(\begin{smallmatrix} \times& 1& 2& 3& 4 \\[0.5ex] 1& 5& 6& 7& 8 \\[0.5ex] 2& 6& 5& 9& 10 \\[0.5ex] 
  3& 7& 9& 11& 12 \\[0.5ex] 12& 8& 10& 4& 11 \end{smallmatrix}\right)
\]
\end{gp}

\begin{gp}
$\langle x,y,z:z^2=x^5=1,\,zx=xz,\,zy=yz,\,yxy^{-1}=x^2\rangle
\cong\bZ_2\times(\bZ_5\rtimes_\alpha\bZ)$, where $\alpha(1)=2$,
from pairing matrices
\[
\left(\begin{smallmatrix} \times& 1& 2& 3& 4 \\[0.5ex] 1& 5& 6& 7& 8 \\[0.5ex] 2& 6& 5& 9& 10 \\[0.5ex] 
  11& 7& 9& 4& 12 \\[0.5ex] 12& 8& 10& 11& 3 \end{smallmatrix}\right)
\quad
\left(\begin{smallmatrix} \times& 1& 2& 3& 4 \\[0.5ex] 1& 5& 6& 7& 8 \\[0.5ex] 2& 6& 5& 9& 10 \\[0.5ex] 
  11& 7& 9& 12& 3 \\[0.5ex] 12& 8& 10& 4& 11 \end{smallmatrix}\right)
\]
\end{gp}

\begin{gp}
$\langle x,y,z:z^2=x^5=1,\,zy=yz,\,yxy^{-1}=x^2,\,zxz=x^{-1}\rangle
\cong\bZ_5\rtimes_\alpha(\bZ\times\bZ_2)$, where $\alpha_{(1,0)}(1)=2$
and $\alpha_{(0,1)}(1)=4$,
from pairing matrix
\[
\left(\begin{smallmatrix} \times& 1& 2& 3& 4 \\[0.5ex] 1& 5& 6& 7& 8 \\[0.5ex] 2& 6& 5& 9& 10 \\[0.5ex] 
  11& 7& 10& 4& 12 \\[0.5ex] 12& 8& 9& 11& 3 \end{smallmatrix}\right)
\]
\end{gp}

\begin{gp}
The semidirect product
$(\bZ_2\times\bZ_2\times\bZ_2)\rtimes_\alpha\bZ$,
where $\alpha$ is the order--four automorphism of $\bZ_2\times\bZ_2\times\bZ_2$
determined by
\begin{align*}
\alpha:&(1,0,0)\mapsto(0,1,0) \\
&(0,1,0)\mapsto(0,0,1) \\
&(0,0,1)\mapsto(1,1,1),
\end{align*}
from pairing matrices
\[
\left(\begin{smallmatrix} \times& 1& 2& 3& 4 \\[0.5ex] 1& 5& 6& 7& 8 \\[0.5ex] 2& 6& 5& 9& 10 \\[0.5ex] 
  7& 4& 9& 11& 12 \\[0.5ex] 8& 3& 10& 12& 11 \end{smallmatrix}\right)
\quad
\left(\begin{smallmatrix} \times& 1& 2& 3& 4 \\[0.5ex] 1& 5& 6& 7& 8 \\[0.5ex] 2& 6& 5& 9& 10 \\[0.5ex] 
  7& 4& 10& 11& 12 \\[0.5ex] 8& 3& 9& 12& 11 \end{smallmatrix}\right)
\]
A quick analysis (using GAP) shows that these pairing matrices yield the group
\[
\Gamma=\langle x,y:y^2=1,\,(xyx^{-1}y)^2=1,\,yx^3y=xyxyx\rangle.
\]
The relations $y^2=1$ and $(xyx^{-1}y)^2=1$ imply that the elements $y$ and $xyx^{-1}$
both have squares equal to the identity and they commute with each other.
From $x^3=yxyxyxy$ we get
$yx^3yx^{-3}=(xyx^{-1})(x^2yx^{-2})$, and since the elements $xyx^{-1}$ and $x^2yx^{-2}$ both have
squares equal to the identity and commute with each other, the same applies to $y$ and $x^3yx^{-3}$.
Let $z_i=x^iyx^{-i}$ for $i=0,1,2,3$.
Then we have $z_iz_{i+1}=z_{i+1}z_i$, with $i+1$ taken modulo $4$, for all $i=0,1,2,3$.
We also have $z_0z_3=z_1z_2$, so $z_3=z_0z_1z_2=z_0z_2z_1$ and $z_3=z_3^{-1}=z_2z_1z_0=z_2z_0z_1$,
so $z_0z_2=z_2z_0$.
Thus, $z_0,z_1,z_2,z_3$ all commute and satisfy $z_3=z_0z_1z_2$.
We clearly have $xz_ix^{-1}=z_{i+1}$ for $i=0,1,2$, while
$xz_3x^{-1}=x(z_0z_1z_2)x^{-1}=z_1z_2z_3=z_0$.
Finally, the relation $x^3=yxyxyxy$ becomes $x^3=z_0z_1z_2z_3x^3$.
Therefore, we get that $\Gamma$ is isomorphic to the group
\begin{align*}
\langle x,z_0,z_1,z_2,z_3:z_i^2=1,\,z_0z_1z_2z_3=1,\,xz_ix^{-1}=z_{i+1}\rangle
\end{align*}
where the subscript $i+1$ is taken modulo $4$.
We easily recognize this as the semidirect product described above.
\end{gp}

\begin{gp}
$\bZ_2\times((\bZ_2\times\bZ_2)\rtimes_\alpha\bZ)$,
where $\alpha$ is the order--three
automorphism of $\bZ_2\times\bZ_2$ that permutes the nontrivial elements in a cycle,
from pairing matrix
\[
\left(\begin{smallmatrix} \times& 1& 2& 3& 4 \\[0.5ex] 1& 5& 6& 7& 8 \\[0.5ex] 2& 6& 5& 9& 10 \\[0.5ex] 
  7& 9& 3& 11& 12 \\[0.5ex] 8& 10& 4& 12& 11 \end{smallmatrix}\right)
\]
A quick analysis (using GAP) shows that this pairing matrix yields the group
\[
\Gamma=\bZ_2\times\langle x,y:\mid y^2=1,\,x^2=yxyxy\rangle.
\]
We have the relations
$xyx=yx^2y$, then $xyx^{-1}=yx^2yx^{-2}$.
So the two elements $y$ and $x^2yx^{=2}$ whose squares are the identity
commute, as their product $xyx^{-1}$ also has square equal to the identity.
But then also $y$ and $xyx^{-1}$ commute, and the three elements $(y,xyx^{-1},x^2yx^{-2})$
are the nontrivial elements of a copy of the group $\bZ_2\times\bZ_2$.
We also have
$x^4=(yxyxy)(yxyxy)=yxyx^2yxy=yx(xyx)xy=yx^2yx^2y$,
so $x^2yx^2=yx^4y$ and $x^2yx^{-2}=yx^4yx^{-4}$.
Since $x^2yx^{-2}=yxyx^{-1}$, we also get $xyx^{-1}=x^4yx^{-4}$, and, therefore, $x^3$ commutes with $y$.
If we write $z_1=y$, $z_2=xyx^{-1}$ and $z_3=x^2yx^{-2}$, then the defining relation $x^2=yxyxy$ becomes
$x^2=z_1z_2z_3x^2$, and we have
\[
\Gamma
\cong\bZ_2\times\langle x,z_1,z_2,z_3:z_i^2=1,\,xz_ix^{-1}=z_{i+1},\,z_1z_2=z_3\rangle,
\]
where the $i+1$ in the subscript is to be taken modulo $3$.
We recognize this as the $\bZ_2$ times the semidirect product described above.
\end{gp}

\begin{gp}
$\langle x,y:x^2y=yx^2,\,y^2=1\rangle\cong\bZ*_\bZ(\bZ\times\bZ_2)$,
the free product of $\bZ=\langle x\rangle$
and $\bZ\times\bZ_2=\langle y,z:yz=zy,\,y^2=1\rangle$
with amalgamation over $\bZ$ by the identification $x^2=z$,
 from the following pairing matrix
\[
\left(\begin{smallmatrix}
\times& 1& 2& 3& 4 \\[0.5ex]
1& 5& 3& 6& 7 \\[0.5ex]
2& 8& 5& 9& 10 \\[0.5ex]
4& 10& 7& 11& 12 \\[0.5ex]
8& 6& 9& 12& 11
\end{smallmatrix}\right)
\]
\end{gp}

\begin{gp}
$\langle x,y:x^4=y^2=1,yx^2=x^2y\rangle\cong(\bZ_2\times\bZ_2)*_{\bZ_2}\bZ_4$,
the free product of $\bZ_2\times\bZ_2=\langle y,z:y^2=z^2=1,\,yz=zy\rangle$
and $\bZ_4=\langle x:x^4=1\rangle$ with amalgamation over $\bZ_2$,
by the identification $x^2=z$,
from the pairing matrices
\begin{alignat*}{4}
\left(\begin{smallmatrix} \times& 1& 2& 3& 4 \\[0.5ex] 1& 5& 3& 2& 6 \\[0.5ex] 6& 4& 7& 8& 9 \\[0.5ex] 
  10& 11& 5& 12& 7 \\[0.5ex] 11& 10& 12& 9& 8 \end{smallmatrix}\right)
&\quad
\left(\begin{smallmatrix} \times& 1& 2& 3& 4 \\[0.5ex] 1& 5& 6& 7& 8 \\[0.5ex] 9& 6& 3& 10& 5 \\[0.5ex] 
  10& 8& 11& 9& 12 \\[0.5ex] 12& 7& 4& 11& 2 \end{smallmatrix}\right)
&\quad
\left(\begin{smallmatrix} \times& 1& 2& 3& 4 \\[0.5ex] 1& 5& 6& 7& 8 \\[0.5ex] 9& 6& 3& 10& 11 \\[0.5ex] 
  10& 8& 5& 9& 12 \\[0.5ex] 12& 7& 4& 11& 2 \end{smallmatrix}\right)
&\quad
\left(\begin{smallmatrix} \times& 1& 2& 3& 4 \\[0.5ex] 1& 5& 6& 7& 8 \\[0.5ex] 9& 6& 3& 10& 11 \\[0.5ex] 
  10& 8& 11& 9& 12 \\[0.5ex] 12& 7& 4& 5& 2 \end{smallmatrix}\right)
&\quad
\left(\begin{smallmatrix} \times& 1& 2& 3& 4 \\[0.5ex] 5& 2& 1& 6& 7 \\[0.5ex] 6& 8& 9& 7& 3 \\[0.5ex] 
  10& 4& 11& 9& 12 \\[0.5ex] 12& 11& 5& 8& 10 \end{smallmatrix}\right) \displaybreak[2] \\[2ex]
\left(\begin{smallmatrix} \times& 1& 2& 3& 4 \\[0.5ex] 5& 2& 1& 6& 7 \\[0.5ex] 6& 8& 9& 10& 3 \\[0.5ex] 
  11& 4& 7& 9& 12 \\[0.5ex] 12& 10& 5& 8& 11 \end{smallmatrix}\right)
&\quad
\left(\begin{smallmatrix} \times& 1& 2& 3& 4 \\[0.5ex] 5& 2& 1& 6& 7 \\[0.5ex] 6& 8& 9& 10& 3 \\[0.5ex] 
  11& 4& 10& 9& 12 \\[0.5ex] 12& 7& 5& 8& 11 \end{smallmatrix}\right)
\end{alignat*}
\end{gp}

\begin{gp}\label{gp:oldN=10}
$\Gamma=\langle x,y:y^2=1,\,x^2y=yx^2,\,(xy)^2=(yx)^2\rangle$
has center $Z(\Gamma)\cong\bZ\times\bZ_2$ generated by $\{x^2,(xy)^2x^{-2}\}$
with quotient $\Gamma/Z(\Gamma)\cong\bZ_2\times\bZ_2$,
from the pairing matrices
\begin{alignat*}{4}
\left(\begin{smallmatrix} \times& 1& 2& 3& 4 \\[0.5ex] 1& 2& 5& 6& 7 \\[0.5ex] 3& 6& 8& 9& 10 \\[0.5ex] 
  4& 10& 11& 7& 12 \\[0.5ex] 9& 5& 12& 8& 11 \end{smallmatrix}\right)
&\quad
\left(\begin{smallmatrix} \times& 1& 2& 3& 4 \\[0.5ex] 1& 2& 5& 6& 7 \\[0.5ex] 3& 8& 9& 4& 10 \\[0.5ex] 
  6& 10& 11& 7& 12 \\[0.5ex] 8& 9& 12& 5& 11 \end{smallmatrix}\right)
&\quad
\left(\begin{smallmatrix} \times& 1& 2& 3& 4 \\[0.5ex] 1& 5& 3& 6& 7 \\[0.5ex] 2& 4& 8& 9& 10 \\[0.5ex] 
  5& 9& 6& 11& 12 \\[0.5ex] 8& 10& 7& 12& 11 \end{smallmatrix}\right)
&\quad
\left(\begin{smallmatrix} \times& 1& 2& 3& 4 \\[0.5ex] 1& 5& 3& 6& 7 \\[0.5ex] 2& 8& 9& 10& 11 \\[0.5ex] 
  4& 7& 12& 11& 5 \\[0.5ex] 8& 6& 10& 9& 12 \end{smallmatrix}\right)
&\quad
\left(\begin{smallmatrix} \times& 1& 2& 3& 4 \\[0.5ex] 1& 5& 3& 6& 7 \\[0.5ex] 2& 8& 9& 10& 11 \\[0.5ex] 
  4& 12& 11& 7& 9 \\[0.5ex] 8& 6& 10& 5& 12 \end{smallmatrix}\right) \displaybreak[2] \\[2ex]
\left(\begin{smallmatrix} \times& 1& 2& 3& 4 \\[0.5ex] 1& 5& 3& 6& 7 \\[0.5ex] 4& 7& 8& 9& 5 \\[0.5ex] 
  6& 10& 11& 12& 9 \\[0.5ex] 10& 2& 12& 11& 8 \end{smallmatrix}\right)
\end{alignat*}
\end{gp}

\begin{gp}
$\langle x,y,z:y^2=z^2=1,\,x^2y=yx^2,\,(xy)^2=(yx)^2,\,xz=zx,\,yz=zy\rangle\cong\bZ_2\times\Gamma$,
where $\Gamma$ is the group from~\ref{gp:oldN=10},
from pairing matrices
\[
\left(\begin{smallmatrix} \times& 1& 2& 3& 4 \\[0.5ex] 1& 5& 6& 7& 8 \\[0.5ex] 2& 6& 5& 9& 10 \\[0.5ex] 
  7& 3& 9& 11& 12 \\[0.5ex] 8& 4& 10& 12& 11 \end{smallmatrix}\right)
\quad
\left(\begin{smallmatrix} \times& 1& 2& 3& 4 \\[0.5ex] 1& 5& 6& 7& 8 \\[0.5ex] 2& 6& 5& 9& 10 \\[0.5ex] 
  7& 3& 10& 11& 12 \\[0.5ex] 8& 4& 9& 12& 11 \end{smallmatrix}\right)
\]
\end{gp}

\begin{gp}
$\Gamma=\langle x,y,z:x^2=1,\,xy=yx,\,yz=zy,\,z^2=y^2,\,(xz)^2=(zx)^2\rangle$
has center $Z(\Gamma)\cong\bZ\times\bZ_2$ generated by $\{y,(xz)^2y^{-2}\}$
with quotient $\Gamma/Z(\Gamma)\cong\bZ_2\times\bZ_2$,
from the following pairing matrices:
\[
\left(\begin{smallmatrix} \times& 1& 2& 3& 4 \\[0.5ex] 1& 5& 6& 7& 8 \\[0.5ex] 2& 6& 5& 9& 10 \\[0.5ex] 
  3& 7& 9& 11& 12 \\[0.5ex] 8& 4& 10& 12& 11 \end{smallmatrix}\right)
\]
\end{gp}

\begin{gp}
$\Gamma=\langle x,y:y^2=1,\,yx^2y=x^{-2},\,(xy)^2=(yx)^2\rangle$
with center $Z(\Gamma)\cong\bZ\times\bZ_2$ generated by $\{x^4,(yx)^2\}$
and quotient $\Gamma/Z(\Gamma)\cong\Dih_4$,
from the pairing matrices:
\[
\left(\begin{smallmatrix} \times& 1& 2& 3& 4 \\[0.5ex] 1& 5& 3& 2& 6 \\[0.5ex] 4& 6& 7& 8& 9 \\[0.5ex] 
  10& 11& 5& 12& 8 \\[0.5ex] 11& 10& 12& 9& 7 \end{smallmatrix}\right)
\]
\end{gp}

\begin{gp}
$\Gamma=\langle x,y:y^4=1,\,xy^2=y^2x,\,x^2y=yx^2,\,(xy)^2=(yx)^2\rangle$
with center $Z(\Gamma)\cong\bZ\times\bZ_2\times\bZ_2$ generated by $\{x^2,y^2,(xy)^2x^{-2}\}$
and quotient $\Gamma/Z(\Gamma)\cong\bZ_2\times\bZ_2$,
from the pairing matrices:
\[
\left(\begin{smallmatrix} \times& 1& 2& 3& 4 \\[0.5ex] 1& 5& 6& 7& 8 \\[0.5ex] 2& 6& 5& 9& 10 \\[0.5ex] 
  3& 7& 9& 11& 12 \\[0.5ex] 12& 10& 8& 4& 11 \end{smallmatrix}\right)
\quad
\left(\begin{smallmatrix} \times& 1& 2& 3& 4 \\[0.5ex] 1& 5& 6& 7& 8 \\[0.5ex] 6& 2& 5& 9& 10 \\[0.5ex] 
  7& 3& 10& 11& 12 \\[0.5ex] 8& 4& 9& 12& 11 \end{smallmatrix}\right)
\]
\end{gp}

\begin{gp}
The group $\Gamma=\langle x,y:x^2=1,\,xy^3=y^3x,\,(xyxy^{-1})^2=1\rangle$
has center $Z(\Gamma)\cong\bZ\times\bZ_2$, generated by $\{y^3,xyxyxy^{-2}\}$,
and $\Gamma/Z(\Gamma)$ is isomorphic to the group $A_4$ of order $12$.
This group arises from the pairing matrices:
\[
\left(\begin{smallmatrix} \times& 1& 2& 3& 4 \\[0.5ex] 1& 5& 6& 7& 8 \\[0.5ex] 2& 6& 5& 9& 10 \\[0.5ex] 
  7& 9& 4& 11& 12 \\[0.5ex] 8& 10& 3& 12& 11 \end{smallmatrix}\right)
\quad
\left(\begin{smallmatrix} \times& 1& 2& 3& 4 \\[0.5ex] 1& 5& 6& 7& 8 \\[0.5ex] 2& 6& 5& 9& 10 \\[0.5ex] 
  7& 10& 3& 11& 12 \\[0.5ex] 8& 9& 4& 12& 11 \end{smallmatrix}\right)
\quad
\left(\begin{smallmatrix} \times& 1& 2& 3& 4 \\[0.5ex] 1& 5& 6& 7& 8 \\[0.5ex] 2& 6& 5& 9& 10 \\[0.5ex] 
  7& 10& 4& 11& 12 \\[0.5ex] 8& 9& 3& 12& 11 \end{smallmatrix}\right)
\]
\end{gp}

\subsubsection{Infinite abelian groups}{$\;$}\newline
There are in the $5\times 5$ case
$78$ inequivalent pairing matrices
not of the form~\eqref{eq:mat1234} whose corresponding groups are infinite
abelian groups.

\subsubsection{Finite groups}{$\;$}\newline
There are $2741$ inequivalent $5\times 5$ pairing matrices not of the form~\eqref{eq:mat1234}
that yield finite groups.
The following table summarizes these
according to the groups' orders and whether they are abelian or not.

\vskip2ex
\noindent
\begin{tabular}{|r|r|r|r|r|r|r|r|r|r|r|r|r|r|r|r|r|r|r|r|}
\hline
order: &           6 &   8 & 10 & 11 & 12 & 13 & 14 & 15 & 16 & 17 & 18 & 19 & 20 & total \\
\hline\hline
\# abelian: &   1606 & 274 & 54 & 21 & 98 &  7 & 12 &  3 & 16 &  7 &  2 &  2 &    & 2102 \\
\hline
\# nonabelian: & 558 &  84 &    &    & 46 &    &  4 &    & 43 &    &  2 &    &  2 &  739 \\
\hline 
\end{tabular}

\vskip2ex
In fact, here is the list of all nonabelian finite groups obtained, and their frequencies.
The GAP code is the group's identifier in GAP's small groups library, which is a pair of
numbers, the first of which is the order of the group.

\vskip2ex
\noindent
\begin{tabular}{|l|l|r||l|l|r|}
\hline
GAP code & Description of Group & freq. & GAP code & Description of Group & freq. \\[0.5ex]
\hline
\quad[6,1] & symmetric group $S_3$ & $558$ & \quad[16,6] & $\bZ_8\rtimes\bZ_2$ & $4$ \\[0.5ex]
\hline
\quad[8,3] & dihedral group $\Dih_4$ & $56$ & \quad[16,7] & dihedral group $\Dih_8$ & $3$ \\[0.5ex]
\hline
\quad[8,4] & quaternion group $Q_8$ & $28$ & \quad[16,8] & quasidihedral group & $9$ \\[0.5ex]
\hline
\quad[12,1] & $\bZ_3\rtimes\bZ_4$ & $26$ & \quad[16,9] & quaternion group $Q_{16}$ & $4$ \\[0.5ex]
\hline
\quad[12,3] & alternating group $A_4$ & $12$ & \quad[16,11] & $\bZ_2\times\Dih_4$ & $2$ \\[0.5ex]
\hline
\quad[12,4] & dihedral group $\Dih_6$ & $8$ & \quad[16,12] & $\bZ_2\times Q_8$ & $1$ \\[0.5ex]
\hline
\quad[14,1] & dihedral group $\Dih_7$ & $4$ & \quad[16,13] & $(\bZ_4\times\bZ_2)\rtimes\bZ_2$ & $7$ \\[0.5ex]
\hline
\quad[16,3] & $(\bZ_4\times\bZ_2)\rtimes\bZ_2$ & $8$ & \quad[18,3] & $\bZ_3\times S_3$ & $2$ \\[0.5ex]
\hline
\quad[16,4] & $\bZ_4\rtimes\bZ_4$ & $5$ & \quad[20,3] & $\bZ_5\rtimes\bZ_4$ & $2$ \\[0.5ex]
\hline
\end{tabular}

\end{document}